\documentclass[a4paper,11pt]{article}

\usepackage{amsfonts,amsthm,amssymb,amsmath}
\usepackage[title,titletoc,toc]{appendix}
\usepackage[utf8]{inputenc}
\usepackage[affil-it]{authblk}
\usepackage{xcolor}
\usepackage{tikz}
\usepackage{enumerate}

\usepackage{authblk}

\newcommand\COMP{\hbox{C\kern -.58em {\raise .54ex \hbox{$\scriptscriptstyle |$}}
\kern-.55em {\raise .53ex \hbox{$\scriptscriptstyle |$}} }}
\newcommand\NN{\hbox{I\kern-.2em\hbox{N}}}
\newcommand\RR{\hbox{I\kern-.2em\hbox{R}}}
\newcommand\sRR{{\it \hbox{I\kern-.2em\hbox{R}}}}
\newcommand\QQ{\hbox{I\kern-.53em\hbox{Q}}}
\newcommand\PP{\hbox{I\kern-.53em\hbox{P}}}
\newcommand\EE{\hbox{I\kern-.53em\hbox{E}}}
\newcommand\ZZ{{{\rm Z}\kern-.28em{\rm Z}}}
\newcommand\be{\begin{equation}}
\newcommand\ee{\end{equation}}
\newtheorem{theorem}{Theorem}[section]

\newtheorem{proposition}[theorem]{Proposition}
\newtheorem{remark}[theorem]{Remark}

\newtheorem{lemma}[theorem]{Lemma}

\newtheorem{definition}[theorem]{Definition}

\newcommand{\E}{\mathbb E}

\makeatletter
\newcommand*\bigcdot{\mathpalette\bigcdot@{.5}}
\newcommand*\bigcdot@[2]{\mathbin{\vcenter{\hbox{\scalebox{#2}{$\m@th#1\bullet$}}}}}
\makeatother
\newcommand{\is}{\bigcdot }

\def \Lbrack {[\![}
\def \Rbrack {]\!]}

\numberwithin{equation}{section}

\begin{document}

\title{Reflected backward stochastic differential equations under stopping with an arbitrary random time\thanks{Tahir Choulli is grateful to Polytechnique for the hospitality, where this work started in March 2018, and he is grateful to Nizar Touzi for introducing him to the RBSDEs and proposing him this project and its main ideas.}}


\author[1]{Safa' Alsheyab}
\author[1]{Tahir Choulli}
\affil[1]{\small{Mathematical and Statistical Sciences Dept., University of  Alberta, Edmonton, AB, Canada, tchoulli@ualberta.ca, alsheyab@ualberta.ca}}

\renewcommand\Authands{ and }

\maketitle

\begin{abstract}
This paper addresses reflected backward stochastic differential equations (RBSDE hereafter) that take the form of 
\begin{eqnarray*}
\begin{cases}
dY_t=f(t,Y_t, Z_t)d(t\wedge\tau)+Z_tdW_t^{\tau}+dM_t-dK_t,\quad Y_{\tau}=\xi,\\
 Y\geq S\quad\mbox{on}\quad \Lbrack0,\tau\Lbrack,\quad \displaystyle\int_0^{\tau}(Y_{s-}-S_{s-})dK_s=0\quad P\mbox{-a.s..}\end{cases}
\end{eqnarray*}
Here $\tau$ is an arbitrary random time that might not be a stopping time for the filtration $\mathbb F$ generated by the Brownian motion $W$. We consider the filtration $\mathbb G$ resulting from the progressive enlargement of $\mathbb F$ with $\tau$ where this becomes a stopping time, and study the RBSDE under $\mathbb G$. Precisely, we focus on answering the following problems: a) What are the sufficient minimal conditions on  the data $(f, \xi, S, \tau)$ that guarantee the existence of the solution of the $\mathbb G$-RBSDE in $L^p$ ($p>1$)? b) How can we estimate the solution in norm using the triplet-data $(f, \xi, S)$? c) Is there an RBSDE under $\mathbb F$ that is intimately related to the current one and how their solutions are related to each other? This paper answers all these questions deeply and beyond. Importantly, we prove that for any random time, having a positive Az\'ema supermartingale, there exists a positive discount factor ${\widetilde{\cal E}}$ --a positive and non-increasing $\mathbb F$-adapted and RCLL process-- that is vital in answering our questions without assuming any further assumption on $\tau$, and determining the space for the triplet-data $(f,\xi, S)$ and the space for the solution of the RBSDE as well. Furthermore, we found that the conditions for the $\mathbb G$-RBSDE are weaker that the conditions for its $\mathbb F$-RBSDE counterpart when the horizon is unbounded. Our approach sounds novel and very robust, as it relies on sharp martingale inequalities that hold no matter what is the filtration, and it treats both the linear and general case of RBSDEs for bounded and unbounded horizon. \end{abstract}

\section{Introduction}
In this paper, we consider a complete probability space $\left(\Omega, {\cal F}, P\right)$, on which we suppose given a  standard Brownian motion $W$ (could be $d$-dimensional). Througout the paper, $\mathbb F:=({\cal F}_t)_{t\geq0}$ denotes the complete and right-continuous filtration generated by $W$. Besides this stochastic basis $(\Omega, {\cal F}, \mathbb F, P)$, we consider an arbitrary random time $\tau$ that might not be an $\mathbb F$-stopping time with values in $[0,+\infty)$, and the data-triplet $(f, S, \xi)$. Here $f(t,y,z)$ is a functional  that is random (the driver rate), $\xi$ is an ${\cal F}_{\tau}$-random variable\footnote{Here ${\cal{F}}_{\tau}$ is the $\sigma$-algebra that is generated by $\{X_{\tau}:\ X\quad\mbox{is}\quad {\mathbb F}\mbox{-optional}\}$}, and $S$ is a RCLL process $\mathbb F$-adapted with values in $[-\infty, +\infty)$. Thus, our main goal is to study the following RBSDE:
\begin{eqnarray}\label{RBSDE1}
\begin{cases}
dY_{t}=-f(t,Y_{t},Z_{t})d(t\wedge\tau)-d(K_{t\wedge\tau}+M_{t\wedge\tau})+Z_{t}dW_{t}^{\tau},\quad {Y}_{\tau}=\xi,\\
 Y\geq S\quad\mbox{on}\quad\Lbrack 0,\tau\Lbrack,\quad\mbox{and}\quad E\left[\displaystyle\int_{0}^{\tau}(Y_{t-}-S_{t-})dK_{t}\right]=0.
\end{cases}
\end{eqnarray}
This RBSDE generalizes the works of \cite{Touzi,Popier} to the case where $\tau$ is not a stopping time. \\
This study (for now) concentrate more on addressing the following points:
\begin{enumerate}
\item What are the conditions (the weakest possible)of the data-triplet $(f, S, \xi)$, without further assumption on $\tau$, that guarantee the existence and uniqueness of the solution to this RBSDE?
\item How (\ref{RBSDE1}) can be {\bf explicitly} connected to an RBSDE in $\mathbb F$? We want to explicitly determine the relationship between the two data triplets and between the solutions of the two RBSDEs.
\item How can we estimate --in norm-- the solution $(Y^{\mathbb G},Z^{\mathbb G}, M^{\mathbb{G}},K^{\mathbb{G}})$ in terms of the data-triplet? What are the adequate norms and adequate spaces for both the solution and the data-triplet?
\end{enumerate}  
\subsection{What the literature says about this RBSDE?} 
It is well known, see \cite{Touzi} for similar discussion, that a BSDE (Backward Stochastic Differential Equation) is an RBSDE with $S\equiv -\infty$ and $K\equiv 0$. Up to our knowledge, the BSDEs were introduced  in \cite{Bismut} with $f$ being linear in the variables $(y,z)$  and $\tau=T$ being a positive fixed constant. However, only after the seminal paper of Pardoux and Peng \cite{Pardoux4} that this class of BSDEs, with $\tau$ being positive fixed constant, got tremendous attention and has been investigated deeply and intensively in many directions. These studies were highly motivated by applications in probabilistic numerical methods and/or the probabilistic representation for  semilinear PDEs, stochastic control, stochastic game theory, theoretical economics and mathematical finance. The huge part of this literature focuses on weakening the Lipschitz property of the coefficient  $f$  with-respect-to the $y$-variable, allowing $\mathbb F$ be more general, and/or weakening the assumption on the barrier process $S$. Only very recently that the novel notion of second order BSDE was introduced in \cite{Cheridito1}, and extended in \cite{SonerRTouziZhang} afterwards, due to its vital role in treating the fully nonlinear PDEs. \\

The first BSDE (or RBSDE) with a random horizon appeared in \cite{Peng91}, where $\tau$ is an $\mathbb F$-stopping time. The author describes how the solution to the class of BSDEs with an unbounded random terminal time $\tau$, that is an $\mathbb F$-stopping time, is related to semilinear elliptic PDE. It is important to mention that  in the case of constant horizon $T$, the solution to the BSDEs are connected to viscosity solutions to a system of semilinear parabolic PDEs, see \cite{PardouxPradeillesRao}  and the references therein for details. Afterwards, this family of RBSDEs with  have been extended in various directions in \cite{BriandConfortola, BriandHu,Darling,Popier,Royer}, and the references therein to cite a few. For the case second order BSDE under random terminal time, that is an $\mathbb F$-stopping time, we refer the reader to the very recent work  \cite{Touzi}. \\

Herein, we address (\ref{RBSDE1}) by letting $\tau$ to be an arbitrary random time and address the main problems aforementioned.  This case is a natural extension of the exiting literature on RDBSEs with random terminal time, and is highly motivated by the two areas of credit risk theory  and life insurance (life market). For the credit risk framework, $\tau$ represent the default time of a firm, while in life insurance it models the death time of an insured, where the mortality  and longevity risks are real challenges for both academia and insurance industry.  Up to our knowledge, all the existing literature treating this class of RBDEs assumes very strong assumption(s) on $\tau$. The most frequent among these,  see \cite{Kharroubi} and the references therein, we cite the case where $W^{\tau}$ should remain a martingale under the enlarged filtration (this case is also known in the literature as the immersion assumption).  \\

\subsection{ Main challenges and  our achievements}
On one hand, the RBSDE (\ref{RBSDE1}) is a {\it natural} extension of  the existing literature on RBSDEs involved with random horizon, see \cite{Touzi} and the references therein to cite a few, to the case where $\tau$ is an arbitrary random time. On the other hand, our hedging and pricing studies in \cite{Choulli5} for some class of informational markets yield to these form of RBSDEs and BSDEs, where the main source of uncertainty is $W^{\tau}$ and the driver $f$ is Lipschitz . \\

The difficulties for addressing (\ref{RBSDE1}) are numerous and challenging. Among these, on the one hand, we mention that $W^{\tau}$ is {\it not} a $\mathbb G$-martingale when $\tau$ is general. This explains why all the literature about BSDE under random horizon, up to our knowledge, assume the immersion assumption on $\tau$, which says that any $\mathbb F$-martingale stopped at $\tau$ remains a $\mathbb G$-martingale. On the other hand, the Burkholder-Davis-Gundy inequalities for martingales, that are really vital in BSDEs and RBSDEs, fail for martingales stopped at $\tau$ that might not be a pseudo-stooping time with respect to $\mathbb F$. In fact, we refer the reader to \cite{Nik2008, Nik2005} for this fact and for the  notion of pseudo-stopping times that is very close to that of immersion. In virtue of the Doob-Meyer decomposition for $W^{\tau}$ under $\mathbb G$, one can think of using the transformation ${\cal T}(W)=W^{\tau}-G_{-}^{-1}I_{\Rbrack0,\tau\Rbrack}\is \langle W, m\rangle $ which is a $\mathbb G$-local martingale. However, this will definitely alters the driver $f(t,y,z)$ of the RBSDE. Precisely,  the process $ZG_{-}^{-1}\is \langle W, m\rangle=:\int_0^{\tau\wedge\cdot}\beta^{(m)}_sZ_sG_{s-}^{-1} ds$ will be transferred to the driver, and this will perturb the Lipschitz conditions and other  features, as the process $\beta^{(m)}$ might not be ``regular" nor ``smooth" enough. Hence, this view does not solve the problem, it makes it very complicated and will lead to assumption on $\tau$. Inspired by \cite{Touzi}  and \cite{Bouchard}, we address (\ref{RBSDE1}) in  two steps. In the first step, we consider the case of bounded horizon and we stop at $T\wedge \tau$ for some $T\in (0,+\infty)$ instead fo $\tau$. For this bounded random horizon case,  thanks to some results of \cite{Choulli1,Choulli2}, we answer fully and in details the main problems aforementioned and beyond. The second step consists of relaxing the boundedness condition on the random horizon  by letting somehow $T$ to go infinity. This rises additional serious challenges.\\

Our achievements are numerous at both methodical and conceptual aspects. In fact, besides answering all the aforementioned problems, we prove the following general fact: For any random time $\tau$ having positive Az\'ema supermartingale, there exists a positive and bounded decreasing process --that we call hereafter by discount factor and we denote by ${\widetilde{\cal E}}$-- that is crucial in defining the spaces and norms for both the solution of the RBSDE and the data-triplet. This discount factor is also vital in bridging the RBSDE (\ref{RBSDE1}) with its counter part RBSDE under $\mathbb F$, and cements their solutions as well in a very explicit manner. At the methodical aspect, we elaborate our prior estimates using different method than the existing ones in the literature. Indeed, we directly establish inequalities without distinguishing the cases on $p$, and this is due to some stronger and deeper martingales inequalities of \cite{Choulli4} that we slightly  generalize. Furthermore, our method is robust towards the nature of the filtration $\mathbb F$, and hence our analysis can be extended to setting with jumps without serious difficulties. Some  of these extensions can be found in \cite{Alsheyab}, while herein we restrict to the Brownian filtration $\mathbb F$ for the sake of keeping the setting accessible to a broad audience, and to avoid overshadowing the main ideas with technicalities related to the general setting.\\
 
 This paper has seven sections including the current one. The second section defines the mathematical model and its preliminaries such as the norms used for the RBSDEs and some vital results on enlargement of filtration $\mathbb F$ with $\tau$ and on martingales for the enlarged filtration. The third section addresses the optimal stopping problem and the Snell  envelop under stopping with $\tau$. This is vital as we know the Snell envelop is intimately related to linear RBSDE. The fourth and fifth sections are devoted to the linear RBSDEs depending whether we stop the RBSDE at $\tau\wedge{T}$ for some fixed planning horizon $T\in (0,+\infty)$, or we stop at $\tau$. The sixth and seventh sections deal with the general RBSDE  (\ref{RBSDE1}), and here again we distinguish the cases when we stop at   $\tau\wedge{T}$ or $\tau$. The paper has Appendixes where we recall some crucial results and/or prove our technical lemmas.
 
\section{The mathematical setting and notation}

This section defines the notations, the financial and the mathematical concepts that the paper addresses or uses, the mathematical model that we focus on,  and some useful existing results. 
Throughout the paper, we consider the complete probability space  $\left(\Omega, {\cal F}, P\right)$. By  ${\mathbb H}$ we denote an arbitrary  filtration that satisfies the usual conditions of completeness and right continuity.  For any process $X$, the $\mathbb H$-optional projection and  the $\mathbb H$-predictable projection, when they exist, will be denoted by $^{o,\mathbb H}X$  and $^{p,\mathbb H}X$ respectively. The set ${\cal M}(\mathbb H, Q)$ (respectively  ${\cal M}^{p}(\mathbb H, Q)$ for $p\in (1,+\infty)$) denotes the set of all $\mathbb H$-martingales (respectively $p$-integrable martingales) under $Q$, while ${\cal A}(\mathbb H, Q)$ denotes the set of all $\mathbb H$-optional processes that are right-continuous with left-limits (RCLL for short) with integrable variation under $Q$. When $Q=P$, we simply omit the probability for the sake of simple notations.  For an $\mathbb H$-semimartingale $X$, by $L(X,\mathbb H)$ we denote the set of $\mathbb H$-predictable processes that are $X$-integrable in the semimartingale sense.  For $\varphi\in L(X,\mathbb H)$, the resulting integral of $\varphi$ with respect to $X$ is denoted by $\varphi\is X$. For $\mathbb H$-local martingale $M$, we denote by $L^1_{loc}(M,\mathbb H)$ the set of $\mathbb H$-predictable processes $\varphi$ that are $X$-integrable and the resulting integral $\varphi\is M$ is an $\mathbb H$-local martingale. If ${\cal C}(\mathbb H)$ is a set of processes that are adapted to $\mathbb H$, then ${\cal C}_{loc}(\mathbb H)$ is the set of processes, $X$, for which there exists a sequence of $\mathbb H$-stopping times, $(T_n)_{n\geq 1}$, that increases to infinity and $X^{T_n}$ belongs to ${\cal C}(\mathbb H)$, for each $n\geq 1$.  The $\mathbb H$-dual optional projection and the $\mathbb H$-dual predictable projection of a process $V$ with finite variation, when they exist, will be denoted by  $V^{o,\mathbb H}$  and $V^{p,\mathbb H}$ respectively. For any real-valued $\mathbb H$-semimartingale, $L$, we denote by ${\cal E}(L)$ the Dol\'eans-Dade (stochastic) exponential. It is the unique solution to the stochastic differential equation $dX=X_{-}dL,\ X_0=1,$  and is given by
\begin{eqnarray}\label{DDequation}
 {\cal E}_t(L)=\exp\left(L_t-L_0-{1\over{2}}\langle L^c\rangle_t\right)\prod_{0<s\leq t}(1+\Delta L_s)e^{-\Delta L_s}.\end{eqnarray}
 
Throughout the paper, on $\left(\Omega, {\cal F}, P\right)$, we consider a standard Brownian motion $W=(W_t)_{t\geq 0}$, and its natural filtration $\mathbb F:=({\cal F}_t)_{t\geq 0}$ that satisfies the usual conditions of right continuity and completeness. On $\Omega\times [0,+\infty)$, we consider the $\mathbb F$-optional  $\sigma$-field  denoted by ${\cal O}(\mathbb F)$ and  the $\mathbb F$-progressive  $\sigma$-field denoted by $\mbox{Prog}(\mathbb F)$ (i.e., a process $X$ is said to be $\mathbb F$-progresssive if $X$, as a map on $\Omega\times [0,t]$, is ${\cal F}_t\otimes{\cal B}(\mathbb R)$-measurable, for any $t\in (0,+\infty)$, where ${\cal B}(\mathbb R)$ is the Borel $\sigma$-field on $\mathbb R$). 
\subsection{RBSDEs: Definition, spaces and norms}
Throughout this subsection we suppose given a complete filtered probability space $\left(\Omega, {\cal F}, \mathbb H=({\cal H}_t)_{t\geq 0}, Q\right)$, where $\mathbb H\supseteq{\mathbb F}$ and $Q$ is any probability measure absolutely continuous with respect to $P$. The following definition of RBSDEs is borrowed from \cite[Definition 2.1]{Briandetal}. 
\begin{definition}\label{Definition-RBSDE} Let $\sigma$ be an $\mathbb H$-stopping time, and $(f^{\mathbb H},S^{\mathbb H},\xi^{\mathbb H})$ be  a triplet such that $f^{\mathbb H}$ is $\mbox{Prog}(\mathbb H)\otimes{\cal B}(\mathbb R)\otimes{\cal B}(\mathbb R)$-measurable functional, $S^{\mathbb H}$ is a RCLL and $\mathbb H$-adapted process, and $\xi^{\mathbb H}$ is an ${\cal H}_{\sigma}$-measurable random variable. Then an $(\mathbb H, Q)$-solution to the following RBSDE
\begin{eqnarray}\label{RBSDE4definition}
\begin{cases}
dY_t=-f^{\mathbb H}(t,Y_t,Z_t)I_{\{t\leq\sigma\}}dt+Z_t dW_{t\wedge\sigma}-dM_t-dK_t,\quad Y_{\sigma}=\xi^{\mathbb H},\\
 \displaystyle{Y}\geq S^{\mathbb H}\ \mbox{on}\ \Lbrack0,\sigma\Lbrack,\quad \int_0^{\sigma}(Y_{u-}-S_{u-}^{\mathbb H})dK_u=0\quad P\mbox{-a.s.}.\end{cases}
\end{eqnarray}
is any quadruplet $(Y^{\mathbb H}, Z^{\mathbb H},M^{\mathbb H},K^{\mathbb H})$ satisfying (\ref{RBSDE4definition}) such that $M^{\mathbb H}\in {\cal M}_{0,loc}(Q,\mathbb H)$, $K^{\mathbb H}$ is a RCLL nondecreasing and $\mathbb H$-predictable, and
\begin{eqnarray}\label{Condition1}
\int_0^{\sigma}\left( (Z_t^{\mathbb H})^2+\vert{f^{\mathbb H}}(t,Y_t^{\mathbb H},Z_t^{\mathbb H})\vert\right) dt<+\infty\quad Q\mbox{-a.s.}
\end{eqnarray}
When $Q=P$ we will simply call the quadruplet an $\mathbb H$-solution, while the filtration is also omitted when there no risk of confusion. 
\end{definition}
In this paper, we are interested in solutions that are integrable somehow. To this end, we recall the following spaces and norms that will be used throughout the paper. We denote by $\mathbb{L}^{p}(Q)$ is the space of $\mathcal{F}$-measurable random variables $\xi'$, such that
\begin{equation*}
\parallel \xi' \parallel_{\mathbb{L}^{p}(Q)}^{p}:=E^{Q}\left[|\xi' | ^{p}\right ]<\infty .
\end{equation*}
$\mathbb{D}_{\sigma}(Q,p)$ is the space of  RCLL  and ${\cal F}\otimes{\cal B}(\mathbb R^+)$-measurable processes, $Y$, such that $Y=Y^{\sigma}$ and
\begin{equation*}
\Vert{Y }\Vert_{\mathbb{D}_{\sigma}(Q,p)}^{p}:=E^{Q}\left[\sup_{0\leq {t}\leq\sigma}\vert{Y_t}\vert^p\right ]<\infty.
\end{equation*}
Here ${\cal B}(\mathbb R^+)$ is the Borel $\sigma$-field of $\mathbb R^+$. $\mathbb{S}_{\sigma}(Q,p)$ is the space of $\mbox{Prog}(\mathbb H)$-measurable processes $Z$ such that $Z=Z^{\sigma}$ and 
\begin{equation*}
\Vert{Z}\Vert_{\mathbb{S}_{\sigma}(Q,p)}^{p}:=E^{Q}\left[\left (\int_{0}^{\sigma}\vert{ Z_t}\vert ^{2}dt\right )^{{p}/{2}}\right ]<\infty .
\end{equation*}
 For any $M\in {\cal M}_{loc}(Q,\mathbb H)$, we define its $p$-norm by 
 \begin{equation*}
\Vert {M} \Vert_{{\cal M}^p(Q)}^{p}:=E^{Q}\left[[M, M]_{\infty} ^{p/2}\right ]<\infty,\end{equation*}
and the $p$-norm of any $K\in {\cal A}_{loc}(Q,\mathbb H)$ is given by 
\begin{eqnarray*}
 \Vert{K}\Vert_{{\cal{A}}(Q,p)}^p:=E^Q\left[\left(\mbox{Var}_{\infty}(K)\right)^p\right].\end{eqnarray*}
 Herein and throughout the paper, Var$(K)$ denotes the total variation process of $K$, and ${\cal A}^p(Q,\mathbb H)$ is the set of $K\in {\cal A}_{loc}(Q,\mathbb H)$ such that $ \Vert{K}\Vert_{{\cal{A}}(Q,p)}<+\infty$.
\begin{definition}\label{RBSDE4Lp} Let $p\in (1,+\infty)$. An $L^p(Q, \mathbb H)$-solution for (\ref{RBSDE4definition}) is a $(Q,\mathbb H)$-solution  $(Y, Z,M,K)$  that belongs to 
$
\mathbb{D}_{\sigma}(Q,p)\otimes \mathbb{S}_{\sigma}(Q,p)\otimes{\cal M}^p(Q,\mathbb H)\otimes{\cal A}^p(Q,\mathbb H).$\end{definition}

\subsection{The random horizon and the progressive enlargement of $\mathbb F$}
In addition to this initial model $\left(\Omega, {\cal F}, \mathbb F,P\right)$, we consider an arbitrary random time, $\tau$, that might not be an $\mathbb F$-stopping time. This random time is parametrized though $\mathbb F$ by the pair $(G, \widetilde{G})$, called survival probabilities or Az\'ema supermartingales, and is given by
\begin{eqnarray}\label{GGtilde}
G_t :=^{o,\mathbb F}(I_{\Rbrack0,\tau\Rbrack})_t= P(\tau > t | {\cal F}_t) \ \mbox{ and } \ \widetilde{G}_t :=^{o,\mathbb F}(I_{\Rbrack0,\tau\Lbrack})_t= P(\tau \ge t | {\cal F}_t).\end{eqnarray}
Furthermore, the following process
\begin{equation} \label{processm}
m := G + D^{o,\mathbb F},
\end{equation}
is a BMO $\mathbb F$-martingale and play important role in the analysis of enlargement of filtration. The flow of information that incorporate both $\tau$ and $\mathbb{F}$ defined using the pair   $(D,\mathbb G)$ given by
\begin{equation}\label{processD}
D:=I_{\Rbrack\tau,+\infty\Rbrack},\ \mathbb G:=({\cal G}_t)_{t\geq 0},\ {\cal G}_t:={\cal G}^0_{t+}\ \mbox{with} \ {\cal G}_t^0:={\cal F}_t\vee\sigma\left(D_s,\ s\leq t\right).
\end{equation}
Thanks to  \cite[Theorem 3]{ACJ} and \cite[Theorem 2.3 and Theorem 2.11]{ChoulliDavelooseVanmaele}, we recall 
\begin{theorem}\label{Toperator} The following assertions hold.\\
{\rm{(a)}} For any $M\in{\cal M}_{loc}(\mathbb F)$,  the process
\begin{equation} \label{processMhat}
{\cal T}(M) := M^\tau -{\widetilde{G}}^{-1} I_{\Lbrack 0,\tau\Lbrack} \is [M,m] +  I_{\Lbrack 0,\tau\Lbrack} \is\Big(\sum \Delta M I_{\{\widetilde G=0<G_{-}\}}\Big)^{p,\mathbb F}\end{equation}
 is a $\mathbb G$-local martingale.\\
 {\rm{(b)}}  The process 
\begin{equation} \label{processNG}
N^{\mathbb G}:=D - \widetilde{G}^{-1} I_{\Lbrack 0,\tau\Lbrack} \is D^{o,\mathbb  F}
\end{equation}
is a $\mathbb G$-martingale with integrable variation. Moreover, $H\is N^{\mathbb G}$ is a $\mathbb G$-local martingale with locally integrable variation for any $H$ belonging to
\begin{equation} \label{SpaceLNG}
{\mathcal{I}}^o_{loc}(N^{\mathbb G},\mathbb G) := \Big\{K\in \mathcal{O}(\mathbb F)\ \ \big|\quad \vert{K}\vert G{\widetilde G}^{-1} I_{\{\widetilde{G}>0\}}\is D\in{\cal A}_{loc}(\mathbb G)\Big\}.
\end{equation}
\end{theorem}
 For any $q\in [1,+\infty)$ and a $\sigma$-algebra ${\cal H}$ on $\Omega\times [0,+\infty)$, we define
\begin{equation}\label{L1(PandD)Local}
L^q\left({\cal H}, P\otimes dD\right):=\left\{ X\ {\cal H}\mbox{-measurable}:\quad \E[\vert X_{\tau}\vert^q I_{\{\tau<+\infty\}}]<+\infty\right\}.\end{equation}

\begin{lemma} \label{G-projection}For any nonnegative or integrable process $X$, we always have 
\begin{equation}\label{converting}E\left[X_{t}|\mathcal{G}_{t}\right]I_{\{t\ <\tau\}}={E\left[X_{t}I_{\{t\ <\tau\}}|\mathcal{F}_{t}\right]}G_t^{-1}I_{\{t\ <\tau\}}.
\end{equation}
\end{lemma}
Throughout the paper, we assume the following assumption 
\begin{eqnarray}\label{Assumption4Tau}
 G>0\quad (\mbox{i.e., $G$ is a positive process) and}\quad 0<\tau<+\infty\quad P\mbox{-a.s.}.
\end{eqnarray}
Now, we recall \cite[Proposition 4.3]{Choulli1} that will be useful throughout the paper.
\begin{proposition}Suppose that $G>0$ and consider the process
 \begin{equation}\label{Ztilde}
\widetilde{Z}:=1/{\cal E}(G_{-}^{-1}\is m).
\end{equation}
Then the following assertions hold.\\
{\rm{(a)}} The process $\widetilde{Z}^{\tau}$ is a $\mathbb G$-martingale, and for any $T\in (0,+\infty)$, $\widetilde{Q}_T$ given by 
 \begin{equation}\label{Qtilde}
 \frac{d{\widetilde{Q}_T}}{dP}:=\widetilde{Z}_{T\wedge\tau}.
\end{equation}
is well defined probability measure on ${\cal G}_{\tau\wedge T}$.\\
{\rm{(a)}} For any  $M\in {\cal M}_{loc}(\mathbb F)$, we have $M^{T\wedge \tau}\in {\cal M}_{loc}(\mathbb G, \widetilde{Q})$.  In particular $W^{T\wedge\tau}$ is a Brownian motion for $(\widetilde{Q}, \mathbb G)$, for any $T\in (0,+\infty)$.
\end{proposition}
\begin{remark} In general, the $\mathbb G$-martingale $\widetilde{Z}^{\tau}$ might not be uniformly integrable, and hence in general $\widetilde{Q}$ might not be extended to $(0,+\infty]$. For these fact, we refer the reader to \cite[Proposition 4.3]{Choulli1} for details, where conditions for $\widetilde{Z}^{\tau}$  being uniformly integrable are fully singled out when $G>0$. 
\end{remark}
\section{The Snell envelop under random horizon}

Throughout the paper,  $\mathcal{J}_{\sigma_1}^{\sigma_2}(\mathbb{H})$ denotes the set of all $\mathbb{H}$-stopping times with values in $\Lbrack
\sigma_1,\sigma_2\Rbrack$, for any two $\mathbb H$-stopping times $\sigma_1$ and  $\sigma_2$  such that $\sigma_1\leq\sigma_2$ . 

\begin{proposition}\label{PropositionG2F}Suppose (\ref{Assumption4Tau}) holds, and let $X^{\mathbb G}$ be a $\mathbb G$-optional process such that $(X^{\mathbb G})^{\tau}=X^{\mathbb G}$. Then there exists a unique pair $(X^{\mathbb F}, k^{(pr)})$ of processes such hat $X^{\mathbb F}$ is $\mathbb F$-optional and $k^{(pr)}$ is $\mathbb F$-progressive and 
\begin{eqnarray}\label{EqualityG2F}
X^{\mathbb G}=X^{\mathbb F} I_{\Lbrack0,\tau\Lbrack}+k^{(pr)}\is D.
\end{eqnarray}
Furthermore,  the following assertions hold.\\
{\rm{(a)}} $X^{\mathbb G}$ is RCLL if and only if $X^{\mathbb F}$ is RCLL.\\
{\rm{(b)}} $X^{\mathbb G}$ is a $\mathbb G$-semimartingale if and only if $X^{\mathbb F}$ is an $\mathbb F$-semimartingale, and 
\begin{eqnarray}\label{Decompo4XG}
(X^{\mathbb G})^{\tau}=(X^{\mathbb F})^{\tau}+(k^{(pr)}-X^{\mathbb F})\is D.
\end{eqnarray}
{\rm{(c)}}$E\left[\sup_{t\geq 0} \vert X^{\mathbb G}_{t}\vert\right]<+\infty$ if and only if 
\begin{eqnarray}
k^{(pr)}\in  L^1\left(\widetilde\Omega, {\rm{Prog}}(\mathbb F), P\otimes D\right)\quad\mbox{and}\quad E\left[\int_0^{+\infty} \sup_{0\leq s< t}\vert X^{\mathbb F}_s\vert dD^{o,\mathbb F}_t\right]<+\infty.\end{eqnarray}
\end{proposition}
\begin{proof} Consider a $\mathbb G$-optional process $X^{\mathbb G}$. Then thanks to \cite[Lemma B.1 ]{Aksamit} (see also \cite[Lemma 4.4]{Jeulin1980}, there exists a pair $(X^{\mathbb F}, k^{(pr)})$ such that $X^{\mathbb F}$ is an $\mathbb F$-optional and $k^{(pr)}$ is $\mbox{Prog}(\mathbb F)$-measurable such that 
\begin{eqnarray*}
X^{\mathbb G}I_{\Lbrack0,\tau\Lbrack}=X^{\mathbb F}I_{\Lbrack0,\tau\Lbrack},\quad\mbox{and}\quad X^{\mathbb G}_{\tau}=k^{(pr)}_{\tau}.
\end{eqnarray*}
Furthermore, this pair is unique due to $G>0$. Thus, the condition  $X^{\mathbb G}=(X^{\mathbb G})^{\tau}$ allows us to derive 
\begin{eqnarray*}
X^{\mathbb G}=X^{\mathbb G}I_{\Lbrack0,\tau\Lbrack}+X^{\mathbb G}_{\tau} I_{\Lbrack\tau,+\infty\Lbrack}=X^{\mathbb F}I_{\Lbrack0,\tau\Lbrack}+k^{(pr)}\is D,
\end{eqnarray*}
and the equality (\ref{EqualityG2F}) is proved. \\
a) Thanks to (\ref{EqualityG2F}) and the fact that $k^{(pr)}\is D$ is a RCLL process, we deduce that $X^{\mathbb G}$ is a RCCL process if and only if $X^{\mathbb F} I_{\Lbrack0,\tau\Lbrack}$ is a RCLL process. Remark, due to $G>0$ and \cite{DellacherieMeyer80}, this latter fact is equivalent to $X^{\mathbb F}$ being RCLL. This ends the proof of assertion (a).\\
b) It is clear that $k^{(pr)}\is D$ is a RCLL $\mathbb G$-semimartingale, and hence $X^{\mathbb G}$ is a RCLL $\mathbb G$-semimartingale if and only if $X^{\mathbb F}I_{\Lbrack0,\tau\Lbrack}$ is a RCLL $\mathbb G$-semimartingale. By stopping, there is no loss of generality in assuming  $X^{\mathbb G}$ is bounded, which leads to the boundedness of $X^{\mathbb F}$, see  \cite[Lemma B.1]{Aksamit}  or  \cite[Lemma 4.4 (b), page 63]{Jeulin1980}. Thus, thanks to \cite[Th\'eor\`eme 47, p. 119 and Th\'eor\`eme 59, p. 268]{DellacherieMeyer80} that implies that the optional projection of a bounded RCLL $\mathbb G$-semimartingale is a RCLL $\mathbb F$-semimartingale, we deduce that $ X^{\mathbb F}G=^{o,\mathbb F}\left( X^{\mathbb F}I_{\Lbrack0,\tau\Lbrack}\right)$ is a RCLL $\mathbb F$-semimartingale. This with the condition $G>0$ and $G$ is an $\mathbb F$-semimartingale yield that $ X^{\mathbb F}$ is an $\mathbb F$-semimartingale. Furthermore, it is clear that when $X^{\mathbb F}$ is an $\mathbb F$-semimartingale, we have 
 \begin{eqnarray}\label{equalityG2Fbis}
 X^{\mathbb F}I_{\Lbrack0,\tau\Lbrack}=( X^{\mathbb F})^{\tau}- X^{\mathbb F}\is D\quad\mbox{is a $\mathbb G$-semimartingale},\end{eqnarray}
 and (\ref{Decompo4XG}) follows from this equality and (\ref{EqualityG2F}). \\
c)  Here, we prove assertion (c). To this end, we use  (\ref{EqualityG2F}) and notice that 
\begin{eqnarray*}
{{I}\over{2}} \leq \sup_{t\geq 0} \vert X^{\mathbb G}_{t}\vert=\max\left(\sup_{ 0\leq{t}<\tau} \vert X^{\mathbb F}_{t}\vert,\vert{k}^{(pr)}_{\tau}\vert\right)\leq I,\quad I:=\int_0^{\infty} \left(\sup_{ 0\leq{u}<t} \vert X^{\mathbb F}_{u}\vert+\vert{k}^{(pr)}_t\vert\right)dD_t.
\end{eqnarray*}
Hence, these inequalities imply that $E\left[\sup_{t\geq 0} \vert X^{\mathbb G}_{t}\vert\right]<+\infty$ iff $E\left[\int_0^{\infty} \vert{k}^{(pr)}_t\vert{d}D_t\right]<+\infty$ and 
\begin{eqnarray*}
E\left[\int_0^{\infty} \sup_{ 0\leq{u}<t} \vert X^{\mathbb F}_{u}\vert{d}D_t\right]=E\left[\int_0^{\infty} \sup_{ 0\leq{u}<t} \vert{X}^{\mathbb F}_{u}\vert{d}D_t^{o,\mathbb F}\right]<+\infty.
\end{eqnarray*}
This altter equality  is due to $\sup_{ 0\leq{u}<t} \vert{X}^{\mathbb F}_{u}$ being $\mathbb F$-optional. This ends the proof of the proposition. 
\end{proof}
\begin{lemma}\label{stoppingTimeLemma}
Let $\sigma_{1}$ and $\sigma_{2}$ be two $\mathbb{F}$-stopping times such that $\sigma_{1}\leq\sigma_{2}$ P-a.s.. Then, for any $\mathbb{G}$- stopping time, $\sigma^{\mathbb{G}}$, satisfying
\begin{equation}\label{sigmaG}
\sigma_{1}\wedge\tau\leq\sigma^{\mathbb{G}}\leq\sigma_{2}\wedge\tau \hspace{5mm} P\mbox{-a.s.},
\end{equation}
there exists an $\mathbb{F}$- stopping time $\sigma^{\mathbb{F}}$ such that 
\begin{equation}\label{sigmaF}
\sigma_{1}\leq\sigma^{\mathbb{F}}\leq\sigma_{2}\quad \mbox{ and }\quad \sigma^{\mathbb{F}}\wedge\tau=\sigma^{\mathbb{G}} \hspace{5mm} P\mbox{-a.s.}
\end{equation}
\end{lemma}
   The following is our main result of this section, where we write in different manners the Snell envelop of a process under $\mathbb G$ as a sum of a transformation of an $\mathbb F$-Snell envelop and $\mathbb G$-martingales.
\begin{theorem}\label{SnellEvelopG2F}Suppose $G>0$, and let $X^{\mathbb G}$ be a RCLL and $\mathbb G$-adapted process such that $(X^{\mathbb G})^{\tau}=X^{\mathbb G}$. Then consider the unique pair of processes $(X^{\mathbb F}, k^{(pr)})$ associated to $X^{\mathbb G}$, and $k^{(op)}$ the $\mathbb F$-optional projection of $k^{(pr)}$ with respect to the measure $P\otimes D$.   Then the following assertions hold.\\
{\rm{(a)}} If either $X^{\mathbb G}$ is nonnegative or $E\left[\sup_{t\geq 0} \max(X^{\mathbb G}_{t},0)\right]<+\infty$, then the $(\mathbb G,P)$-Snell envelop of $X^{\mathbb G}$, denoted ${\cal S}(X^{\mathbb G};\mathbb G,P)$, is given by
\begin{eqnarray}\label{Snell4(G,P)}
{\cal S}(X^{\mathbb G};\mathbb G,P)&&={{{\cal S}(X^{\mathbb F}G+k^{(op)}\is D^{o,\mathbb F}; \mathbb F, P)}\over{G}}I_{\Lbrack0,\tau\Lbrack}+(k^{(pr)}-k^{(op)})\is D+{{(k^{(op)}\is D^{o,\mathbb F})_{-}}\over{G_{-}^2}}\is {\cal T}(m)\nonumber\\
&&+\left(k^{(op)}+{{k^{(op)}\is D^{o,\mathbb F}}\over{G}}\right)\is N^{\mathbb G},
\end{eqnarray}
where ${\cal S}(X^{\mathbb F}G+k^{(op)}\is D^{o,\mathbb F}; \mathbb F, P)$ is the $(\mathbb F, P)$-Snell envelop of $X^{\mathbb F}G+k^{(op)}\is D^{o,\mathbb F}$.\\
{\rm{(b)}} Let T$\in (0,+\infty)$ and $\widetilde{Q}$ be given in (\ref{Qtilde}).  If either $X^{\mathbb G}\geq 0$ or $E^{\widetilde{Q}}\left[\sup_{T\geq t\geq 0} \max(X^{\mathbb G}_{t},0)\vert\right]<+\infty$, then the $(\mathbb G, \widetilde{Q})$-Snell envelop of $X^{\mathbb G}$, denoted ${\cal S}(X^{\mathbb G};\mathbb G,\widetilde{Q})$, is given on $[0,T]$ by
\begin{eqnarray}\label{Snell4(G,Qtilde)}
{\cal S}(X^{\mathbb G};\mathbb G,\widetilde{Q})={{{\cal S}(X^{\mathbb F}{\widetilde{\cal E}}-k^{(op)}\is{\widetilde{\cal E}}; \mathbb F, P)}\over{\widetilde{\cal E}}}I_{\Lbrack0,\tau\Lbrack}+(k^{(pr)}-k^{(op)})\is D^T+\left(k^{(op)}-{{k^{(op)}\is \widetilde{\cal E}}\over{\widetilde{\cal E}}}\right)\is (N^{\mathbb G})^T.
\end{eqnarray}
\end{theorem}
\begin{proof}
Let $\theta\in {\cal T}_{t\wedge\tau}^{\tau}(\mathbb G)$, then thanks to Lemma \ref{stoppingTimeLemma} there exists $\sigma\in {\cal T}_t^{\infty}(\mathbb F)$ such that  $\theta=\sigma\wedge\tau$.  Then notice that 
\begin{eqnarray}\label{equa100}
X_{\theta}^{\mathbb G}&&=X_{\sigma\wedge\tau}^{\mathbb G}I_{\{\sigma<\tau\}}+ h^{(pr)}_{\tau}I_{\{\sigma\geq\tau\}}=X_{\sigma}^{\mathbb F}I_{\{\sigma<\tau\}}+\int_0^{\sigma} h^{(pr)}_s dD_s \nonumber\\
&&=X_{\sigma}^{\mathbb F}I_{\{\sigma<\tau\}}+ ({{h^{(op)}}\over{\widetilde G}} \is D^{o,\mathbb F})_{\sigma\wedge\tau}+ (h^{(op)} \is N^{\mathbb G})_{\theta}+(h^{(pr)}- h^{(op)})\is D_{\theta} .
\end{eqnarray}
Furthermore, it is clear that both processes  $ h^{(op)} \is N^{\mathbb G}$ and $ (h^{(pr)}-h^{(op)})\is D$ are $\mathbb G$-martingale, and hence  by combining these remarks with Lemma \ref{G-projection}-(a) and taking conditional expectation with respect to ${\cal G}_t$ on both sides of the above equality, we derive 
\begin{eqnarray*}
&&Y_t(\theta)\\
&&:= E\left[ X_{\theta}^{\mathbb G}\big|{\cal G}_t\right]= E\left[ X_{\sigma}^{\mathbb F}I_{\{\sigma<\tau\}}+\int_0^{\sigma\wedge\tau} {{h^{(op)}_s}\over{{\widetilde G}_t}}  dD^{o,\mathbb F}_s \big|{\cal G}_t\right]+(h^{(op)} \is N^{\mathbb G})_t+(h^{(pr)}- h^{(op)})\is D_t \\
&&=E\left[ X_{\sigma}^{\mathbb F}I_{\{\sigma<\tau\}}+\int_{t\wedge\tau}^{\sigma\wedge\tau} {{h^{(op)}_s}\over{{\widetilde G}_t}} dD^{o,\mathbb F}_s \bigg|{\cal G}_t\right]+({{h^{(op)}_s}\over{{\widetilde G}_t}}\is D^{o,\mathbb F})_{t\wedge\tau}+(h^{(op)} \is N^{\mathbb G})_t+(h^{(pr)}- h^{(op)})\is D_t \\
&&=E\left[ X_{\sigma}^{\mathbb F}I_{\{\sigma<\tau\}}+\int_{t\wedge\tau}^{\sigma\wedge\tau} {{h^{(op)}_s}\over{{\widetilde G}_t}} dD^{o,\mathbb F}_s \bigg|{\cal F}_t\right]{{I_{\{\tau>t\}}}\over{G_t}}+({{h^{(op)}_s}\over{{\widetilde G}_t}}\is D^{o,\mathbb F})_{t\wedge\tau}+(h^{(op)} \is N^{\mathbb G})_t+(h^{(pr)}- h^{(op)})\is D_t \\
&&=E\left[ G_{\sigma}X_{\sigma}^{\mathbb F} +\int_t^{\sigma}h^{(op)}_s dD^{o,\mathbb F}_s \bigg|{\cal F}_t\right]{{I_{\{\tau>t\}}}\over{G_t}}+({{h^{(op)}_s}\over{{\widetilde G}_t}}\is D^{o,\mathbb F})_{t\wedge\tau}+(h^{(op)} \is N^{\mathbb G})_t+(h^{(pr)}- h^{(op)})\is D_t \\
&&=:{{X^{\mathbb F}_t(\sigma)}\over{G_t}} I_{\{t<\tau\}}-{{(h^{(op)}\is D^{o,\mathbb F})_t}\over{G}} I_{\{t<\tau\}}+({{h^{(op)}_s}\over{{\widetilde G}_t}}\is D^{o,\mathbb F})_{t\wedge\tau}+(h^{(op)} \is N^{\mathbb G})_t+(h^{(pr)}- h^{(op)})\is D_t 
\end{eqnarray*}
Thus, by taking the essential supremum over all $\theta\in {\cal T}_{t\wedge\tau,\tau}(\mathbb G)$, we deduce that 
\begin{eqnarray}\label{mainequality1}
{\cal S}(X^{\mathbb G};\mathbb G,P)&&={{{\cal S}(X^{\mathbb F}G+h^{(op)}\is D^{o,\mathbb F}; \mathbb F, P)}\over{G}}I_{\Lbrack0,\tau\Lbrack}-{{(h^{(op)}\is D^{o,\mathbb F})}\over{G}} I_{\Lbrack0,\tau\Lbrack}\nonumber\\
&&+({{h^{(op)}_s}\over{{\widetilde G}_t}}\is D^{o,\mathbb F})_{t\wedge\tau}+(h^{(op)} \is N^{\mathbb G})_t+(h^{(pr)}- h^{(op)})\is D_t .
\end{eqnarray}
Furthermore, put $V:=h^{(op)}\is D^{o,\mathbb F}$ and by remarking that $d(1/G^{\tau})=(G\widetilde{G})^{-1}I_{\Rbrack0,\tau\Rbrack} dD^{o,\mathbb F}-G_{-}^{-2}d{\cal T}(m)$ and using It\^o, we derive 
\begin{eqnarray*}
d\left({{V^{\tau}}\over{G^{\tau}}}\right)=V_{-}d(1/G^{\tau})+G^{-1}h^{(op)}I_{\Rbrack0,\tau\Rbrack} dD^{o,\mathbb F}={{V}\over{G{\widetilde{G}}}}I_{\Rbrack0,\tau\Rbrack} dD^{o,\mathbb F}-{{V_{-}}\over{G_{-}^2}}d {\cal T}(m)+  {{h^{(op)}}\over{\widetilde{G}}}I_{\Rbrack0,\tau\Rbrack} dD^{o,\mathbb F}.
\end{eqnarray*}
Thus, (\ref{Snell4(G,P)}) follows immediately from combining this equality with (\ref{mainequality1}) and the easy fact that 
\begin{eqnarray}\label{X-Fsemimartinagle}
XI_{\Lbrack0,\tau\Lbrack}=X^{\tau}-X\is D,\quad \mbox{for any}\quad \mathbb F\mbox{-semimartingale}\ X.\end{eqnarray}
This ends the proof of assertion (a).\\
2) Here, we fix $T\in (0,+\infty)$ and let $\theta\in {\cal T}_{t\wedge\tau}^{T\wedge\tau}(\mathbb G)$ and $\sigma\in  {\cal T}_t^T(\mathbb F)$ such that $\theta=\sigma\wedge\tau$. Then, similarly as in part 1),  by taking conditional expectation under $\widetilde{Q}$ and using the fact that the two processes  both processes  $ h^{(op)} \is N^{\mathbb G}$ and $ (h^{(pr)}-h^{(op)})\is D$ are remain  $\mathbb G$-martingale under $\widetilde{Q}$, we write  
\begin{eqnarray*}
 &&{\widetilde Y}_t(\theta):=E^{\widetilde Q} \left[ X_{\theta}^{\mathbb G}\big|{\cal G}_t\right]\\
 &&=E^{\widetilde Q}\left[ X_{\sigma}^{\mathbb F}I_{\{\sigma<\tau\}}+\int_t^{\sigma\wedge\tau} {{h^{(op)}_s}\over{{\widetilde G}_s}} dD^{o,\mathbb F}_s \big|{\cal G}_t\right]+({{h^{(op)}}\over{{\widetilde G}}}\is D^{o,\mathbb F})_{t\wedge\tau}+(h^{(op)} \is N^{\mathbb G})_t+(h^{(pr)}- h^{(op)})\is D_t \\
 &&=E\left[ {{\widetilde{Z}_{\sigma}}\over{\widetilde{Z}_t}}X_{\sigma}^{\mathbb F}I_{\{\sigma<\tau\}}+\int_t^{\sigma\wedge\tau} {{h^{(op)}_s\widetilde{Z}_s}\over{{\widetilde G}_s\widetilde{Z}_t}} dD^{o,\mathbb F}_s \big|{\cal G}_t\right]+({{h^{(op)}}\over{{\widetilde G}}}\is D^{o,\mathbb F})_{t\wedge\tau}+(h^{(op)} \is N^{\mathbb G})_t+(h^{(pr)}- h^{(op)})\is D_t \\
 &&=E\left[ \widetilde{Z}_{\sigma}X_{\sigma}^{\mathbb F}I_{\{\sigma<\tau\}}+\int_t^{\sigma\wedge\tau} {{h^{(op)}_s}\over{{\widetilde G}_s}}dV^{\mathbb F}_s \big|{\cal F}_t\right]{{I_{\{\tau>t\}}}\over{\widetilde{Z}_tG_t}}+({{h^{(op)}}\over{{\widetilde G}}}\is D^{o,\mathbb F})_{t\wedge\tau}+(h^{(op)} \is N^{\mathbb G})_t+(h^{(pr)}- h^{(op)})\is D_t \\
 &&=E\left[ \widetilde{\cal E}_{\sigma}X_{\sigma}^{\mathbb F}+\int_t^{\sigma} h^{(op)}_sdV^{\mathbb F}_s \big|{\cal F}_t\right]{{I_{\{\tau>t\}}}\over{\widetilde{\cal E}_t}}+({{h^{(op)}}\over{{\widetilde G}}}\is D^{o,\mathbb F})_{t\wedge\tau}+(h^{(op)} \is N^{\mathbb G})_t+(h^{(pr)}- h^{(op)})\is D_t \\
 &&=:{{X^{\mathbb F}_t(\sigma)}\over{\widetilde{\cal E}_t}}I_{\{\tau>t\}}-{{ (h^{(op)}\is V^{\mathbb F})_t}\over{\widetilde{\cal E}_t}}I_{\{\tau>t\}}+({{h^{(op)}}\over{{\widetilde G}}}\is D^{o,\mathbb F})_{t\wedge\tau}+(h^{(op)} \is N^{\mathbb G})_t+(h^{(pr)}- h^{(op)})\is D_t \\
\end{eqnarray*}
By taking essential supremum over all $\theta$, we get 
\begin{eqnarray}\label{mainequality2}
{\cal S}(X^{\mathbb G};\mathbb G,\widetilde{Q})&&={{{\cal S}(X^{\mathbb F}{\widetilde{\cal E}}+h^{(op)}\is D^{o,\mathbb F} ;\mathbb F, P)}\over{\widetilde{\cal E}}}I_{\Lbrack0,\tau\Lbrack}-{{ (h^{(op)}\is V^{\mathbb F})_t}\over{\widetilde{\cal E}_t}}I_{\{\tau>t\}}\nonumber\\
&&+({{h^{(op)}_s}\over{{\widetilde G}_t}}\is D^{o,\mathbb F})_{t\wedge\tau}+(h^{(op)} \is N^{\mathbb G})_t+(h^{(pr)}- h^{(op)})\is D_t .
\end{eqnarray}
Similar arguments, as in part 1) after equation (\ref{mainequality1}) applied to $V= h^{(op)}\is V^{\mathbb F}:= - h^{(op)}\is {\cal E}$, leads to 
$$-(h^{(op)}\is V^{\mathbb F}){\widetilde{\cal E}}^{-1}I_{\Lbrack0,\tau\Lbrack}+ (h^{(op)}_s{\widetilde G}^{-1}\is D^{o,\mathbb F})^{\tau}=(h^{(op)}\is V^{\mathbb F}){\widetilde{\cal E}}^{-1}\is N^{\mathbb G}.$$
Thus, (\ref{Snell4(G,Qtilde)}) follows from combining this fact with (\ref{mainequality2}), and the proof of assertion (b) is completed.  This ends the proof of theorem.\end{proof}
\section{The case of linear RBSDEs with bounded horizon}\label{LinearboundedSection}
In this section, we start by a given triplet $\left(f, S, \xi\right)$, called the data-triplet, where $f$ is an $\mathbb F$-progressively measurable process representing the driver of the BSDE, $S$ is a RCLL $\mathbb F$-adapted process that models the barrier of the RBSDE, and $\xi$ is ${\cal F}_{T\wedge\tau}$-measurable random variable which is the terminal condition such that $\xi\geq S_{\tau\wedge T}$. Therefore, in virtue of Proposition \ref{PropositionG2F}, there exists an $\mathbb{F}$-optional $h$ such that
\begin{equation}\label{sh}
 \xi=h_{T\wedge\tau},\quad P\mbox{-a.s..}
 \end{equation}
 Hence, the $\mathbb G$-triplet  data $\left(f, S, \xi\right)$ is equivalent to the $\mathbb F$-triplet  data $\left(f, S, h\right)$. In this section, our aim lies in addressing the following RBSDE under $\mathbb G$ given by 
 \begin{eqnarray}\label{RBSDEG}
\begin{cases}
dY=-f(t)d(t\wedge\tau)-d(K+M)+ZdW^{\tau},\quad Y_{\tau}=Y_{T}=\xi,\\
Y_{t}\geq S_{t};\quad 0\leq t<  T\wedge\tau,\quad \displaystyle\int_{0}^{T\wedge\tau}(Y_{t-}-S_{t-})dK_{t}=0,\quad P\mbox{-a.s..}
\end{cases}\end{eqnarray}
This section is divided into two subsection. The first subsection elaborates estimates inequalities for the solution of the RBSDE (when it exists), while the second subsection address the existence and uniqueness of the solution and the $\mathbb F$-RBSDE counterpart of (\ref{RBSDEG}).
\subsection{Various norm-estimates for the solutions} 
   This subsection elaborates estimates for the solution of the RBSDE  (\ref{RBSDEG}). To this end, we start elaborating some useful intermediate results that we summarize in two lemmas.
    \begin{lemma}\label{Lemma4.11} The following assertions hold.\\
      {\rm{(a)}} For any $T\in (0,+\infty)$, $m^{T\wedge\tau}$ is a BMO  $(\widetilde Q, \mathbb G)$-martingale. Furthermore, we have 
      \begin{eqnarray}\label{BMO(m)}
      E^{\widetilde Q}\left[[m,m]_{T\wedge\tau}-[m,m]_{t\wedge\tau-}\big|{\cal G}_t\right]\leq \Vert m\Vert_{BMO(P)},\quad P\mbox{-a.s.}.
      \end{eqnarray}
     {\rm{(b)}} For any $t\in (0,T]$, we have 
     \begin{eqnarray}\label{domination1}
        E^{\widetilde{Q}}\left[D^{o,\mathbb{F}}_{T\wedge\tau} -D^{o,\mathbb{F}}_{(t\wedge\tau)-}\big|{\cal G}_{t}\right]\leq G_{t-}I_{\{t<\tau\}}\leq 1,\quad P\mbox{-a.s.}.\end{eqnarray}
       {\rm{(c)}} For any $a\in(0,+\infty)$, it always holds that 
         \begin{eqnarray}\label{estimation1}
     \max(a,1) {\widetilde G}^{-1}\is D^{o,\mathbb F}-\widetilde {V}^{(a)}\quad \mbox{is nondecreasing and}\quad E\left[\int_{t\wedge\tau}^{T\wedge\tau}{\widetilde G}^{-1}_s dD^{o,\mathbb F}_s\big| {\cal G}_t\right]\leq 1 ,\quad P-\mbox{a.s.},\end{eqnarray}
      where  $\widetilde {V}^{(a)}$ is the process defined by 
      \begin{eqnarray}\label{Vepsilon}
     \widetilde {V}^{(a)}:={{a}\over{{\widetilde G}}}\is D^{o,\mathbb F}+\sum \left(-{{a\Delta D^{o,\mathbb F}}\over{\widetilde G}}+1-\left(1-{{\Delta D^{o,\mathbb F}}\over{\widetilde G}}\right)^a\right)
       \end{eqnarray}
     \end{lemma}
     The proof of this lemma is relegated to Appendix \ref{Appendix4Proofs}. The following lemma connects under some assumptions the solution to  (\ref{RBSDEG}) --when it exists-- to Snell envelop.
     
     \begin{lemma}\label{Solution2SnellEnvelop} Let  $p\in [1,+\infty)$, and suppose that the triplet $(f, S, \xi)$ satisfies 
\begin{eqnarray}\label{MainAssumption}
E^{\widetilde{Q}}\left[\vert\xi\vert^p+\left(\int_0^{T\wedge\tau}\vert f(s)\vert ds\right)^p+\sup_{0\leq u\leq\tau\wedge T}(S_u^+)^p\right]<+\infty.
\end{eqnarray}
   If $(Y^{\mathbb G}, Z^{\mathbb G}, M^{\mathbb G}, K^{\mathbb G})$ is a solution to (\ref{RBSDEG}), then   \begin{eqnarray}\label{RBSDE2Snell}
 Y^{\mathbb G}_t=\rm{ess}\sup_{\theta\in \mathcal{J}_{t\wedge\tau}^{T\wedge\tau}(\mathbb{G})}E^{\widetilde{Q}}\left[\int_{t\wedge\tau}^{\theta}f(s)ds + S_{\theta}1_{\{\theta <T\wedge\tau\}}+\xi 1_{\{\theta=T\wedge \tau\}}\ \Big|\ \mathcal{G}_{t}\right],\quad 0\leq t\leq T.
 \end{eqnarray}
     \end{lemma}
     The proof of this lemma is relegated to Appendix \ref{Appendix4Proofs}, while herein we elaborate our first estimate.
   \begin{theorem}\label{EstimatesUnderQtilde} For any $p\in (1,+\infty)$ there exists a positive constant $C$ that depends on $p$ only such that if ($Y^{\mathbb{G}},Z^{\mathbb{G}},K^{\mathbb{G}}, M^{\mathbb{G}}$) is a solution to (\ref{RBSDEG}), then 
\begin{align}\label{estimate100}
&E^{\widetilde{Q}}\left[\sup_{0\leq t\leq T\wedge\tau}\vert{Y}^{{\mathbb{G}}}_{t}\vert^p+\left(\int_{0}^{T\wedge\tau}\vert{Z}^{{\mathbb{G}}}_{s}\vert^{2}ds+[ M^{\mathbb{G}}, M^{\mathbb{G}}]_{T\wedge\tau}\right)^{p/2}+(K^{{\mathbb{G}}}_{T\wedge\tau})^p\right]\nonumber\\
&\leq{C} E^{\widetilde{Q}}\left[\vert\xi\vert^p+\left(\int_{0}^{T\wedge\tau}\vert{f}(s)\vert{d}s\right)^p+\sup_{0\leq t\leq T\wedge\tau}(S^{+}_{t})^p\right].
\end{align}
\end{theorem}
\begin{proof} This proof is divided into four parts, where we control and estimate, in a way or another, the four terms in the left-hand-side of (\ref{estimate100}).\\
{\bf Part 1.} Thanks to (\ref{RBSDE2Snell}), we conclude that $Y^{\mathbb G}$ satisfies 
    \begin{equation*}
  Y_{t}^{\mathbb{G}}=\underset{\upsilon\in \mathcal{J}_{t\wedge\tau,T\wedge\tau}(\mathbb{G})}{\rm{ess}\sup}\hspace{2mm}E^{\widetilde{Q}}\left[\int_{t\wedge\tau}^{\upsilon\wedge\tau}f(s)ds + S_{\upsilon}1_{\{\upsilon\ <T\wedge\tau\}}+\xi 1_{\{\upsilon =T\wedge \tau\}}\ \Big|\ \mathcal{G}_{t}\right].
  \end{equation*}
Therefore, by taking $\upsilon=T\wedge\tau\in \mathcal{J}_{t\wedge\tau}^{T\wedge\tau}(\mathbb G)$ and using $\xi\geq -\xi^-$ and $\int_{t\wedge\tau}^{\upsilon\wedge\tau}f(s)ds \geq -\int_{0}^{T\wedge\tau}(f(s))^-ds $, we deduce that 
\begin{eqnarray*}
 E^{\widetilde{Q}}\left[-\int_0^{T\wedge\tau} (f(s))^- ds-\xi^-\ \Big|\ \mathcal{G}_{t}\right] \leq Y_{t}^{\mathbb{G}}\leq E^{\widetilde{Q}}\left[\int_0^{T\wedge\tau} (f(s))^+ ds + \sup_{0\leq u\leq\tau\wedge T} S_u^+ +\xi^+ \ \Big|\ \mathcal{G}_{t}\right].
 \end{eqnarray*}
 This clearly leads to  
 \begin{eqnarray}\label{Domination4YG}
 \vert Y_{t}^{\mathbb{G}}\vert\leq  {\widetilde M}_t:=E^{\widetilde{Q}}\left[\int_0^{T\wedge\tau}\vert f(s)\vert ds + \sup_{0\leq u\leq\tau\wedge T} S_u^+ +\vert\xi\vert\ \Big|\ \mathcal{G}_{t}\right].
  \end{eqnarray}
Hence, by applying Doob's inequality to $\widetilde M$ under $(\widetilde Q, \mathbb G)$, denoting $C_{DB}$ the Doob's constant, and using $(\sum_{i=1}^n \vert x_i\vert)^p\leq n^{p-1}\sum_{i=1}^n\vert x_i\vert^p$, we derive 
\begin{align}\label{yesyes}
E^{\widetilde{Q}}\left[\sup_{0\leq t\leq T\wedge\tau}\vert Y^{{\mathbb{G}}}_t\vert^p\right]\leq 3^{p-1}C_{DB} E^{\widetilde{Q}}\left[\left(\int_{0}^{T\wedge\tau}\vert f(s)\vert ds\right)^p +\underset{0\leq s \leq T\wedge\tau}{\sup}(S^{+}_{s})^p+\vert\xi \vert^p\right].
\end{align}
{\bf Part 2.}  By combining $ K_{T\wedge\tau}^{\mathbb{G}}=Y^{\mathbb{G}}_{0}-\xi  +\int_{0}^{T\wedge\tau}f(t)dt- M^{\mathbb{G}}_{T\wedge\tau}+  \int_{0}^{T\wedge\tau}Z^{\mathbb{G}}_{s}dW_{t}^{\tau}$,  (\ref{yesyes}), $(\sum_{i=1}^n x_i )^p\leq n^{p-1}\sum_{i=1}^n x_i^p$, and the BDG inequalities for the $(\widetilde{Q}, \mathbb G)$-martingale $-M^{\mathbb{G}}+  Z^{\mathbb{G}}\is{W}^{\tau}$, we get
       \begin{align}\label{Control4KG}
  &E^{\widetilde{Q}} \left[(K_{T\wedge\tau}^{\mathbb{G}})^p\right]\nonumber\\
  &\leq 4^{p-1}E^{\widetilde{Q}}\left[\vert{Y}^{\mathbb{G}}_{0}\vert^p+\vert\xi\vert^p +  \left(\int_{0}^{T\wedge\tau}\vert f(t)\vert dt\right)^p+ C_{BDG}\left([M^{\mathbb{G}}, M^{\mathbb{G}}]_{T\wedge \tau}+  \int_{0}^{T\wedge\tau}\vert{Z}^{\mathbb{G}}_{s}\vert^2 dt\right)^{p/2}\right]\nonumber\\
  &\leq 4^{p-1}(3^{p-1}+1)   E^{\widetilde{Q}}\left[\vert\xi \vert^p+\left(\int_{0}^{T\wedge\tau}\vert f(s)\vert ds\right)^p +\underset{0\leq s \leq T\wedge\tau}{\sup}(S^{+}_{s})^p\right]\nonumber\\
  &+ 4^{p-1}C_{BDG}E^{\widetilde{Q}}\left[\left([M^{\mathbb{G}}, M^{\mathbb{G}}]_{T\wedge \tau}+  \int_{0}^{T\wedge\tau}\vert{Z}^{\mathbb{G}}_{s}\vert^2 dt\right)^{p/2}\right].    \end{align}
{\bf Part 3.} A combination of It\^o and (\ref{RBSDEG}) implies that 
\begin{align}
d(Y^{\mathbb G})^2&=2Y_{-}^{{\mathbb{G}}}dY^{\mathbb G}+d[ Y^{\mathbb G},Y^{\mathbb G}]\nonumber\\
&=-2Y_{s-}^{\mathbb G}f(s)d(s\wedge\tau)-2Y_{-}^{\mathbb G}dK^{\mathbb G}+2Y_{-}^{\mathbb G} Z^{\mathbb G}dW^{\tau}-2Y_{-}^{\mathbb G}dM^{\mathbb G}\nonumber\\
&+d[M^{\mathbb G},M^{\mathbb G}]+d[K^{\mathbb G},K^{\mathbb G}]+(Z^{\mathbb G})^2d(s\wedge\tau)+2d[{K}^{\mathbb G},{M}^{\mathbb G}].\label{Ito1}\end{align}
As the three processes $[ M^{\mathbb{G}}, K^{\mathbb{G}}]$, $Y_{-}^{\mathbb{G}}Z^{\mathbb{G}}\is W^{\tau}$ and $Y_{-}^{\mathbb{G}}\is M^{\mathbb{G}}$ are $\widetilde Q$-local martingales, then there exists a sequence of $\mathbb G$-stopping times $(T_n)_n$ that increases to infinity such that when these processes are stopped at each $T_n$ they become true martingale. Thus, by using Young's inequality when it is convenient, we get 
\begin{align}
2\sum\vert\Delta{K}^{\mathbb G}\Delta{M}^{\mathbb G}\vert&\leq 2\sqrt{\sum(\Delta{K}^{\mathbb G})^2}\sqrt{\sum(\Delta{M}^{\mathbb G})^2}\leq \epsilon [M^{\mathbb G},M^{\mathbb G}]+\epsilon^{-1}[K^{\mathbb G}, K^{\mathbb G}]\nonumber\\
&\leq \epsilon [M^{\mathbb G},M^{\mathbb G}]+\epsilon^{-1}\sup_{0\leq s\leq T\wedge\tau}E^{\widetilde{Q}}[\sup_{0\leq t\leq T\wedge\tau}\vert Y^{{\mathbb{G}}}_t\vert\ \big|{\cal G}_s]K^{\mathbb G}\nonumber\\
&\leq \epsilon [M^{\mathbb G},M^{\mathbb G}]+\epsilon^{-3}\sup_{0\leq s\leq T\wedge\tau}E^{\widetilde{Q}}[\sup_{0\leq t\leq T\wedge\tau}\vert Y^{{\mathbb{G}}}_t\vert\ \big|{\cal G}_s]^2+\epsilon(K^{\mathbb G})^2,\end{align}
and for any stopping times $\tau_n$, for any $\epsilon\in(0,1)$, we derive
\begin{align}
&(1-\epsilon)[M^{\mathbb G},M^{\mathbb G}]_{\tau\wedge\tau_n}+\int_0^{\tau\wedge\tau_n}(Z^{\mathbb G}_s)^2ds\nonumber\\
&\leq (2+\epsilon^{-1})\sup_{0\leq s\leq\tau\wedge\tau_n}\vert Y^{\mathbb G}_s\vert^2+\left(\int_0^{\tau\wedge\tau_n}\vert f(s)\vert ds\right)^2\nonumber\\
&+\epsilon^{-3}\sup_{0\leq s\leq T\wedge\tau}E^{\widetilde{Q}}[\sup_{0\leq t\leq T\wedge\tau}\vert Y^{{\mathbb{G}}}_t\vert\ \big|{\cal G}_s]^2+2\epsilon({K}^{\mathbb G}_{\tau\wedge\tau_n})^2+2\sup_{0\leq s\leq\tau\wedge\tau_n}\vert (Y_{-}^{\mathbb G} \is (Z^{\mathbb G}\is{W}^{\tau}+{M}^{\mathbb G}))_s\vert.\label{Ito2}
\end{align}
On the other hand, by applying Lemma \ref{Lemma4.8FromChoulliThesis} to $Z^{\mathbb G}\is{W}^{\tau}+{M}^{\mathbb G}$ and $a=b=p$, we get 
\begin{align*}
&E^{\widetilde{Q}}\left[\sup_{0\leq s\leq\tau\wedge\tau_n}\vert (Y_{-}^{\mathbb G} \is (Z^{\mathbb G}\is{W}^{\tau}+{M}^{\mathbb G}))_s\vert^{p/2}\right]\\
&\leq C_p E^{\widetilde{Q}}\left[\sup_{0\leq s\leq\tau\wedge\tau_n}\vert Y^{\mathbb G} \vert^p\right]^{1/2}E^{\widetilde{Q}}\left[\left(\int_0^{\tau\wedge\tau_n} (Z^{\mathbb G}_s)^2ds+[M^{\mathbb G},M^{\mathbb G}]_{\tau\wedge\tau_n}\right)^{p/2}\right]^{1/2}\\
&\leq {{C_p^2}\over{\epsilon}} E^{\widetilde{Q}}\left[\sup_{0\leq s\leq\tau\wedge\tau_n}\vert Y^{\mathbb G} \vert^p\right]+\epsilon E^{\widetilde{Q}}\left[\left(\int_0^{\tau\wedge\tau_n} Z^{\mathbb G}_s)^2ds+[M^{\mathbb G},M^{\mathbb G}]_{\tau\wedge\tau_n}\right)^{p/2}\right].\end{align*} 
Thus, by combining this latter inequality (that is due to Young's inequality) and (\ref{Ito2}),  we put 
\begin{eqnarray}\label{C3}
C_3:= (2+\epsilon^{-1})^{p/2}+2^{p/2} C_p^2{\epsilon}^{-1}+\epsilon^{-3p/2}C_{DB},\end{eqnarray} and we derive
\begin{align*}\label{Ito3}
&((1-\epsilon)^{p/2}-(10)^{p/2}\epsilon)E^{\widetilde{Q}}\left[\left(\int_0^{\tau\wedge\tau_n} (Z^{\mathbb G}_s)^2ds+[M^{\mathbb G},M^{\mathbb G}]_{\tau\wedge\tau_n}\right)^{p/2}\right]\nonumber\\
&\leq 5^{p/2}E^{\widetilde{Q}}\left[C_3 \sup_{0\leq s\leq\tau\wedge\tau_n}\vert Y^{\mathbb G} \vert^p+\left(\int_0^{\tau\wedge\tau_n}\vert f(s)\vert ds\right)^p+(2\epsilon)^{p/2}({K}^{\mathbb G}_{\tau\wedge\tau_n})^p\right].
\end{align*} 
Therefore, by inserting (\ref{yesyes}) and (\ref{Control4KG}) in the above inequality, we obtain 
 \begin{align*}
&\left((1-\epsilon)^{p/2}-(10)^{p/2}\epsilon-(20\epsilon)^{p/2}{{C_{BDG}}\over{4}}\right)E^{\widetilde{Q}}\left[\left(\int_0^{\tau\wedge\tau_n} (Z^{\mathbb G}_s)^2ds+[M^{\mathbb G},M^{\mathbb G}]_{\tau\wedge\tau_n}\right)^{p/2}\right]\nonumber\\
&\leq  C_4 E^{\widetilde{Q}}\left[\left(\int_{0}^{T\wedge\tau}\vert f(s)\vert ds\right)^p +\underset{0\leq s \leq T\wedge\tau}{\sup}(S^{+}_{s})^p+\vert\xi \vert^p\right],\end{align*} 
where 
$$
C_4:=5^{p/2}( 3^{p-1}C_3C_{DB}+1)+(20\epsilon)^{p/2}{{3C_{DB}+1}\over{4}}.
$$
Hence, it is enough to choose $\epsilon>0$ very small such that $C_{\epsilon}:=(1-\epsilon)^{p/2}-(10)^{p/2}\epsilon-(20\epsilon)^{p/2}{{C_{BDG}}\over{4}}>0$, and combine the above inequality with (\ref{yesyes}), the proof of (\ref{estimate100}) follows immediately with the constant equal to $C=C_4C_{\epsilon}^{-1}+3^{p-1}C_{DB} $ that depends on $p$ only. This ends the proof of the theorem. 
     \end{proof} 
We end this subsection, by elaborating norm-estimate for the difference of solutions as follows. 
 \begin{theorem}\label{EstimatesUnderQtilde1} Suppose that ($Y^{\mathbb{G},i},Z^{\mathbb{G},i},K^{\mathbb{G},i}, M^{\mathbb{G},i}$)  is a  solution to the RBSDE (\ref{RBSDEG}) that correspond to  $(f^{(i)}, S^{(i)}, \xi^{(i)})$, for each $i=1,2$. Then for any $p>1$, there exist positive $C_1$ and $C_2$ that depend on $p$ only such that
  \begin{eqnarray}\label{estimate1001}
   &&E^{\widetilde{Q}}\left[\sup_{0\leq t\leq{T}\wedge\tau}\vert\delta Y^{{\mathbb{G}}}_{t}\vert^p\right]+E^{\widetilde{Q}}\left[\left(\int_{0}^{T\wedge\tau}\vert\delta Z^{{\mathbb{G}}}_{s}\vert^{2}ds+[\delta M^{\mathbb{G}}, \delta M^{\mathbb{G}}]_{T\wedge\tau}\right)^{p/2}\right]\nonumber\\
   &&\leq{C_1}E^{\widetilde{Q}}\left[\vert\delta\xi\vert^p+\left(\int_{0}^{T\wedge\tau}\vert \delta f(s)\vert ds\right)^p +\sup_{0\leq s \leq T\wedge\tau}\vert\delta S_{s}\vert^p\right]\nonumber\\
   &&+C_2\Vert\sup_{0\leq t\leq T\wedge\tau}\vert\delta S_{t}\vert\Vert_{L^p(\widetilde{Q})}^{p/2}\sqrt{\sum_{i=1}^2E^{\widetilde{Q}}\left[\vert\xi^{(i)}\vert^p+\left(\int_{0}^{T\wedge\tau}\vert{f}^{(i)}(s)\vert ds\right)^p +\sup_{0\leq s \leq T\wedge\tau}((S_s^{(i)})^+)^p\right]},\end{eqnarray}
  where  $ \delta Y^{\mathbb{G}},\delta Z^{\mathbb{G}},\delta K^{\mathbb{G}},\delta M^{\mathbb{G}},\delta f,\delta \xi,$ and $\delta S$ are given by
  \begin{eqnarray*}
   &&\delta Y^{\mathbb{G}}:=Y^{\mathbb{G},1}-Y^{\mathbb{G},2},\quad \delta Z^{{\mathbb{G}}}:=Z^{\mathbb{G},1}-Z^{\mathbb{G},2},\quad \delta M^{{\mathbb{G}}}:=M^{\mathbb{G},1}-M^{\mathbb{G},2}, \quad \delta K^{{\mathbb{G}}}:=K^{\mathbb{G},1}-K^{\mathbb{G},2},\\
   && \delta f:=f^{(1)}-f^{(2)},\quad \delta\xi:=\xi^{(1)}-\xi^{(2)},\quad  \delta S:=S^{(1)}-S^{(2)},
   \end{eqnarray*}
\end{theorem}
\begin{proof} This proof is achieved in two  parts, where we control in norm the first and the second terms of the left-hand-side of (\ref{estimate1001}). \\
{\bf {Part 1.}} Thanks to (\ref{RBSDE2Snell}), we conclude that  
   \begin{align*}
Y^{\mathbb{G},1}-Y^{\mathbb{G},2}&\leq \underset{\upsilon\in \mathcal{J}_{t\wedge\tau,T\wedge\tau}(\mathbb{G})}{\rm{ess}\sup}\hspace{2mm}E^{\widetilde{Q}}\left[\int_{t\wedge\tau}^{\upsilon\wedge\tau}  \delta f(s)ds +  \delta S_{\upsilon}1_{\{\upsilon\ <T\wedge\tau\}}+  \delta \xi 1_{\{\upsilon =T\wedge \tau\}}\ \Big|\ \mathcal{G}_{t}\right]\\
&\leq E^{\widetilde{Q}}\left[\int_{t\wedge\tau}^{T\wedge\tau} \vert \delta f(s)\vert ds +  \underset{t\wedge\tau\leq s\leq T\wedge\tau}{\sup}\vert\delta S_{s}\vert+ \vert \delta \xi \vert \Big|\ \mathcal{G}_{t}\right]
  \end{align*} 
  and 
   \begin{align*}
Y^{\mathbb{G},2}-Y^{\mathbb{G},1}&\leq \underset{\upsilon\in \mathcal{J}_{t\wedge\tau,T\wedge\tau}(\mathbb{G})}{\rm{ess}\sup}\hspace{2mm}E^{\widetilde{Q}}\left[\int_{t\wedge\tau}^{\upsilon\wedge\tau}  -\delta f(s)ds  - \delta S_{\upsilon}1_{\{\upsilon\ <T\wedge\tau\}} - \delta \xi 1_{\{\upsilon =T\wedge \tau\}}\ \Big|\ \mathcal{G}_{t}\right]\\
&\leq E^{\widetilde{Q}}\left[\int_{t\wedge\tau}^{T\wedge\tau} \vert \delta f(s)\vert ds +  \underset{t\wedge\tau\leq s\leq T\wedge\tau}{\sup}\vert\delta S_{s}\vert+ \vert \delta \xi \vert \Big|\ \mathcal{G}_{t}\right].
  \end{align*} 
Therefore, these yield
    \begin{equation*}
   \vert\delta  Y_{t}^{\mathbb{G}}\vert\leq {\widetilde M}_t:=E^{\widetilde{Q}}\left[\int_{0}^{T\wedge\tau} \vert \delta f(s)\vert ds +  \underset{0\leq s\leq T\wedge\tau}{\sup}\vert\delta S_{s}\vert+ \vert \delta \xi \vert \Big|\ \mathcal{G}_{t}\right].
  \end{equation*}
By applying Doob's inequality to $\widetilde M$ under $(\widetilde Q, \mathbb G)$ and using $(\sum_{i=1}^n \vert x_i\vert)^p\leq n^{p-1}\sum_{i=1}^n\vert x_i\vert^p$, we get 
\begin{align}\label{yesyes1}
E^{\widetilde{Q}}\left[\sup_{0\leq t\leq T\wedge\tau}\vert \delta Y^{{\mathbb{G}}}_t\vert^p\right]\leq 3^{p-1}C_{DB} E^{\widetilde{Q}}\left[\left(\int_{0}^{T\wedge\tau}\vert\delta f(s)\vert ds\right)^p +\sup_{0\leq s \leq T\wedge\tau}\vert\delta S_{s}\vert^p+\vert\delta\xi \vert^p\right], 
\end{align}
where $C_{DB}$ is the universal Doob's constant that depends on $p$ only.\\
{\bf Part 2.} Here we focus on  $\displaystyle\int_{0}^{ \cdot}(\delta  Z_{s}^{\mathbb{G}})^{2}ds+[\delta  M^{\mathbb{G}},\delta  M^{\mathbb{G}}]$. Thus, we apply It\^o to $(\delta{Y}^{\mathbb G})^2$ and get 
  \begin{align}\label{Ito000}
&d [\delta{M}^{\mathbb G},\delta{M}^{\mathbb G}]+(\delta{Z}^{\mathbb G})^2d(\cdot\wedge\tau)+d [\delta{K}^{\mathbb G},\delta{K}^{\mathbb G}]\nonumber\\
&=d(\delta{Y}^{\mathbb G})^2+2\delta  Y_{-}^{\mathbb{G}}\delta {f}d(\cdot\wedge\tau)+2\delta  Y_{-}^{\mathbb{G}}d\delta{K}^{\mathbb{G}}-2\Delta\delta{K}^{\mathbb G}d\delta{M}^{\mathbb G}+2\delta{Y}_{-}^{\mathbb G}d(\delta{Z}^{\mathbb G}\is{W}^{\tau}-\delta{M}^{\mathbb G})). \end{align}
Thus, we use this equality and mimic the first step of Part 3 in the proof of Theorem \ref{EstimatesUnderQtilde}, and derive
  \begin{align}
&{\cal Q}^{\mathbb G}:= [\delta{M}^{\mathbb G},\delta{M}^{\mathbb G}]+\int_0^{\cdot} (\delta{Z}^{\mathbb G}_s)^2d(s\wedge\tau)\nonumber\\
&\leq \sup_{0\leq{t}\leq\cdot}(\delta{Y}^{\mathbb G}_t)^2+2\int_{0}^{\cdot}\delta  Y_{s-}^{\mathbb{G}}\delta  f(s)d{s}+2\delta  Y_{-}^{\mathbb{G}}\is\delta{K}^{\mathbb{G}}-2[\delta{K}^{\mathbb G},\delta{M}^{\mathbb G}]+\overbrace{2\sup_{0\leq{t}\leq\cdot}\vert (\delta{Y}_{-}^{\mathbb G}\is(\delta{Z}^{\mathbb G}\is{W}^{\tau}-\delta{M}^{\mathbb G}))_t\vert}^{=:\Gamma^{\mathbb G}}\nonumber\\
 &\leq 2\sup_{0\leq{t}\leq\cdot}(\delta{Y}^{\mathbb G}_t)^2+\int_{0}^{\cdot}\vert\delta  f(s)\vert{d}{s}+2\delta  S_{s-}^{\mathbb{G}}\is\delta{K}-2\Delta(\delta{K}^{\mathbb G})\is\delta{M}^{\mathbb G}+ \Gamma^{\mathbb G}.\label{Ito10}
 \end{align}
 As $M^{\mathbb{G}}$ and $\delta  Z^{\mathbb{G}}\is W^{\tau}-M^{\mathbb{G}}$ are $(\widetilde Q,\mathbb G)$-local martingales and $^{p,\mathbb G}(\Delta(\delta  Y_{-}^{\mathbb{G}}))=-\Delta(\delta{K}^{\mathbb G})$, we consider  $(T_n)_n$  a sequence of $\mathbb G$-stopping times that increases to infinity such that when the processes are stopped at each $T_n$ they become true martingale, and by applying Lemma \ref{Lemma4.8FromChoulliThesis} and using Young's inequality afterwards, we obtain
 \begin{align*}
& E^{\widetilde{Q}}\left[(\Gamma^{\mathbb G}_{\tau\wedge{T_n}})^{p/2}\right]\leq C^p\epsilon^{-1}\Vert\sup_{0\leq s \leq T\wedge\tau} \vert\delta Y_{s}^{\mathbb{G}}\vert\Vert_{L^p(\widetilde{Q})}^p+\epsilon {E}^{\widetilde{Q}}\left[\left(\int_{0}^{ T_n\wedge\tau}(\delta  Z_{s}^{\mathbb{G}})^{2}ds+ [\delta  M^{\mathbb{G}},\delta  M^{\mathbb{G}}]_{T_n\wedge}\right)^{p/2}\right],\\
&\mbox{and}\\
 &E^{\widetilde{Q}}\left[\sup_{0\leq{t}\leq\tau\wedge{T_n}}\vert-2[\delta{K}^{\mathbb G},\delta{M}^{\mathbb G}]_t\vert^{p/2}\right]\leq C^p\epsilon^{-1}\Vert\sup_{0\leq s \leq T\wedge\tau} \vert\delta Y_{s}^{\mathbb{G}}\vert\Vert_{L^p(\widetilde{Q})}^p+\epsilon {E}^{\widetilde{Q}}\left[[\delta  M^{\mathbb{G}},\delta  M^{\mathbb{G}}]_{T_n\wedge}^{p/2}\right].
  \end{align*}
Therefore, by taking expectation on both sides of (\ref{Ito10}) and inserting the two inequalities above in the resulting inequality afterwards, we get 
 \begin{align}
 &(1-2{\epsilon}5^{p/2})E^{\widetilde{Q}}\left[ ({\cal Q}^{\mathbb G}_{\tau\wedge{T_n}})^{p/2}\right]\leq 2(10C)^{p/2}\epsilon^{-1}\Vert\sup_{0\leq s \leq T\wedge\tau} \vert\delta Y_{s}^{\mathbb{G}}\vert\Vert_{L^p(\widetilde{Q})}^p\nonumber\\
 &+5^{p/2}\Vert\int_{0}^{\tau\wedge{T_n}}\vert\delta  f(s)\vert{d}{s}\Vert_{L^p(\widetilde{Q})}^p+5^{p/2}\sqrt{\Vert\sup_{0\leq{t}\leq{\tau\wedge{T_n}}}\vert\delta{S}_t\vert\Vert_{L^p(\widetilde{Q})}^p\Vert\mbox{Var}_{\tau\wedge{T_n}}(\delta{K}^{\mathbb G})\Vert_{L^p(\widetilde{Q})}^p}.\label{Control4QG}
    \end{align}
Furthermore, remark that $\mbox{Var}_{\tau\wedge{T_n}}(\delta{K}^{\mathbb G})\leq {K}^{\mathbb G,1}_{\tau\wedge{T_n}}+{K}^{\mathbb G,2}_{\tau\wedge{T_n}}$. Thus, by inserting this latter inequality in  (\ref{Control4QG})  and applying Theorem \ref{EstimatesUnderQtilde} to each ${K}^{\mathbb G, i}$, $i=1,2$, and using Fatou afterwards, we get 

 \begin{align*}
 &(1-2{\epsilon}5^{p/2})E^{\widetilde{Q}}\left[ ({\cal Q}^{\mathbb G}_{T\wedge\tau})^{p/2}\right]\leq 2(10C)^{p/2}\epsilon^{-1}\Vert\delta Y^{\mathbb{G}}\Vert_{{\mathbb{D}}_{ T\wedge\tau}(\widetilde{Q},p)}^p+5^{p/2}\Vert\int_{0}^{\tau\wedge{T}}\vert\delta  f(s)\vert{d}{s}\Vert_{L^p(\widetilde{Q})}^p\nonumber\\
 &+5^{p/2}\sqrt{C} \Vert\delta S\Vert_{{\mathbb{D}}_{ T\wedge\tau}(\widetilde{Q},p)}^{p/2}\sqrt{\sum_{i=1}^2\left\{\Vert\xi^{(i)}\Vert_{L^p(\widetilde{Q})}^p+\Vert\int_{0}^{T\wedge\tau}\vert{f}^{(i)}(s)\vert ds\Vert_{L^p(\widetilde{Q})}^p +\Vert(S^{(i)})^+\Vert_{{\mathbb{D}}_{ T\wedge\tau}(\widetilde{Q},p)}^p\right\}}.
    \end{align*}
Therefore, by combinung this inequality with (\ref{yesyes1}) and putting
  \begin{align*}
 \epsilon=5^{-p/2}/4,\quad C_1=3^{p-1}4C_{DB}(50C)^{p/2}+5^{p/2}4+3^{p-1}C_{DB}\quad\mbox{and}\quad C_2=2\sqrt{C}5^{p/2},
  \end{align*}
the theorem follows immediately. This ends the proof of the theorem.
     \end{proof} 
\subsection{Existence for the $\mathbb G$-RBSDE and its relationship to $\mathbb F$-RBSDE}\label{Subsection4.1}
In this subsection, we prove the existence and the uniqueness of the solution to the RBSDE (\ref{RBSDEG}), and we establish explicit connection between this RBSDE and its $\mathbb F$-RBSDE counterpart, and highlight the explicit relationship between their solutions as well. 
\begin{theorem}\label{abcde}Let $p\in (1,\infty)$, suppose that (\ref{MainAssumption}) holds, and consider $(f^{\mathbb{F}},S^{\mathbb{F}})$ and $(\xi^{\mathbb{F}},V^{\mathbb F})$ given by 
  \begin{eqnarray}
  f^{\mathbb{F}}:={\widetilde{\cal E}}f,\quad S^{\mathbb{F}}:= {\widetilde{\cal E}}S,\quad  \xi^{\mathbb{F}}:={\widetilde{\cal E}_T}h_{T},\quad V^{\mathbb F}:=1-{\widetilde{\cal E}},\quad\mbox{where}\quad {\widetilde{\cal E}}:={\cal E}\left(-{\widetilde G}^{-1}\is D^{o,\mathbb{F}}\right). \label{ProcessVFandXiF}
  \end{eqnarray}  
Then the following  assertions hold.\\
{\rm{(a)}} The following RBSDE under $\mathbb F$, associated to the triplet $ \left(f^{\mathbb{F}},S^{\mathbb{F}}, \xi^{\mathbb F}\right)$,
\begin{eqnarray}\label{RBSDEF}
\begin{cases}
Y_{t}= \displaystyle\xi^{\mathbb{F}}+\int_{t}^{T}f^{\mathbb{F}}(s)ds+\int_{t}^{T}h_{s}dV^{\mathbb{F}}_{s}+K_{T}-K_{t}-\int_{t}^{T}Z_{s}dW_{s},\\
Y_{t}\geq S_{t}^{\mathbb{F}}1_{\{t\ <T\}}+\xi^{\mathbb{F}}1_{\{t\ =T\}},\quad
 \displaystyle\int_{0}^{T}(Y_{t-}-S_{t-}^{\mathbb{F}})dK_{t}=0 ,\quad P\mbox{-a.s.,}
\end{cases}
\end{eqnarray}
has a unique $L^p(P,\mathbb F)$-solution $(Y^{\mathbb F},  Z^{\mathbb F}, K^{\mathbb F})$, and 
 \begin{eqnarray}\label{RBSDE2SnellF}
 Y^{\mathbb F}_t=\rm{ess}\sup_{\sigma\in \mathcal{J}_{t}^{T}(\mathbb{F})}E\left[\int_{t\wedge\tau}^{\sigma}f^{\mathbb F}(s)ds+\int_{t\wedge\tau}^{\sigma}h_s dV^{\mathbb F}_s + S_{\sigma}^{\mathbb F}1_{\{\sigma <T\}}+\xi^{\mathbb F} I_{\{\sigma =T\}}\ \Big|\ \mathcal{F}_{t}\right],\quad 0\leq t\leq T.
 \end{eqnarray}
{\rm{(b)}} The RBSDE  defined in (\ref{RBSDEG}) has a unique  $L^p(\widetilde{Q},\mathbb G)$-solution $(Y^{\mathbb{G}},Z^{\mathbb{G}},K^{\mathbb{G}},M^{\mathbb{G}})$ given by 
\begin{eqnarray}
   Y^{\mathbb{G}}= \frac{Y^{\mathbb{F}}}{\widetilde{\cal E}}I_{\Lbrack0,\tau\Lbrack}+\xi{I_{\Lbrack\tau,+\infty\Lbrack}},\ 
  Z^{\mathbb{G}}=\frac{Z^{\mathbb{F}}}{{\widetilde{\cal E}}_{-}} I_{\Rbrack0,\tau\Rbrack},\  
   K^{\mathbb{G}}=\frac{1}{{\widetilde{\cal E}}_{-}}\is (K ^{\mathbb{F}})^{\tau}\ \mbox{and}\ 
      M^{\mathbb{G}}=\left(h-\frac{Y^{\mathbb{F}}}{{\widetilde{\cal E}}}\right)\is N^{\mathbb{G}}.\label{secondrelation}
       \end{eqnarray}
        \end{theorem}
\begin{proof} Assertion (a) is the linear case of a general RBSDE under $\mathbb F$ given in Subsection \ref{GeneralRBSDEfromG2F}, see (\ref{RBSDEFGENERAL}). Thus, the proof of the existence and uniqueness of the $L^p(\mathbb F, P)$-solution will be omitted here, and we refer the reader to Subsection \ref{GeneralRBSDEfromG2F}. Furthermore, the proof of (\ref{RBSDE2SnellF}) mimics exactly the proof of (\ref{RBSDE2Snell}). Thus, the remaining part of this proof will focus on proving assertion (b). To this end, on the one hand, we remark that in virtue of Theorem \ref{EstimatesUnderQtilde} and (\ref{MainAssumption}), we conclude that a solution to (\ref{RBSDEG}), when it exists, it is in fact an $L^p(\widetilde{Q}, \mathbb G)$-solution. On the other hand, thanks to Theorem \ref{EstimatesUnderQtilde1} and the assumption (\ref{MainAssumption}), we deduce that there is at most one $L^p(\widetilde{Q},\mathbb G)$-solution. Thus, the rest of this proof focuses on proving the existence of the solution to (\ref{RBSDEG}) that is given by   (\ref{secondrelation}). To this end, we put 
\begin{eqnarray}\label{Yoverline}
\overline{Y}:= \frac{Y^{\mathbb{F}}}{\widetilde{\cal E}}I_{\Lbrack0,\tau\Lbrack}+\xi{I_{\Lbrack\tau,+\infty\Lbrack}},\quad \overline{Z}:=\frac{Z^{\mathbb{F}}}{{\widetilde{\cal E}}_{-}} I_{\Rbrack0,\tau\Rbrack},\quad \overline{M}:=\left(h-\frac{Y^{\mathbb{F}}}{{\widetilde{\cal E}}}\right)\is N^{\mathbb{G}}\quad \overline{K}:=\frac{1}{{\widetilde{\cal E}}_{-}}\is (K ^{\mathbb{F}})^{\tau}, 
\end{eqnarray}
 and prove that  $(\overline{Y}, \overline{Z}, \overline{M},\overline{K}$  is a solution to  (\ref{RBSDEG}). Hence, we put
\begin{eqnarray*}\label{Gamma}
\Gamma:=\frac{Y^{\mathbb F}}{\widetilde{\cal E}}= Y^{\mathbb F}{\cal E}(G^{-1}\is D^{o,\mathbb F}).\end{eqnarray*}
 and remark that, in virtue of the first equality in (\ref{secondrelation}), we have 
\begin{eqnarray}\label{YGGamma}
\overline{Y}=\Gamma^{\tau} +(h-\Gamma)\is D.
\end{eqnarray}
Thanks to It\^o, the facts that ${\widetilde{\cal E}}^{-1}={\cal E}(G^{-1}\is D^{o,\mathbb F})$ and ${\widetilde{\cal E}}={\widetilde{\cal E}}_{-}G/{\widetilde{G}}$, (\ref{RBSDEF}) and (\ref{ProcessVFandXiF}), we derive 
\begin{align}
d\Gamma&={{\Gamma}\over{\widetilde G}}dD^{o,\mathbb F}+{1\over{\widetilde{\cal E}_{-}}}dY^{\mathbb F}={{\Gamma- h}\over{\widetilde G}} dD^{o,\mathbb F}-f(t)dt-{1\over{\widetilde{\cal E}_{-}}}dK^{\mathbb F}+{{Z^{\mathbb F}}\over{\widetilde{\cal E}_{-}}}dW.\label{GammaYF}
\end{align}
Thus, by inserting this latter equation in (\ref{YGGamma}) and arranging terms we get 
\begin{align}\label{SDE4YG}
d\overline{Y}=-f(t)d(t\wedge\tau)-{1\over{\widetilde{\cal E}_{-}}}d(K^{\mathbb F})^{\tau}+(h- \Gamma)dN^{\mathbb{G}}+{{Z^{\mathbb F}}\over{\widetilde{\cal E}_{-}}}dW^{\tau}.
\end{align} 
This proves that the processes defined in (\ref{secondrelation}) satisfy the first equation in (\ref{RBSDEG}). To prove the second condition in (\ref{RBSDEG}), it is enough to remark that we have 
\begin{eqnarray*}
Y^{\mathbb{F}}_{t}\geq S_{t}^{\mathbb{F}}I_{\{t\ <T\}}+\xi^{\mathbb{F}}I_{\{t\ =T\}},\end{eqnarray*}
which implies that for any  $t\in[0,T)$, $Y_{t}^{\mathbb{F}}({\widetilde{\cal E}}_{t})^{-1}I_{\{t\ <\tau\}}\geq  S_{t}I_{\{t\ <\tau\}}.$
    This is obviously equivalent to the second condition of (\ref{RBSDEG}). To prove the Skorokhod condition (the last condition in (\ref{RBSDEG})), we  use the Skorokhod condition for the triplet $(Y^{\mathbb F}, Z^{\mathbb F}, K^{\mathbb F})$ (as it is the solution to the RBSDE (\ref{RBSDEF}) with the data-triplet $(f^{\mathbb F}, S^{\mathbb F}, \xi^{\mathbb F})$) given by 
  \begin{eqnarray}\label{SkorokhodF}
  \int_0^T(Y^{\mathbb F}_{t-}-S^{\mathbb F}_{t-})dK^{\mathbb F}_t=0,\quad P\mbox{-a.s..}\end{eqnarray}
   As $\overline{Y}_{-}-S_{-}\geq 0$ on $\Rbrack0,\tau\Rbrack$ and  $\overline{K}$ is an increasing process, we get 
  \begin{eqnarray*}
\int_{0}^{T}(\overline{Y}_{t-}-S_{t-})d\overline{K}_{t}=\int_{0}^{T\wedge\tau}(Y^{\mathbb F}_{t-}-S^{\mathbb F}_{t-}){\widetilde{\cal E}_{t-}}^{-2}dK^{\mathbb F}_t\leq \int_{0}^{T}(Y^{\mathbb F}_{t-}-S^{\mathbb F}_{t-}){\widetilde{\cal E}_{t-}}^{-2}dK^{\mathbb F}_t=0,\quad P\mbox{-a.s..} \end{eqnarray*}
It is clear that the last equality is equivalent to (\ref{SkorokhodF}) due the fact that $K^{\mathbb F}$ is nondecreasing and $Y^{\mathbb F}_{-}-S^{\mathbb F}_{-}\geq 0$ .  This ends the proof of the theorem.\end{proof}

\begin{remark} One can prove, under weaker integrability conditions than those of (\ref{MainAssumption}), that  any solution to the RBSDE  (\ref{RBSDEG}), denoted by $(Y,Z,K, M)$, coincides with $(Y^{\mathbb G},Z^{\mathbb G},K^{\mathbb G}, M^{\mathbb G})$ defined in (\ref{secondrelation}). To this end, thanks to the Doob-Meyer decomposition under $(\widetilde{Q},\mathbb G)$, we remark that $(Y,Z,K, M)=(Y^{\mathbb G},Z^{\mathbb G},K^{\mathbb G}, M^{\mathbb G})$ is equivalent to $Y=Y^{\mathbb G}$. To prove this equality, we notice that due to (\ref{RBSDE2Snell}), we have 
\begin{eqnarray*}
Y+\int_0^{\tau\wedge{T}\wedge\cdot}f(s)ds={\cal S}(X^{\mathbb G}; \mathbb G, \widetilde{Q})\quad\mbox{with}\quad X^{\mathbb G}:=\int_0^{\tau\wedge{T}\wedge\cdot}f(s)ds+SI_{\Lbrack0,\tau\wedge{T}\Lbrack}+h_{\tau\wedge{T}}I_{\Lbrack\tau\wedge{T},+\infty\Lbrack}.
\end{eqnarray*}
Therefore, to apply Theorem \ref{SnellEvelopG2F}-(b), we need to find the unique pair $(X^{\mathbb F}, k^{(pr)})$ associated to $X^{\mathbb G}$. To this end, we remark that 
\begin{eqnarray*}
SI_{\Lbrack0,\tau\wedge{T}\Lbrack}=SI_{\Lbrack0,\tau\Lbrack}I_{\Lbrack0,{T}\Lbrack}\quad\mbox{and}\quad h_{\tau\wedge{T}}I_{\Lbrack\tau\wedge{T},+\infty\Lbrack}=h_{\tau\wedge{T}}I_{\Lbrack0,\tau\Lbrack}+h_{T}I_{\Lbrack0,{T}\Lbrack}I_{\Lbrack{T},+\infty\Lbrack},\end{eqnarray*}
and derive
\begin{eqnarray*}
X^{\mathbb F}=\int_0^{T\wedge\cdot}f(s)ds+SI_{\Lbrack0,{T}\Lbrack}+h_{T}I_{\Lbrack{T},+\infty\Lbrack},\quad
k^{(pr)}=k^{(op)}=\int_0^{T\wedge\cdot}f(s)ds+hI_{\Lbrack0,T\Lbrack}+h_{T}I_{\Lbrack{T},+\infty\Lbrack}.
\end{eqnarray*}
Furthermore, we have 
\begin{eqnarray*}
&&{\widetilde{\cal E}}X^{\mathbb F}-k^{(op)}\is{\widetilde{\cal E}}=\int_0^{T\wedge\cdot}f^{\mathbb F}(s)ds+S^{\mathbb F} I_{\Lbrack0,T\Lbrack}+(h\is V^{\mathbb F})^T+\xi^{\mathbb F}I_{\Lbrack{T},+\infty\Lbrack},\\
&& k^{(op)}{\widetilde{\cal E}}-k^{(op)}\is{\widetilde{\cal E}}=\int_0^{T\wedge\cdot}f^{\mathbb F}(s)ds+(h\is V^{\mathbb F})^T+\widetilde{\cal E}hI_{\Lbrack0,T\Lbrack} +\xi^{\mathbb F}I_{\Lbrack{T},+\infty\Lbrack}.
\end{eqnarray*}
Remark also that
\begin{eqnarray*}
Y^{\mathbb F}+L^{\mathbb F}={\cal S}\left(L^{\mathbb F}+\xi^{\mathbb F}I_{\Lbrack{T},+\infty\Lbrack}+S^{\mathbb F} I_{\Lbrack0,T\Lbrack};\mathbb F, P\right),\quad L^{\mathbb F}:=\int_0^{T\wedge\cdot}f(s)ds+\int_0^{T\wedge\cdot}h_sdV^{\mathbb F}_s.
\end{eqnarray*}
Thus, by directly applying  Theorem \ref{SnellEvelopG2F}-(b) to $Y$, on $\Lbrack0,T\Rbrack$, we have
\begin{align*}
Y+\int_0^{\tau\wedge{T}\wedge\cdot}f(s)ds&={\cal S}(X^{\mathbb G}; \mathbb G, \widetilde{Q})\\
&={{{\cal S}\left(X^{\mathbb F}{\widetilde{\cal E}}-k^{(op)}\is \widetilde{\cal E};\mathbb F, P\right)}\over{ \widetilde{\cal E}}}I_{\Lbrack0,\tau\Lbrack}+{{L^{\mathbb F}+\widetilde{\cal E}hI_{\Lbrack0,T\Lbrack} +\xi^{\mathbb F}I_{\Lbrack{T},+\infty\Lbrack}}\over{\widetilde{\cal E}}}\is N^{\mathbb G}\\
&={{Y^{\mathbb F}+L^{\mathbb F}}\over{ \widetilde{\cal E}}}I_{\Lbrack0,\tau\Lbrack}+{{L^{\mathbb F}}\over{ \widetilde{\cal E}}}\is N^{\mathbb G}+\left(hI_{\Lbrack0,T\Lbrack} +h_T{I}_{\Lbrack{T},+\infty\Lbrack}\right)\is N^{\mathbb G}\\
&={{Y^{\mathbb F}}\over{ \widetilde{\cal E}}}I_{\Lbrack0,\tau\Lbrack}+{1\over{ \widetilde{\cal E}_{-}}}\is (L^{\mathbb F})^{\tau}+h_{T\wedge\cdot}\is D- {{h}\over{\widetilde{G}}}I_{\Rbrack0,\tau\wedge{T}\Rbrack}\is D^{o,\mathbb F}\\
&={{Y^{\mathbb F}}\over{ \widetilde{\cal E}}}I_{\Lbrack0,\tau\Lbrack}+\int_0^{\tau\wedge{T}\wedge\cdot}f(s)ds+\xi{I}_{\Lbrack\tau,+\infty\Lbrack}=Y^{\mathbb G}+\int_0^{\tau\wedge{T}\wedge\cdot}f(s)ds.
\end{align*}
The fourth equality follows from the following lemma, that we prove in Appendix \ref{Appendix4Proofs}.\end{remark}
\begin{lemma}\label{L/EpsilonTilde} For any $\mathbb F$-semimartingale $L$, the following holds.
\begin{eqnarray*}
{{L}\over{\widetilde{\cal E}}}I_{\Lbrack0,\tau\Lbrack}+{{L}\over{\widetilde{\cal E}}}\is N^{\mathbb G}={1\over{\widetilde{\cal E}_{-}}}\is L^{\tau}.
\end{eqnarray*}
\end{lemma}
\section{The case of linear RBSDE with unbounded horizon} \label{LinearUnboundedSection}
This section focuses on the following RBSDE 
\begin{eqnarray}\label{RBSDEGinfinite}
\begin{cases}
dY=-f(t)d(t\wedge\tau)-d(K+M)+ZdW^{\tau},\quad Y_{\tau}=\xi=h_{\tau},\\
\\
Y_{t}\geq S_{t}\quad; 0\leq t<  \tau,\quad \displaystyle\int_{0}^{\tau}(Y_{t-}-S_{t-})dK_{t}=0,\quad P\mbox{-a.s..}
\end{cases}
\end{eqnarray}
It is important to mention that the probability $\widetilde{Q}$ depends heavily on the finite horizon planning $T$, and the process ${\widetilde Z}^{\tau}$ defined in (\ref{Qtilde}) might not be a uniformly integrable martingale. This raises serious challenges in different direction.
 \subsection{Existence, uniqueness and estimates}
As we aforementioned it, the probability $\widetilde Q$ depends heavily on $T$, see \cite{Choulli1} for details, and in general ${\widetilde Z}^{\tau}$ might not be a uniformly integrable martingale.  Thus, the fact of letting $T$ goes to infinity triggers serious challenges in both the technical and the conceptual sides. In fact, both the condition (\ref{MainAssumption}) and the RBSDE (\ref{RBSDEF}) might not make sense when we take $T$ to infinity due to the fact the limit of $h_T$ when $T$ goes to infinity might not exist even. Our approach to these challenges will be in two steps. The first step relies on the following lemma and the two theorems that follow it, and aims to get rid-off of $\widetilde Q$ in the left-hand-sides of the estimates of Theorems \ref{EstimatesUnderQtilde} and \ref{EstimatesUnderQtilde1}.
     \begin{lemma}\label{technicallemma1} Let $T\in (0,+\infty)$, $\widetilde{Q}$ be the probability given in (\ref{Qtilde}), and $\widetilde{\cal E}$ be the process defined in (\ref{ProcessVFandXiF}). Then the following assertions hold.\\
     {\rm{(a)}} For any $p\in(1,+\infty)$ and any RCLL $\mathbb G$-semimartingale $Y$, we have
    \begin{eqnarray}\label{Equality4YG}
     E\left[\sup_{0\leq s\leq{T\wedge\tau}}{\widetilde{\cal E}}_s\vert{Y_s}\vert^p\right]\leq {G_0^{-1} }E^{\widetilde{Q}}\left[\sup_{0\leq s\leq{T\wedge\tau}}\vert{Y_s}\vert^p\right].     \end{eqnarray}
      {\rm{(b)}} For any $a\in (0,+\infty)$ and RCLL, nondecreasing and $\mathbb G$-optional process $K$ with $K_0=0$, we have 
    \begin{eqnarray}\label{Equality4KG}
     E\left[\left(\int_0^{T\wedge\tau} ({\widetilde{\cal E}}_{s-})^{a}dK_s\right)^{1/a}\right]\leq \underbrace{3^{1/a}(5+(\max(a, a^{-1}))^{1/a})}_{=:\kappa(a)} G_0^{-1} E^{\widetilde{Q}}\left[K_{T\wedge\tau}^{1/a}+\sum_{0< s\leq _{T\wedge\tau}} \widetilde{G}_t(\Delta K_s)^{1/a}\right].\end{eqnarray}
        {\rm{(c)}} For any $p>1$ and any nonnegative and $\mathbb G$-optional process $H$, we have 
     \begin{eqnarray}\label{Equality4MG}
    E\left[({\widetilde{\cal E}}_{-}^{2/p}H\is [N^{\mathbb G},N^{\mathbb G}])_{T\wedge\tau} ^{p/2}\right]\leq \kappa(a)G_0^{-1}  E^{\widetilde{Q}}\left[(H\is [N^{\mathbb G},N^{\mathbb G}]_{T\wedge\tau})^{p/2}+ (H^{p/2}\widetilde{G}\is \mbox{Var}(N^{\mathbb G}))_{T\wedge\tau}\right].
     \end{eqnarray}
     {\rm{(d)}} For any $p>1$ and any nonnegative and $\mathbb F$-optional process $H$, we have 
     \begin{eqnarray}\label{Equality4MGOptionalF}
    E\left[({\widetilde{\cal E}}_{-}^{2/p}H\is [N^{\mathbb G},N^{\mathbb G}])_{T\wedge\tau} ^{p/2}\right]\leq \kappa(a)G_0^{-1}  E^{\widetilde{Q}}\left[(H\is [N^{\mathbb G},N^{\mathbb G}]_{T\wedge\tau})^{p/2}+ 2(H^{p/2}I_{\Rbrack0,\tau\Lbrack}\is D^{o,\mathbb F})_T\right].
     \end{eqnarray}
        \end{lemma}
  For the sake of simple exposition, we relegate the proof of this lemma to Appendix \ref{Appendix4Proofs}. In the following, we elaborate estimates for the solution to  (\ref{RBSDEG}) under the probability $P$ instead.  
  \begin{theorem}\label{estimates} For any $p>1$, there exists a positive constant $C=C(p)$ that depends on $p$ only such that the unique solution to the RBSDE (\ref{RBSDEG}), denoted by ($Y^{\mathbb{G}},Z^{\mathbb{G}},K^{\mathbb{G}}, M^{\mathbb{G}}$), satisfies  
\begin{eqnarray}\label{Estimate4P}
&&E\left[\sup_{0\leq t\leq T}\widetilde{\cal E}_t\vert{Y}^{{\mathbb{G}}}_{t}\vert^p+\left(\int_{0}^{T}(\widetilde{\cal E}_{s-})^{2/p}\vert Z^{{\mathbb{G}}}_{s}\vert^{2}ds\right)^{p/2}+\left((\widetilde{\cal E}_{-})^{1/p}\is K^{\mathbb{G}}_T\right)^p+((\widetilde{\cal E}_{-})^{2/p}\is[ M^{\mathbb{G}}, M^{\mathbb{G}}]_T)^{p/2}\right]\nonumber\\
&&\leq C E^{\widetilde{Q}}\left[\left(\int_{0}^{T\wedge\tau}\vert f(s)\vert ds\right)^p +\underset{0\leq s \leq T\wedge\tau}{\sup}(S^{+}_{s})^p+\vert\xi \vert^p\right],
\end{eqnarray}
where $\widetilde{\cal E}$ is the process given by 
\begin{eqnarray}\label{Xitilde111}
\widetilde{\cal E}:={\cal E}(-{\widetilde{G}}^{-1}\is D^{o,\mathbb{F}}).\end{eqnarray}
\end{theorem}

\begin{proof}  
By applying Lemma \ref{technicallemma1}-(b) to the process $K_t=\int_0^{t\wedge\tau} (Z^{\mathbb G}_{s})^{2}ds$ and $a=2/p$, we get 
 \begin{eqnarray}\label{Estimate4ZGepsilon}
     E\left[\left(\int_0^{T\wedge\tau} ({\widetilde{\cal E}}_{s-})^{2/p}(Z^{\mathbb G}_{s})^{2}ds\right)^{p/2}\right]\leq \kappa(a)G_0^{-1} E^{\widetilde{Q}}\left[\left(\int_0^{T\wedge\tau}
     (Z^{\mathbb G}_{s})^{2}ds\right)^{p/2}\right].\end{eqnarray}
By applying Lemma \ref{technicallemma1}-(a) to the process $Y=Y^{\mathbb G}$, we get 
  \begin{eqnarray}\label{Estimate4YGepsilon}
     E\left[\sup_{0\leq s\leq{T\wedge\tau}}{\widetilde{\cal E}}_s\vert{Y_s^{\mathbb G}}\vert^p\right]\leq {G_0^{-1} }E^{\widetilde{Q}}\left[\sup_{0\leq s\leq{T\wedge\tau}}\vert{Y_s^{\mathbb G}}\vert^p\right]. \end{eqnarray}

By applying Lemma \ref{technicallemma1}-(b) to the process $K=K^{\mathbb G}$ and $a=1/p$,  and using the fact that we always have  $\sum_{0< s\leq _{T\wedge\tau}} \widetilde{G}_s(\Delta K_s^{\mathbb G})^p\leq (K_{T\wedge\tau}^{\mathbb G})^p$, we get 
 \begin{eqnarray}\label{Estimate4KGepsilon}
     E\left[\left(\int_0^{T\wedge\tau} ({\widetilde{\cal E}}_{s-})^{1/p}dK_s^{\mathbb G}\right)^p\right]\leq{{\kappa(a)}\over{ G_0}} E^{\widetilde{Q}}\left[(K_{T\wedge\tau}^{\mathbb G})^p+\sum_{0\leq s\leq _{T\wedge\tau}} \widetilde{G}_s(\Delta K_s^{\mathbb G})^p\right]\leq {{2\kappa(a)}\over{ G_0}}  E^{\widetilde{Q}}\left[(K_{T\wedge\tau}^{\mathbb G})^p\right].\end{eqnarray}
Thanks to Theorem \ref{abcde}, we write $[M^{\mathbb G}, M^{\mathbb G}]=H\is [N^{\mathbb G}, N^{\mathbb G}]$ with $H:=(h-Y^{\mathbb F}{\widetilde {\cal E}}^{-1})^2$ is a nonnegative $\mathbb F$-optional process. Thus, a direct application of Lemma  \ref{technicallemma1}-(d), we get \\
 \begin{align}\label{Equality4MG00}
    E\left[({\widetilde{\cal E}}_{-}^{2/p}H\is [N^{\mathbb G},N^{\mathbb G}])_{T\wedge\tau} ^{p/2}\right]\leq C(p)G_0^{-1}  E^{\widetilde{Q}}\left[(H\is [N^{\mathbb G},N^{\mathbb G}]_{T\wedge\tau})^{p/2}+ 2(H^{p/2}I_{\Rbrack0,\tau\Lbrack}\is D^{o,\mathbb F})_T\right].
     \end{align}
Thus, we need to control the second term in the right-hand-side of the above inequality. To this end, we remark that $(H^{p/2}I_{\Rbrack0,\tau\Lbrack}\is D^{o,\mathbb F})\leq 2^{p-1}(\vert{h}\vert^p+\vert{Y}^{\mathbb G}\vert^pI_{\Rbrack0,\tau\Lbrack}\is D^{o,\mathbb F})$. Thus, by using this, we derive 
\begin{align*}
2E^{\widetilde{Q}}\left[(H^{p/2}I_{\Rbrack0,\tau\Lbrack}\is D^{o,\mathbb F})_T\right]\leq 2^pE^{\widetilde{Q}}\left[\vert{h}_{\tau}\vert^p I_{\{\tau\leq{T}\}}\right]+2^pE^{\widetilde{Q}}\left[\sup_{0\leq{t}\leq\tau\wedge{T}}\vert{Y}^{\mathbb G}_t\vert^p\right].
\end{align*}
 Therefore, by combining this inequality with ${h}_{\tau}I_{\{\tau\leq{T}\}}=\xi I_{\{\tau\leq{T}\}}$, (\ref{Equality4MG00}), (\ref{Estimate4KGepsilon}), (\ref{Estimate4YGepsilon}), (\ref{Estimate4ZGepsilon}) and Theorem \ref{EstimatesUnderQtilde}, the proof of the theorem follows immediately.
\end{proof}

Similarly, the following theorem gives a version of Theorems \ref{EstimatesUnderQtilde1} where the left-hand-side of its estimate does not involve the probability $\widetilde Q$. 

      \begin{theorem}\label{estimates1} 
       Let ($Y^{\mathbb{G},i},Z^{\mathbb{G},i},K^{\mathbb{G},i}, M^{\mathbb{G},i}$)  be a  solution to the RBSDE (\ref{RBSDEG}) that correspond to  $(f_{i}, S_{i}, \xi^i)$, $i=1,2$ respectively. Then, there exist $C_1$ and $C_2$ that depend on $p$ only such that 
       \begin{align}\label{Estimate4P1}
&\Vert (\widetilde{\cal E})^{1/p}\delta Y^{\mathbb G} \Vert_{\mathbb{D}_{T\wedge\tau}(P,p)}^{p}+\Vert(\widetilde{\cal E}_{-})^{1/p}\delta Z^{\mathbb{G}}\Vert_{\mathbb{S}_{T\wedge\tau}(P,p)}^p+\Vert(\widetilde{\cal E}_{-})^{1/p} \is \delta {M^{\mathbb{G}}}\Vert^{p}_{{\cal {M}}^p_{T\wedge\tau}(\widetilde{P})}\\
&\leq C_1\left\{\Vert \delta \xi\Vert^{p}_{\mathbb{L}^{p}(\widetilde{Q})}+ \Vert\int_{0}^{T\wedge\tau}\vert\delta f(s)\vert ds\Vert^{p}_{\mathbb{L}^{p}(\widetilde{Q})}+\Vert  \delta S\Vert_{\mathbb{D}_{T\wedge\tau}(\widetilde{Q},p)}^p\right\} \\
&+C_2\Vert  \delta S\Vert_{\mathbb{D}_{T\wedge\tau}(\widetilde{Q},p)}^{p/2}\sqrt{\sum_{i=1}^2
\left\{\Vert\xi^{(i)}\Vert^{p}_{\mathbb{L}^{p}(\widetilde{Q})}+ \Vert\int_{0}^{T\wedge\tau}\vert{f}^{(i)}(s)\vert ds\Vert^{p}_{\mathbb{L}^{p}(\widetilde{Q})}+\Vert (S^{(i)})^+\Vert_{\mathbb{D}^{p}_{T\wedge\tau}(\widetilde{Q})}^p\right\}}.
\end{align}
\end{theorem}
\begin{proof}
Remark that by applying Lemma \ref{technicallemma1}-(a) to $Y=\delta{Y}^{\mathbb G}$ and $a=1/p$, we deduce that 
\begin{align}\label{Control4deltaYGInfinity}
E\left[\sup_{0\leq t\leq T}\widetilde{\cal E}_t\vert\delta Y^{\mathbb{G}}_{t}\vert^p\right]\leq \kappa E^{\widetilde{Q}}\left[\sup_{0\leq t\leq T}\vert\delta Y^{\mathbb{G}}_{t}\vert^p\right]\end{align}
Thus, the rest of the proof focuses on controlling the remaining terms in the left-hand-side of (\ref{Estimate4P1}). To this end, we apply  Lemma \ref{technicallemma1}-(b) to $K=\int_0^{\cdot }(\delta{Z}^{\mathbb G}_s)^2ds +[\delta{M}^{\mathbb G},\delta{M}^{\mathbb G}]$ and $a=2/p$, and get 
\begin{align}\label{Control4deltaZGInfinity}
&E\left[\left(\int_0^{T\wedge\tau}\widetilde{\cal E}_{-})^{2/p}(\delta Z^{\mathbb{G}})^2 ds+\int_0^{T\wedge\tau}(\widetilde{\cal E}_{-})^{2/p}d[\delta{M}^{\mathbb G},\delta{M}^{\mathbb G}]_s\right)^{p/2}\right]\nonumber\\
&\leq\kappa E^{\widetilde{Q}}\left [\left(\int_0^{T\wedge\tau}(\delta Z^{\mathbb{G}})^2 ds+[\delta{M}^{\mathbb G},\delta{M}^{\mathbb G}]_{T\wedge\tau}\right)^{p/2}+\sum_{0\leq{t}\leq {T\wedge\tau}}{\widetilde{G}_t}(\Delta(\delta{M}^{\mathbb G})_t)^p/2\right]
\end{align}
Then, thank to Theorem \ref{abcde} which implies that $\Delta(\delta{M}^{\mathbb G})=(h-Y^{\mathbb F}{\widetilde{\cal E}}^{-1})\Delta N^{\mathbb G}$, and by mimicking the footsteps of step 3 in the proof of Lemma \ref{technicallemma1}, we derive
\begin{align}\label{Control4deltaMGInfinity}
 E^{\widetilde{Q}}\left [\sum_{0\leq{t}\leq {T\wedge\tau}}{\widetilde{G}_t}\vert\Delta(\delta{M}^{\mathbb G})_t\vert^p\right]\leq 2^p E^{\widetilde{Q}}\left [(\vert{h}\vert^p+\vert{Y}^{\mathbb G}\vert^p)I_{\Rbrack0,\tau\Lbrack}\is D^{o,\mathbb F}_T\right]\leq 2^pE^{\widetilde{Q}}\left[\vert\xi\vert^p+ \sup_{0\leq{t}\leq{T\wedge\tau}}\vert Y^{\mathbb G}\vert^p\right].
\end{align}
Therefore, by combining (\ref{Control4deltaYGInfinity}), (\ref{Control4deltaZGInfinity}), (\ref{Control4deltaMGInfinity}) and Theorem \ref{EstimatesUnderQtilde1}, the proof of theorem follows immediately. This ends the prof of the theorem.
\end{proof}     
 Our second step, in solving (\ref{RBSDEGinfinite}), relies on the following lemma, and focuses on simultaneously letting $T$ to go to infinity and getting rid-off $\widetilde{Q}$ in the norms of the data-triplet.
     \begin{lemma}\label{ExpecationQtilde2P} Let $X$ be a non-negative and $\mathbb F$-optional process with $X_0=0$. Then the following hold.
     {\rm{(a)}} For any $T\in (0,\infty)$, we always have
     \begin{eqnarray}\label{XunderQtilde}
     E^{\widetilde Q}[X_{T\wedge\tau}]=G_0 E\left[\int_0^T X_sdV_s^{\mathbb F}+X_T{\widetilde {\cal E}}_T\right]\quad\mbox{and}\quad  E^{\widetilde Q}[X_{\tau}I_{\{\tau\leq{T}\}}]=G_0 E\left[\int_0^T X_sdV_s^{\mathbb F}\right].
     \end{eqnarray}
     {\rm{(b)}} If $X/{\cal E}(G_{-}^{-1}\is m)$ is bounded, then we get
     \begin{eqnarray}
     \lim_{T\to\infty} E^{\widetilde Q}[X_{T\wedge\tau}]=G_0\Vert X\Vert_{L^1(P\otimes V^{\mathbb F})}:=G_0 E\left[\int_0^{\infty} X_sdV_s^{\mathbb F}\right].\end{eqnarray}
     \end{lemma}
This lemma, that will be proved in Appendix \ref{Appendix4Proofs}, allows us to take the limit of expectations under $\widetilde Q$ when some conditions hold. Below, we elaborate our principal result of this subsection.

\begin{theorem}\label{EstimateInfinite} Let $p\in (1,+\infty)$ and suppose that $G>0$ and  the data-triplet $(f, S, h)$ satisfies 
\begin{eqnarray}\label{MainAssumption4InfiniteHorizon}
E\left[\int_0^{\infty}\left(\vert{h}_t\vert^p+(F_t)^p+\sup_{0\leq{u}\leq t}(S_u^+)^p\right)dV^{\mathbb F}_t\right]<+\infty,\quad\mbox{where}\quad F_t:=\int_0^t\vert f(s)\vert ds.
\end{eqnarray}
Then the following assertions hold.\\
{\rm{(a)}} The RBSDE (\ref{RBSDEGinfinite}) admits a unique solution $\left(Y^{\mathbb{G}},Z^{\mathbb{G}},K^{\mathbb{G}},M^{\mathbb{G}}\right)$. \\
{\rm{(b)}} There exists a positive constant $C$, that depends on $p$ only, such that 
\begin{eqnarray}\label{Estimate4PTinfinity}
&&E\left[\sup_{0\leq t\leq \tau}\widetilde{\cal E}_t\vert{Y}^{{\mathbb{G}}}_{t}\vert^p+\left(\int_{0}^{\tau}(\widetilde{\cal E}_{s-})^{2/p}\vert Z^{{\mathbb{G}}}_{s}\vert^{2}ds\right)^{p/2}+\left(\int_0^{\tau}({\widetilde{\cal E}_{s-}})^{1/p}dK^{\mathbb{G}}_s\right)^p+\left(\int_{0}^{\tau}(\widetilde{\cal E}_{s-})^{2/p}d[ M^{\mathbb{G}}, M^{\mathbb{G}}]_s\right)^{p/2}\right]\nonumber\\
&&\leq C E\left[\int_0^{\infty}\left\{\vert{h}_t\vert^p+(F_t)^p+\sup_{0\leq s\leq t}(S_u^+)^p\right\}dV^{\mathbb F}_t\right].\end{eqnarray}
{\rm{(c)}}   Let $\left(Y^{\mathbb{G},i},Z^{\mathbb{G},i},K^{\mathbb{G},i}, M^{\mathbb{G},i}\right)$  be a  solution to the RBSDE (\ref{RBSDEGinfinite}) corresponding to  $(f^{(i)}, S^{(i)}, h^{(i)})$,  for $i=1,2$. Then, there exist positive $C_1$ and $C_2$ that depend on $p$ only such that 
 \begin{align}\label{Estimate4P1Tinifinite}
&E\left[\sup_{0\leq t\leq \tau}\widetilde{\cal E}_t\vert\delta Y^{{\mathbb{G}}}_{t}\vert^p+\left(\int_{0}^{\tau}(\widetilde{\cal E}_{s-})^{2/p}\vert \delta Z^{{\mathbb{G}}}_{s}\vert^{2}ds+\int_{0}^{\tau}(\widetilde{\cal E}_{s-})^{2/p}d[ \delta M^{\mathbb{G}},\delta M^{\mathbb{G}}]_s\right)^{p/2}\right]\nonumber\\
&\leq C_1 E\left[\displaystyle\int_0^{\infty}\left\{\vert\delta h_t\vert^p+\vert{\delta F_t}\vert^p+\sup_{0\leq{u}\leq{t}}\vert\delta{S}_u\vert^p\right\}dV^{\mathbb F}_t\right]\nonumber\\
&+C_2\sqrt{E\left[\displaystyle\int_0^{\infty}\sup_{0\leq{u}\leq{t}}\vert\delta{S}_u\vert^p dV^{\mathbb F}_t\right]} \sqrt{\sum_{i=1}^2
 E\left[\displaystyle\int_0^{\infty}\left\{\vert{h}_t^{(i)}\vert^p+\vert{ F}_t^{(i)}\vert^p+\sup_{0\leq{u}\leq{t}}\vert(S^{(i)}_u)^+\vert^p\right\}dV^{\mathbb F}_t\right]},
\end{align}
where 
\begin{eqnarray}\label{processesDelta}
\begin{cases}\delta{Y}^{\mathbb{G}}:=Y^{\mathbb{G},1}-Y^{\mathbb{G},2},\quad\delta{Z}^{\mathbb{G}}:=Z^{\mathbb{G},1}-Z^{\mathbb{G},2},\quad\delta{K}^{\mathbb{G}}:=K^{\mathbb{G},1}-K^{\mathbb{G},2},\quad \delta{M}^{\mathbb{G}}:=M^{\mathbb{G},1}-M^{\mathbb{G},2},\label{deltaProcesses}\\
\delta{h}:=h^{(1)}-h^{(2)} ,\quad\delta{F}:=F^{(1)}-F^{(2)},\quad F^{(i)}:=\int_0^{\cdot}\vert{f}^{(i)}_s\vert{d}s,\quad \delta{S}:=S^{(1)}-S^{(2)}.\end{cases}\end{eqnarray}
\end{theorem}
The theorem states that, in the case when the horizon might be unbounded, we ``{\it discount}" somehow the solution of the RBSDE  (\ref{RBSDEGinfinite}), using the discount factor ${\widetilde{\cal E}}$,  and we estimate afterwards the resulting processes under $P$ with the data-triplet processes $(f, h, S)$ using the space $L^p(\widetilde{\Omega}, {\cal{F}}\otimes{\cal{B}}(\mathbb R^+),P\otimes V^{\mathbb F})$ and its norm instead.  This norm appeared naturally in our analysis, and it reflects the fact that $\tau$ is a random horizon that might span the whole set of fixed planning horizons. Equivalently, for the pair $(Y^{\mathbb{G}},Z^{\mathbb{G}})$ of the solution, we use the following two spaces and their norms 
given by 
\begin{eqnarray}\label{DtildeSpace}
 \begin{cases}
 \widetilde{\mathbb{D}}_{\sigma}(P,p):=\left\{Y\in {\mathbb{D}}_{\sigma}(P,p):\quad \Vert{Y}\Vert_{ \widetilde{\mathbb{D}}_{\sigma}(P,p)}:= \Vert{Y}{\widetilde{\cal E}}^{1/p}\Vert_{ {\mathbb{D}}_{\sigma}(P,p)}<+\infty\right\},\\
  \widetilde{\mathbb{S}}_{\sigma}(P,p):=\left\{Z\in {\mathbb{S}}_{\sigma}(P,p):\quad \Vert{Z}\Vert_{ \widetilde{\mathbb{S}}_{\sigma}(P,p)}:= \Vert{Z}{\widetilde{\cal E}_{-}}^{1/p}\Vert_{ {\mathbb{S}}_{\sigma}(P,p)}<+\infty\right\}.\end{cases}
 \end{eqnarray}
 For the remaining pair  $(K^{\mathbb{G}},M^{\mathbb{G}})$ of the solution we take the norm of the ``discounted" processes ${\widetilde{\cal E}_{-}}^{1/p}\is{K}^{\mathbb{G}}$  and ${\widetilde{\cal E}_{-}}^{1/p}\is{M}^{\mathbb{G}}$ instead of those of $({K}^{\mathbb{G}},{M}^{\mathbb{G}})$.\\
  Furthermore, direct It\^o calculations show that $\left(Y^{\mathbb{G}},Z^{\mathbb{G}},K^{\mathbb{G}},M^{\mathbb{G}}\right)$ is the unique solution to  (\ref{RBSDEGinfinite}) if and only if $\left(\widetilde{Y}^{\mathbb{G}},\widetilde{Z}^{\mathbb{G}},\widetilde{K}^{\mathbb{G}},\widetilde{M}^{\mathbb{G}}\right):=\left({\widetilde{\cal E}}^{1/p}Y^{\mathbb{G}},{\widetilde{\cal E}_{-}}^{1/p}Z^{\mathbb{G}},{\widetilde{\cal E}_{-}}^{1/p}\is{K}^{\mathbb{G}},{\widetilde{\cal E}_{-}}^{1/p}\is{M}^{\mathbb{G}}\right)$ is the unique solution to 
 \begin{align}\label{EquivalentRBSDE}
 \begin{cases}
 dY=-Y\left({{\widetilde{G}}\over{G}}\right)^{1/p}I_{\Rbrack0,\tau\Rbrack}dV^{(1/p)}-{\widetilde{\cal E}_{-}}^{1/p}f(t)d(t\wedge\tau)-dK-dM+ZdW^{\tau},\\
 Y_{\tau}={\widetilde{\cal E}_{\tau}}^{1/p}\xi,\displaystyle\quad Y\geq {\widetilde{\cal E}}^{1/p}S\quad\mbox{on}\quad \Lbrack0,\tau\Lbrack,\quad \int_0^{\tau}(Y_{u-}- {\widetilde{\cal E}_{u-}}^{1/p}S_{u-})dK_u=0,
 \end{cases}
 \end{align}
 where $V^{(1/p)}$ is defined in (\ref{Vepsilon}). Furthermore, under (\ref{MainAssumption4InfiniteHorizon}), this solution is an $L^p(P,\mathbb G)$-solution and there exists a positive constant $C$ that depends on $p$ only such that 
 \begin{align*}
 \Vert \widetilde{Y}^{\mathbb{G}}\Vert_{\mathbb{D}_{\tau}(P,p)}+ \Vert \widetilde{Z}^{\mathbb{G}}\Vert_{\mathbb{S}_{\tau}(P,p)}+\Vert\widetilde{K}^{\mathbb{G}}_T\Vert_{L^p(P)}+\Vert\widetilde{M}^{\mathbb{G}}\Vert_{{\cal{M}}^p(P)}\leq C\Vert{F}+\vert{h}\vert+\sup_{0\leq u\leq \cdot}S_u^+\Vert_{L^p(P\otimes{V}^{\mathbb{F}})}.
 \end{align*}
\begin{proof}[Proof of Theorem \ref{EstimateInfinite}] In virtue of Lemma \ref{ExpecationQtilde2P}, we will prove the theorem in two parts.\\
{\bf Part 1.} Here, we assume that there exists a positive constant $C\in (0,+\infty)$ such that 
\begin{eqnarray}\label{BoundednessAssumption}
\max\left(\vert{h}\vert^p, \left(\int_0^{\cdot}\vert f_s\vert ds\right)^p, \sup_{0\leq u\leq\cdot}\vert{S}_u\vert^p\right)\leq C {\cal E}(G_{-}^{-1}\is m),\end{eqnarray}
and prove that the theorem holds under this assumption. To this end, we consider the sequence  of data $(f^{(n)}, h^{(n)}, S^{(n)})$ given by 
\begin{eqnarray}\label{Sequence4[0,n]}
f^{(n)}:=fI_{\Lbrack0,n\Rbrack},\quad h^{(n)}:=hI_{\Lbrack0,n\Rbrack},\quad S^{(n)}_t:=S_{n\wedge{t}},\quad\xi^{(n)}:=h_{\tau}I_{\{\tau\leq{n}\}},\quad \forall\ n\geq 1.
\end{eqnarray}
For any $n\geq 1$, the RBSDE (\ref{RBSDEGinfinite}) has a unique solution for any horizon $T\geq n$. For any $n, m\geq 1$, we apply  Theorem \ref{estimates} to each $(f^{(n)}, h^{(n)}, S^{(n)})$ and Theorem \ref{estimates1} to the triplet
\begin{eqnarray*}
\left(\delta{f}, \delta{h},\delta{ S},\delta\xi\right):=\left(f^{(n)}-f^{(n+m)}, h^{(n)}-h^{(n+m)}, S^{(n)}-S^{(n+m)}, h_{\tau}I_{\{\tau\leq n\}}-h_{\tau}I_{\{\tau\leq n+m\}}\right),
\end{eqnarray*}
and the horizon for both theorems is $T=n+m$, and get
\begin{eqnarray}\label{Inequality4Limit}
&&E\left[\sup_{0\leq t\leq T}\widetilde{\cal E}_t\vert{Y}^{\mathbb{G},n}_{t}\vert^p+\left(\int_{0}^{T}(\widetilde{\cal E}_{s-})^{2/p}\vert Z^{\mathbb{G},n}_{s}\vert^{2}ds\right)^{p/2}+\left((\widetilde{\cal E}_{-})^{1/p}\is K^{\mathbb{G},n}_T\right)^p+((\widetilde{\cal E}_{-})^{2/p}\is[ M^{\mathbb{G},n}, M^{\mathbb{G},n}]_T)^{p/2}\right]\nonumber\\
&&\leq C E^{\widetilde{Q}}\left[\left(\int_{0}^{n\wedge\tau}\vert f(s)\vert ds\right)^p +\underset{0\leq s \leq n\wedge\tau}{\sup}(S^{+}_{s})^p+\vert\xi_n\vert^p\right],
\end{eqnarray}
and 
 \begin{align}\label{Inequa4Convergence}
&E\left[\sup_{0\leq t\leq T}\widetilde{\cal E}_t\vert{Y}^{\mathbb{G},n}_{t}-Y^{\mathbb{G},n+m}_{t}\vert^p+\left(\int_{0}^{T\wedge\tau}(\widetilde{\cal E}_{s-})^{2/p}\vert{ Z}^{\mathbb{G},n}_{s}-Z^{\mathbb{G},n+m}_{s}\vert^{2}ds\right)^{p/2}\right]\nonumber\\
&+E\left[\left(\int_{0}^{T}(\widetilde{\cal E}_{s-})^{2/p}d[ M^{\mathbb{G},n}-M^{\mathbb{G},n+m},M^{\mathbb{G},n}-M^{\mathbb{G},n+m}]_s\right)^{p/2}\right] \nonumber\nonumber\\
&\leq C_1  E^{\widetilde{Q}}\left[\vert\xi_n-\xi_{n+m}\vert^p+\left(\int_{0}^{\tau}\vert{f}_n(s)-f_{n+m}(s)\vert ds\right)^p +\sup_{0\leq t\leq \tau}\vert{S}_{t\wedge{n}}-S_{t\wedge{(n+m)}}\vert^p\right]\nonumber\\
&+C_2\sqrt{\Vert\sup_{0\leq t\leq\tau}\vert{S}_{t\wedge{n}}-S_{t\wedge{(n+m)}}\vert\Vert_{L^p(\widetilde{Q})}^p}\sqrt{\sum_{i\in\{n,n+m\}}E^{\widetilde{Q}}\left[\vert\xi^{(i)}\vert^p+\left(\int_{0}^{\tau}\vert{f}^{(i)}(s)\vert ds\right)^p +\sup_{0\leq t\leq \tau}({S}^{(i)}_t)^+)^p\right]}.
\end{align}
  The rest of this part is divided into two steps.\\
{\bf Step 1.} Here we calculate the limits, when $n$ and/or $m$ go to infinity, of the right-hand-sides of the inequalities (\ref{Inequality4Limit}) and (\ref{Inequa4Convergence}). \\
By directly applying Lemma \ref{ExpecationQtilde2P} to $\left(\int_{0}^{\cdot}\vert f(s)\vert ds\right)^p$, $\sup_{0\leq s \leq \cdot}(S^{+}_{s})^p$, and $\vert{h}\vert^p$, we deduce that 
\begin{align}\label{Limits4RightHandSide}
\begin{cases}
\lim_{n\to\infty}E^{\widetilde{Q}}\left[\left(\int_{0}^{n\wedge\tau}\vert f(s)\vert ds\right)^p\right]=E\left[\int_0^{\infty} (F_t)^p dV^{\mathbb F}_t\right],\\
\lim_{n\to\infty}E^{\widetilde{Q}}\left[\sup_{0\leq s \leq n\wedge\tau}(S^{+}_{s})^p\right]=E\left[\int_0^{\infty} \sup_{0\leq s \leq{t}}(S^{+}_{s})^pdV^{\mathbb F}_t\right],\\
\lim_{n\to\infty}E^{\widetilde{Q}}\left[\vert\xi_n\vert^p\right]=\lim_{n\to\infty}E^{\widetilde{Q}}\left[\vert{h}_{\tau}\vert^pI_{\{\tau\leq{n}\}}\right]=E\left[\int_0^{\infty} \vert{h}_t\vert^p dV^{\mathbb F}_t\right].
\end{cases}
\end{align} 
Furthermore, similar arguments allow us to derive
\begin{align}\label{LimitZero}
\begin{cases}
\displaystyle\lim_{n\to\infty}\sup_{m\geq1}E^{\widetilde{Q}}\left[\left(\int_{0}^{\tau}\vert{f}_n(s)-f_{n+m}(s)\vert {d}s\right)^p\right]=\lim_{n\to\infty}E^{\widetilde{Q}}\left[\left(\int_{n\wedge\tau}^{\tau}\vert{f}\vert {d}s\right)^p\right]=0,\\
\displaystyle\lim_{n\to\infty}\sup_{m\geq1}E^{\widetilde{Q}}\left[\sup_{0\leq t\leq \tau}\vert{S}_{t\wedge{n}}-S_{t\wedge{(n+m)}}\vert^p\right]=\lim_{n\to\infty}E^{\widetilde{Q}}\left[\sup_{n\leq{t}\leq \tau}\vert{S}_{n}-S_{t}\vert^pI_{\{\tau>n\}}\right]=0,\\
\displaystyle\lim_{n\to\infty}\sup_{m\geq1}E^{\widetilde{Q}}\left[\vert\xi_n-\xi_{n+m}\vert^p\right]=\lim{n\to\infty}E^{\widetilde{Q}}\left[\vert{h}_{\tau}\vert^pI_{\{\tau>n\}}\right]=0.\end{cases}
\end{align} 

{\bf Step 2.} This step proves assertions (a), (b) and (c) of the theorem under the assumption (\ref{BoundednessAssumption}).\\
Thus, by combining  (\ref{LimitZero}) and (\ref{Inequa4Convergence}), we deduce that the sequence $(Y^{\mathbb{G},n}, Z^{\mathbb{G},n}, K^{\mathbb G,n},M^{\mathbb{G},n})$ is Cauchy sequence in norm, and hence it converges to $(Y^{\mathbb{G}}, Z^{\mathbb{G}}, K^{\mathbb G},M^{\mathbb{G}})$ in norm and almost surely for a subsequence.  It is clear then that $(Y^{\mathbb{G}}, Z^{\mathbb{G}}, K^{\mathbb G},M^{\mathbb{G}})$ is a solution to (\ref{RBSDEGinfinite}), and hence this equation admits a solution. The uniqueness of this solution is a direct consequence of assertion (c), and hence assertion (a) follows immediately as soon as we prove assertion (c). Besides, by taking the limit in (\ref{Inequality4Limit}), and using  Fatou and (\ref{Limits4RightHandSide}), assertion (b) follows immediately. Thus, the rest of this step deals with assertion (c). To this end, we consider two triplets $(f^{(i)}, S^{(i)},h^{(i)})$, $i=1,2$, that satisfy the boundedness assumption (\ref{BoundednessAssumption}). Then for each $i=1,2$ we associate a sequence $(f^{(i, n)}, S^{(i,n)}, h^{(i,n)})$ as in (\ref{Sequence4[0,n]}). Thus, on the one hand, we apply Theorem \ref{estimates1} for each  $(\delta\xi^{(i)}, \delta f^{(i)}, \delta S^{(i)}):=\left(f_{i, n}-f_{i, n+m}, S_{i,n}-S_{i,n+m}, \xi^{i,n}-\xi^{i,n+m}\right)$ and get similar inequalities as (\ref{Inequa4Convergence}) for each $i=1,2$ and deduce afterwards that the sequence $(Y^{\mathbb{G},(i,n)}, Z^{\mathbb{G},(i,n)}, K^{\mathbb G,(i,n)},M^{\mathbb{G},(i,n)})$ converses in norm and almost surely for a subsequence to the solution $(Y^{\mathbb{G},(i)}, Z^{\mathbb{G},(i)}, K^{\mathbb G,(i)},M^{\mathbb{G},(i)})$. On the other hand, for each $n\geq 1$, we apply Theorem \ref{estimates1} for  
\begin{align*}
&\delta f^{(n)}:=f^{(1, n)}- f^{(2, n)},\quad\delta S^{(n)}:=S^{(1,n)}-S^{(2,n)},\quad\delta\xi^{(n)}):= \xi^{(1,n)}-\xi^{(2,n)},\\
& \mbox{and}\\
&\begin{cases}
\delta{Y}^{\mathbb{G},(n)}:=Y^{\mathbb{G},(1,n)}-Y^{\mathbb{G},(2,n)} \quad ,\delta{ Z}^{\mathbb{G},(n)}:=Z^{\mathbb{G},(1,n)}-Z^{\mathbb{G},(2,n)},\\
 \delta{K}^{\mathbb G,(n)}:=K^{\mathbb G,(1,n)}-K^{\mathbb G,(2,n)},\quad \delta{M}^{\mathbb{G},(n)}:=M^{\mathbb{G},(1,n)})-M^{\mathbb{G},(2,n)},
\end{cases}\end{align*}
 and get 
 \begin{align}\label{Convergence4Differences}
 &E\left[\sup_{0\leq t\leq T}\widetilde{\cal E}_t\vert\delta{Y}^{\mathbb{G},(n)}_{t}\vert^p+\left(\int_{0}^{T\wedge\tau}(\widetilde{\cal E}_{s-})^{2/p}\vert\delta{ Z}^{\mathbb{G},(n)}_{s}\vert^{2}ds+\int_{0}^{T}(\widetilde{\cal E}_{s-})^{2/p}d[ \delta{M}^{\mathbb{G},(n)},\delta{M}^{\mathbb{G},(n)}]_s\right)^{p/2}\right]\nonumber\\
&\leq C_1  E^{\widetilde{Q}}\left[\vert\delta\xi^{(n)}\vert^p+\left(\int_{0}^{\tau}\vert\delta f^{(n)}_s\vert ds\right)^p +\sup_{0\leq t\leq \tau}\vert \delta S^{(n)}_t\vert^p\right]\nonumber\\
&+C_2\sqrt{\Vert\sup_{0\leq t\leq\tau}\vert \delta S^{(n)}_t\vert\Vert_{L^p(\widetilde{Q})}^p}\sqrt{\sum_{i=1}^2E^{\widetilde{Q}}\left[\vert\xi^{(i,n)}\vert^p+\left(\int_{0}^{\tau}\vert{f}^{(i,n)}(s)\vert ds\right)^p +\sup_{0\leq t\leq \tau}({S}^{(i,n)}_t)^+)^p\right]}.
\end{align}
Similarly, as in the proof of (\ref{Limits4RightHandSide}), we use Lemma \ref{ExpecationQtilde2P}  and the boundedness assumption  (\ref{BoundednessAssumption}) that each  triplet $(f^{(i)}, S^{(i)},h^{(i)})$ satisfies, $i=1,2$, and get
\begin{align}\label{Limits4Differences}
\begin{cases}
\displaystyle\lim_{n\to\infty}E^{\widetilde{Q}}\left[\left(\int_{0}^{n\wedge\tau}\vert \delta{f}^{(n)}_s)\vert ds\right)^p\right]=E\left[\int_0^{\infty} \vert\delta{F}_t\vert^p dV^{\mathbb F}_t\right],\ \lim_{n\to\infty}E^{\widetilde{Q}}\left[\vert\delta\xi^{(n)}\vert^p\right]=E\left[\int_0^{\infty} \vert\delta{h}_t\vert^p dV^{\mathbb F}_t\right],\\
\displaystyle\lim_{n\to\infty}E^{\widetilde{Q}}\left[\sup_{0\leq s \leq{n}\wedge\tau}((S^{(i)})^{+}_{s})^p\right]=E\left[\int_0^{\infty} \sup_{0\leq s \leq{t}}((S^{(i)})^{+}_{s})^pdV^{\mathbb F}_t\right],\quad i=1,2.\end{cases}
\end{align}
 Thus, by taking the limit in (\ref{Convergence4Differences}), using Fatou's lemma for its left-hand-side term, and using (\ref{Limits4Differences}) for its right-hand-side term, we conclude that assertion (c) holds. This ends the first part.\\
{\bf Part 2.} This step proves the theorem without any assumption. Hence, we consider the following sequence of stopping times
\begin{eqnarray*}
T_n:=\inf\left\{t\geq 0\ :\quad {{\vert{S}_t\vert^p}\over{{\cal E}_t(G_{-}^{-1}\is m)}} >n\quad\mbox{or}\quad  {{(\int_0^t\vert f(s)\vert ds)^p}\over{{\cal E}_t(G_{-}^{-1}\is m)}} >n\right\},\end{eqnarray*}
and the sequences
\begin{eqnarray}\label{Consutrction4DataSequence}
h^{(n)}:=h I_{\{\vert{h}\vert^p\leq n{\cal E}(G_{-}^{-1}\is m)\}}I_{\Lbrack0, T_n\Lbrack},\quad f^{(n)}:=fI_{\Lbrack0, T_n\Rbrack},\quad S^{(n)}:=S I_{\Lbrack0, T_n\Lbrack}.
\end{eqnarray}
Thus, it is clear that  all triplets $(f^{(n)}, h^{(n)}, S^{(n)})$ satisfy  (\ref{BoundednessAssumption}), for any $n\geq 1$.  Thus, thanks to the first part, we deduce the existence of unique solution to (\ref{RBSDEGinfinite}), denoted by $(Y^{\mathbb{G},(n)}, Z^{\mathbb{G},(n)}, K^{\mathbb G,(n)},M^{\mathbb{G},(n)})$, associated to the data $(f^{(n)}, h^{(n)}, S^{(n)})$ and  satisfying\\
\begin{align}\label{Estimate4PTinfinityproof}
&E\left[\sup_{0\leq t\leq \tau}\widetilde{\cal E}_t\vert{Y}^{{\mathbb{G}},n}_{t}\vert^p+\left(\int_{0}^{\tau}(\widetilde{\cal E}_{s-})^{2/p}\vert Z^{{\mathbb{G}},n}_{s}\vert^{2}ds+\int_{0}^{\tau}(\widetilde{\cal E}_{s-})^{2/p}d[ M^{\mathbb{G},n}, M^{\mathbb{G},n}]_s\right)^{p/2}+\left(\int_0^{\tau}({\widetilde{\cal E}_{s-}})^{1/p}dK^{\mathbb{G},n}_s\right)^p\right]\nonumber\\
&\leq C E\left[\int_0^{\infty}\left\{\vert{h}^{(n)}_t\vert^p+({ F}_t^{(n)})^p+\sup_{0\leq s\leq t}\vert(S^{(n)}_u)^+\vert^p\right\}dV^{\mathbb F}_t\right],\end{align}
 and for any $n\geq 1$ and $m\geq 1$
 \begin{align}\label{Estimate4P1Tinifiniteproof}
&E\left[\sup_{0\leq t\leq \tau}\widetilde{\cal E}_t\vert{Y}^{{\mathbb{G},n}}_{t}-{Y}^{{\mathbb{G},n+m}}_{t}\vert^p+\left(\int_{0}^{\tau}(\widetilde{\cal E}_{s-})^{2/p}\vert{ Z}^{{\mathbb{G}},n}_{s}-Z^{{\mathbb{G}},n+m}_{s}\vert^{2}ds\right)^{p/2}\right]\nonumber\\
&+E\left[\left(\int_{0}^{\tau}(\widetilde{\cal E}_{s-})^{2/p}d[ M^{\mathbb{G},n}-M^{\mathbb{G},n+m},M^{\mathbb{G},n}-M^{\mathbb{G},n+m}]_s\right)^{p/2}\right]\nonumber\\
&\leq C_1 E\left[\displaystyle\int_0^{\infty}\left\{\vert {h}_t^{(n)}-{h}_t^{(n+m)} \vert^p+\vert{ F}_t^{(n)}-{ F}_t^{(n+m)}\vert^p+\sup_{0\leq{u}\leq{t}}\vert{S}_u^{(n)}-{S}_u^{(n+m)}\vert^p\right\}dV^{\mathbb F}_t\right]\nonumber\\
&+C_2\sqrt{E\left[\displaystyle\int_0^{\infty}\sup_{0\leq{u}\leq{t}}\vert{S}_u^{(n)}-{S}_u^{(n+m)}\vert^p dV^{\mathbb F}_t\right]} \sqrt{\sum_{i\in\{n,n+m\}}
 E\left[\displaystyle\int_0^{\infty}\left\{\vert{h}_t^{(i)}\vert^p+\vert{ F}_t^{(i)}\vert^p+\sup_{0\leq{u}\leq{t}}\vert(S^{(i)}_u)^+\vert^p\right\}dV^{\mathbb F}_t\right]}.
\end{align}
This latter inequality is obtained from (\ref{Estimate4P1Tinifinite}) by considering 
\begin{eqnarray*}
\begin{cases}\delta{Y}^{\mathbb{G}}:=Y^{\mathbb{G},n}-Y^{\mathbb{G},n+m},\ \delta{Z}^{\mathbb{G}}:=Z^{\mathbb{G},n}-Z^{\mathbb{G},n+m},\ \delta{K}^{\mathbb{G}}:=K^{\mathbb{G},n}-K^{\mathbb{G},n+m},\ \delta{M}^{\mathbb{G}}:=M^{\mathbb{G},n}-M^{\mathbb{G},n+m},\label{deltaProcesses}\\
\delta{h}:=h^{(n)}-h^{(n+m)} ,\quad\delta{F}:=F^{(n)}-F^{(n+m)},\quad F^{(i)}:=\int_0^{\cdot}\vert{f}^{(i)}_s\vert{d}s,\quad \delta{S}:=S^{(n)}-S^{(n+m)}.\end{cases}\end{eqnarray*}
Thus, in virtue of (\ref{MainAssumption4InfiniteHorizon}) and the dominated convergence theorem, we derive 
\begin{align*}
&\lim_{n\to+\infty}\sup_{m\geq 1}E\left[\displaystyle\int_0^{\infty}\left\{\vert {h}_t^{(n)}-{h}_t^{(n+m)} \vert^p+\vert{ F}_t^{(n)}-{ F}_t^{(n+m)}\vert^p+\sup_{0\leq{u}\leq{t}}\vert{S}_u^{(n)}-{S}_u^{(n+m)}\vert^p\right\}dV^{\mathbb F}_t\right]\\
&\leq\lim_{n\to+\infty}{E}\left[\displaystyle\int_{T_n}^{\infty}\left\{\vert {h}_t\vert^p+\vert{ F}_t\vert^p+\sup_{0\leq{u}\leq{t}}\vert{S}_u\vert^p\right\}dV^{\mathbb F}_t\right]=0.
\end{align*}
A combination of this with (\ref{Estimate4P1Tinifiniteproof}) proves that the sequence $(Y^{\mathbb{G},(n)}, Z^{\mathbb{G},(n)}, K^{\mathbb G,(n)},M^{\mathbb{G},(n)})$ is a Cauchy sequence in norm, and hence it converges  to $(Y^{\mathbb{G}}, Z^{\mathbb{G}}, K^{\mathbb G},M^{\mathbb{G}})$ in norm and almost surely for a subsequence. Furthermore, $(Y^{\mathbb{G}}, Z^{\mathbb{G}}, K^{\mathbb G},M^{\mathbb{G}})$ clearly satisfies (\ref{RBSDEGinfinite}), and due to Fatou's lemma and (\ref{Estimate4PTinfinityproof}) we conclude that (\ref{Estimate4PTinfinity}) holds. To prove that (\ref{Estimate4P1Tinifinite})  holds, we repeat this analysis for the pair of data $(f^{(i)}, S^{(i)},h^{(i)}, \xi^{(i)})$, $i=1,2$, and obtain two sequences  $(Y^{\mathbb{G},(n,i)}, Z^{\mathbb{G},(n,i)}, K^{\mathbb G,(n,i)},M^{\mathbb{G},(n,i)})_{n\geq 1}$  ($i=1,2$) solution to  (\ref{RBSDEGinfinite}) corresponding to the data $(f^{(n,i)}, h^{(n,i)}, S^{(n,i)})$ that is constructed from $(f^{(i)}, h^{(i)}, S^{(i)})$ via (\ref{Consutrction4DataSequence}). Furthermore, each $(Y^{\mathbb{G},(n,i)}, Z^{\mathbb{G},(n,i)}, K^{\mathbb G,(n,i)},M^{\mathbb{G},(n,i)})_{n\geq1}$ converges (in norm and almost surely for a subsequence) to  $(Y^{\mathbb{G},i}, Z^{\mathbb{G},i}, K^{\mathbb G,i},M^{\mathbb{G},i})$ that is solution to  (\ref{RBSDEGinfinite}) corresponding to $(f^{(i)}, S^{(i)},h^{(i)}, \xi^{(i)})$, for each $i=1,2$. Thus, by applying (\ref{Estimate4P1Tinifinite})  to $(\delta{f}, \delta{h}, \delta{S}):=(f^{(n,1)}-f^{(n,2)}, h^{(n,1)}-h^{(n,2)}, S^{(n,1)}-S^{(n,2)})$, and taking the limits on both sides, we easily deduce that (\ref{Estimate4P1Tinifinite}) holds for the general case.  This ends the second part, and completes the proof of theorem.\end{proof}
\subsection{Relationship to RBSDE under $\mathbb F$}
In this subsection, we establish the RBSDE under $\mathbb F$ that is directly related to  (\ref{RBSDEGinfinite}).
   
\begin{theorem}\label{Relationship4InfiniteBSDE} Suppose that $G>0$ and $(f, S, h)$ satisfies (\ref{MainAssumption4InfiniteHorizon}) and
\begin{eqnarray}\label{MainAssumption4InfiniteHorizonBIS}
E\left[\left(F_{\infty}{\widetilde{\cal E}}_{\infty}\right)^p\right]<+\infty,\quad\mbox{where}\quad {\widetilde{\cal E}}:={\cal E}(-{\widetilde G}^{-1}\is D^{o,\mathbb{F}})\quad\mbox{and}\quad F_{\infty}:=\int_0^{\infty}\vert{f}_s\vert ds .\end{eqnarray}
Consider the pair $(f^{\mathbb{F}},S^{\mathbb{F}})$ given by (\ref{ProcessVFandXiF}).  Then the following assertions hold.\\
{\rm{(a)}} The following RBSDE, under $\mathbb F$, generated by the triplet $ \left(f^{\mathbb{F}},S^{\mathbb{F}},h\right)$
\begin{eqnarray}\label{RBSDEFinfinite}
\begin{cases}
Y_{t}= \displaystyle\int_{t}^{\infty}f^{\mathbb{F}}(s)ds+\int_{t}^{\infty}h_{s}dV^{\mathbb{F}}_{s}+K_{\infty}-K_{t}-\int_{t}^{\infty}Z_{s}dW_{s},\\
Y_{t}\geq S_{t}^{\mathbb{F}},\quad
 \displaystyle{E}\left[\int_{0}^{\infty}(Y_{t-}-S_{t-}^{\mathbb{F}})dK_{t}\right]=0 ,
\end{cases}
\end{eqnarray}
has a unique  $L^p(P,\mathbb F)$-solution $(Y^{\mathbb F},  Z^{\mathbb F}, K^{\mathbb F})$ .\\
{\rm{(b)}} The unique solution to (\ref{RBSDEGinfinite}), that we denote by $\left(Y^{\mathbb{G}},Z^{\mathbb{G}},K^{\mathbb{G}},M^{\mathbb{G}}\right)$, satisfies
\begin{eqnarray}\label{firstrelationInfnite}
   Y^{\mathbb{G}}= \frac{Y^{\mathbb{F}}}{{\widetilde{\cal E}}}I_{\Lbrack0,\tau\Lbrack}+\xi{I}_{\Lbrack\tau,+\infty\Lbrack},\ 
  Z^{\mathbb{G}}=\frac{Z^{\mathbb{F}}}{{\widetilde{\cal E}}_{-}} I_{\Rbrack0,\tau\Rbrack},\ 
   K^{\mathbb{G}}=\frac{1}{{\widetilde{\cal E}}_{-}}\is (K ^{\mathbb{F}})^{\tau},\ \mbox{and}\ 
      M^{\mathbb{G}}=\left(h-\frac{Y^{\mathbb{F}}}{\widetilde{\cal E}}\right)\is N^{\mathbb{G}}.\label{secondrelationInfinite}
       \end{eqnarray}
\end{theorem}
\begin{proof} On the one hand, remark that, due to the assumptions  (\ref{MainAssumption4InfiniteHorizon})  and (\ref{MainAssumption4InfiniteHorizonBIS}), both random variables $\int_0^{\infty}\vert f^{\mathbb F}_s\vert ds$ and $\int_0^{\infty}\vert{h}_s\vert dV^{\mathbb F}_s$ belong to $L^p(P)$. In fact, this fact follows from the following two inequalities 
\begin{align*}
\int_0^{\infty}\vert f^{\mathbb F}_s\vert ds={\widetilde{\cal E}}_{\infty}F_{\infty}+\int_0^{\infty} F_s dV^{\mathbb F}_s,\quad\mbox{and}\quad \left(\int_0^{\infty}\vert{h}_s\vert dV^{\mathbb F}_s\right)^p\leq (V^{\mathbb F}_{\infty})^{p-1}\int_0^{\infty}\vert{h}_s\vert^p dV^{\mathbb F}_s\leq\int_0^{\infty}\vert{h}_s\vert^p dV^{\mathbb F}_s .
\end{align*}
On the other hand, similar arguments as in the proof of Theorem \ref{abcde}, one can prove that any solution to (\ref{RBSDEFinfinite}) $(Y,Z,K)$ satisfies
\begin{eqnarray*}
Y_t=\mbox{ess}\sup_{\sigma\in{\cal T}_t^{\infty}(\mathbb F)}E\left[\int_t^{\sigma} f^{\mathbb F}_s ds+\int_t^{\sigma}h_s dV^{\mathbb F}_s+S^{\mathbb F}_{\sigma}I_{\{\sigma<+\infty\}}\ \big|{\cal F}t\right]=:\overline{Y}.   \end{eqnarray*}
Furthermore, due to the Snell envelop theory, see \cite{ElqarouiBSDE} for details  or see also the proof of Theorem \ref{abcde}, the Doob-Meyer decomposition of the supermartingale $\overline{Y}+\int_0^{\cdot}h_s dV^{\mathbb F}_s+\int_0^{\cdot} f^{\mathbb F}_s ds=\overline{M}-\overline{K}$ gives the solution triplet to  (\ref{RBSDEFinfinite}) as $(\overline{Y},\overline{M},\overline{K})$. This proves that  (\ref{RBSDEFinfinite})  has a unique solution, and the proof of assertion (a) is complete.\\
 The rest of the proof deals with assertion (b). Remark that, in virtue of Theorem \ref{EstimateInfinite},  the RBSDE (\ref{RBSDEGinfinite}) has at most one solution. Therefore, we will prove that the triplet $(\widehat{Y}, \widehat{Z}, \widehat{K},\widehat{M})$ given by 
\begin{align*}
  \widehat{Y}:= \frac{Y^{\mathbb{F}}}{{\widetilde{\cal E}}}I_{\Lbrack0,\tau\Lbrack}+\xi{I}_{\Lbrack\tau,+\infty\Lbrack},\ 
\widehat{Z}:=\frac{Z^{\mathbb{F}}}{{\widetilde{\cal E}}_{-}} I_{\Rbrack0,\tau\Rbrack},\ 
  \widehat{K}:=\frac{1}{{\widetilde{\cal E}}_{-}}\is (K ^{\mathbb{F}})^{\tau},\ \mbox{and}\ 
   \widehat{M}:=\left(h-\frac{Y^{\mathbb{F}}}{\widetilde{\cal E}}\right)\is N^{\mathbb{G}},  \end{align*}
   is a solution to  (\ref{RBSDEGinfinite}). To prove this fact, we put 
$\widehat{\Gamma}:={Y}^{\mathbb F}/{\widetilde{\cal E}},$ and on the one hand we remark that  
\begin{eqnarray}\label{YGGammainfinite}
\widehat{Y} =\widehat{\Gamma}I_{\Lbrack0,\tau\Lbrack}+  h_{\tau} I_{\Lbrack\tau,+\infty\Lbrack}=\widehat{\Gamma}^{\tau} +(h-\widetilde{\Gamma})\is D.\end{eqnarray}
On the other hand, by combining It\^o applied to $\widehat{\Gamma}$, (\ref{RBSDEFinfinite}) that the triplet $(Y^{\mathbb F}, Z^{\mathbb{F}}, K^{\mathbb{F}})$ satisfies, ${\widetilde{\cal E}}^{-1}={\cal E}(G^{-1}\is D^{o,\mathbb F})$, ${\widetilde{\cal E}}={\widetilde{\cal E}}_{-}G/{\widetilde{G}}$, and $dV^{\mathbb F}={\widetilde{\cal E}}_{-}{\widetilde{G}}^{-1}dD^{o,\mathbb F}$, we derive 
\begin{eqnarray*}
d\widehat{\Gamma}&&=Y^{\mathbb F}d{\widetilde{\cal E}}^{-1}+{1\over{{\widetilde{\cal E}}_{-}}}dY^{\mathbb F}={{\widehat{\Gamma}}\over{\widetilde G}}dD^{o,\mathbb F}+{1\over{{\widetilde{\cal E}}_{-}}}d{Y}^{\mathbb F}={{\widehat{\Gamma}}\over{\widetilde G}}dD^{o,\mathbb F}-{{f^{\mathbb{F}}}\over{{\widetilde{\cal E}}_{-}}}ds-{{h}\over{{\widetilde{\cal E}}_{-}}}dV^{\mathbb{F}}-{{1}\over{{\widetilde{\cal E}}_{-}}}dK^{\mathbb{F}}+{{Z^{\mathbb{F}}}\over{{\widetilde{\cal E}}_{-}}}dW\nonumber\\\
&&={{\widehat{\Gamma}-h}\over{\widetilde G}}dD^{o,\mathbb F}-fds-{{1}\over{{\widetilde{\cal E}}_{-}}}dK^{\mathbb{F}}+{{Z^{\mathbb{F}}}\over{{\widetilde{\cal E}}_{-}}}dW.\label{equation4Gamma}
\end{eqnarray*}
Thus, by  stopping $\widehat{\Gamma}$ and inserting the above equality in (\ref{YGGammainfinite}) and arranging terms we get 
\begin{eqnarray}\label{SDE4YG}
d\widehat{Y} =-f(t)d(t\wedge\tau)-d\widehat{K}+d \widehat{M}+{\widehat{Z}}dW^{\tau},\quad\mbox{and}\quad \widehat{Y}_{\tau}=\xi.
\end{eqnarray} 
This proves that $(\widehat{Y}, \widehat{Z}, \widehat{K},\widehat{M})$ satisfies the first equation in (\ref{RBSDEGinfinite}). Furthermore, it is clear that   $
{Y}^{\mathbb{F}}_{t}\geq S_{t}^{\mathbb{F}}$ implies  the second condition in (\ref{RBSDEGinfinite}). To prove the Skorokhod condition (the last condition in (\ref{RBSDEGinfinite})), we combine the Skorokhod condition for the triplet $({Y}^{\mathbb F}, {Z}^{\mathbb F}, {K}^{\mathbb F})$, the fact that $\widehat{Y}_{-}\geq S_{-}$ on $\Rbrack0,\tau\Rbrack$, and 
  \begin{eqnarray*}
0\leq \int_{0}^{\tau}({\widehat{Y}}_{t-}-S_{t-})d\widehat{K}_t=\int_{0}^{\tau}({Y}^{\mathbb F}_{t-}-S^{\mathbb F}_{t-}){\widetilde{\cal E}}_{t-}^{-2}d{K}^{\mathbb F}_t\leq \int_{0}^{\infty}({Y}^{\mathbb F}_{t-}-S^{\mathbb F}_{t-}){\widetilde{\cal E}}_{t-}^{-2}dK^{\mathbb F}_t=0,\quad P\mbox{-a.s..}
 \end{eqnarray*}
This ends the second part, and the proof of theorem is complete.
\end{proof}
  \begin{remark} It is clear that, in general, the existence of an $L^p(P)$-solution to (\ref{RBSDEFinfinite}) requires stronger assumptions than the existence of $L^p(P)$-solution to (\ref{RBSDEGinfinite}).
  \end{remark}
  We end this section by elaborating the BSDE version of this section.
   \begin{theorem}\label{LinearBSDEcase}
   Suppose $G>0$ and consider a pair $(f, h)$ of $\mathbb F$-optional processes satisfying
   \begin{eqnarray}\label{MainAssumption4InfiniteHorizonBSDE}
   E\left[\left(F_{\infty}{\widetilde{\cal E}_{\infty}}\right)^p+\int_0^{\infty}\left(\vert{h}_t\vert^p+(F_t)^p\right)dV^{\mathbb F}_t\right]<+\infty,\quad\mbox{where}\quad F_t:=\int_0^t\vert f(s)\vert ds.
   \end{eqnarray}
 If $f^{\mathbb{F}}$ and  ${\widetilde{\cal E}}$ denote the processes defined in (\ref{ProcessVFandXiF}), then the following assertions hold.\\
{\rm{(a)}} The following BSDE under $\mathbb F$ 
\begin{eqnarray}\label{BSDEFinfinite}
Y_{t}= \displaystyle\int_{t}^{\infty}f^{\mathbb{F}}(s)ds+\int_{t}^{\infty}h_{s}dV^{\mathbb{F}}_{s}-\int_{t}^{\infty}Z_{s}dW_{s},\end{eqnarray}
has a unique $L^p(\mathbb{F}, P)$-solution $(Y^{\mathbb F},  Z^{\mathbb F})$.\\
{\rm{(b)}} The following BSDE
\begin{eqnarray}\label{BSDEGinfinite}
dY=-f(t)d(t\wedge\tau)-dM+ZdW^{\tau},\quad Y_{\tau}=\xi=h_{\tau},
\end{eqnarray}
has a unique solution, denoted by $\left(Y^{\mathbb{G}},Z^{\mathbb{G}},M^{\mathbb{G}}\right)$, satisfying 
\begin{eqnarray}\label{firstrelationInfnite}
   Y^{\mathbb{G}}= \frac{Y^{\mathbb{F}}}{{\widetilde{\cal E}}}I_{\Lbrack0,\tau\Lbrack}+\xi{I}_{\Lbrack\tau,+\infty\Lbrack},\ 
  Z^{\mathbb{G}}=\frac{Z^{\mathbb{F}}}{{\widetilde{\cal E}}_{-}} I_{\Rbrack0,\tau\Rbrack},\ \mbox{and}\ 
      M^{\mathbb{G}}=\left(h-\frac{Y^{\mathbb{F}}}{\widetilde{\cal E}}\right)\is N^{\mathbb{G}}.\label{secondrelationInfinite}
       \end{eqnarray}
{\rm{(c)}} Let $(f^{(i)},h^{(i)})$, $i=1,2,$ be two pairs satisfying (\ref{MainAssumption4InfiniteHorizonBSDE}), and let any $p\in (1,+\infty)$. Then there exists a positive constant $C$ that depends on $p$ only such that 
\begin{align}\label{Estimate4P1TinifiniteBSDE}
&E\left[\sup_{0\leq t\leq \tau}\widetilde{\cal E}_t\vert\delta Y^{{\mathbb{G}}}_{t}\vert^p+\left(\int_{0}^{\tau}(\widetilde{\cal E}_{s-})^{2/p}\vert \delta Z^{{\mathbb{G}}}_{s}\vert^{2}ds+\int_{0}^{\tau}(\widetilde{\cal E}_{s-})^{2/p}d[ \delta M^{\mathbb{G}},\delta M^{\mathbb{G}}]_s\right)^{p/2}\right]\nonumber\\
&\leq C  E\left[\displaystyle\int_0^{\infty}\left(\vert\delta h_t\vert^p+\vert{\delta F_t}\vert^p\right)dV^{\mathbb F}_t\right],
\end{align}
where $(Y^{\mathbb{G},i},Z^{\mathbb{G},i},M^{\mathbb{G},i})$ is the solution to (\ref{BSDEGinfinite}) associated to $(f^{(i)},h^{(i)})$, for $i=1,2$, and 
\begin{eqnarray*}
&&\delta{Y}^{\mathbb{G}}:=Y^{\mathbb{G},1}-Y^{\mathbb{G},2},\quad\delta{Z}^{\mathbb{G}}:=Z^{\mathbb{G},1}-Z^{\mathbb{G},2},\quad \delta{M}^{\mathbb{G}}:=M^{\mathbb{G},1}-M^{\mathbb{G},2},\label{deltaProcesses}\\
&&\delta{h}:=h^{(1)}-h^{(2)} ,\quad\delta{F}:=\int_0^{\cdot}\vert{f}^{(1)}(s)-f^{(2)}(s)\vert{ds}.\end{eqnarray*}
   \end{theorem}
\begin{proof} Remark that a BSDE is an RBSDE with $S\equiv -\infty$ and the nondecreasing process part of its quadruplet solution is null, i.e. $K\equiv 0$. Thus, by keeping these in mind, we conclude the following.
\begin{enumerate}
\item The condition (\ref{MainAssumption4InfiniteHorizon}) take the form of (\ref{MainAssumption4InfiniteHorizonBSDE}). 
\item Assertion (a) is a particular case of assertion (a) of Theorem \ref{Relationship4InfiniteBSDE}, while assertion (b) is direct consequence of a combination of assertion (a) of Theorem \ref{RBSDEFinfinite} and assertion (b) of Theorem \ref{Relationship4InfiniteBSDE}.
\item Assertions (c) and (d) follow from Theorem \ref{EstimateInfinite}.
\end{enumerate}
This ends the proof of theorem.
\end{proof}
\section{Stopped general RBSDEs: The case of bounded horizon}
In this section, we address the RBSDE with general generator satisfying the following assumption
\begin{equation}\label{LipschitzAssumption}
\exists\ C_{Lip}> 0,\quad \vert f(t,y,z)-f(t,y^{'},z^{'})\vert\leq C_{Lip}(\vert y-y^{'}\vert+\Vert z-z ^{'}\Vert),\quad \mbox{for all}\ y,y^{'}\in\mathbb{R},\ z,z^{'}\in\mathbb{R}^{d}.
\end{equation}
In this section, we are interested in the following RBSDE, 
\begin{eqnarray}\label{nonLinear}
\begin{cases}
dY_{t}=-f(t,Y_{t},Z_{t})d(t\wedge\tau)-d(K_{t\wedge\tau}+M_{t\wedge\tau})+Z_{t}dW_{t}^{\tau},\quad Y_{\tau}=\xi=Y_{T},\\
\\
Y_{t}\geq S_{t},\quad 0\leq t<  T\wedge\tau,\quad\rm{and}\quad E\left[\displaystyle\int_{0}^{T\wedge\tau}(Y_{t-}-S_{t-})dK_t\right]=0,
\end{cases}
\end{eqnarray}
where $(\xi, S, f)$ is such that $S$ is an $\mathbb F$-adapted and RCLL process, $f(t,y,z)$ is a $\mbox{Prog}(\mathbb F)\times {\cal B}(\mathbb R)\times {\cal B}(\mathbb R)$-measurable functional and $\xi\in L^2({\cal G}_{T\wedge\tau})$. 

\subsection{Estimate inequalities for the solution}
This subsection derives a number of norm-estimates for the solution of the RBSDEs when this exists. These inequalities play important role in the proof of the existence of uniqueness of the solution of this RBSDE on the one hand. On the other hand, the role of these estimates in studying the stability of RBSDEs is without reproach.  
\begin{lemma}\label{Estimation4DeltaY} Let $T\in (0,+\infty)$. Then the following assertions hold.\\
{\rm{(a)}}  If ($Y^{\mathbb{G}},Z^{\mathbb{G}},K^{\mathbb{G}}, M^{\mathbb{G}}$)  is a  solution to the RBSDE (\ref{nonLinear}) that corresponds to  $(f, S, \xi)$, then \begin{eqnarray}\label{YGessSup}
Y^{\mathbb G}_t=\rm{ess}\sup_{\theta\in {\cal T}_{t\wedge\tau}^{T\wedge\tau}(\mathbb G)} E^{\widetilde Q}\left[\int_{t\wedge\tau}^{\theta} f(s, Y^{\mathbb G}_s, Z^{\mathbb G}_s)ds +S_{\theta}I_{\{\theta<T\wedge\tau\}}+\xi I _{\{\theta=T\wedge\tau\}}\ \Bigg|\ {\cal G}_t\right].\end{eqnarray}
{\rm{(b)}} If ($Y_{i}^{\mathbb{G}},Z_{i}^{\mathbb{G}},K_{i}^{\mathbb{G}}, M_{i}^{\mathbb{G}}$)  is a  solution to the RBSDE (\ref{nonLinear}) that corresponds to  $(f^{(i)}, S^{(i)}, \xi^{(i)})$, $i=1,2$, then for any $\alpha>0$ the following holds
\begin{eqnarray}\label{inequa4deltaYG2bis}
\exp\left({{\alpha(t\wedge\tau)}\over{2}}\right)\vert \delta Y^{\mathbb G}_t\vert&&\leq 
E^{\widetilde Q}\left[\sup_{0<s\leq T\wedge\tau}e^{\alpha s/2}\vert \delta S_u\vert +e^{\alpha(T\wedge\tau)/2}\vert \delta\xi\vert+{{C_{Lip}}\over{\sqrt{\alpha}}} \sqrt{\int_0^{T\wedge\tau} e^{\alpha s}(\delta Z^{\mathbb G}_s)^2 ds}\ \Bigg|\ {\cal G}_t\right]\nonumber\\
&&+{1\over{\sqrt{\alpha}}}E^{\widetilde Q}\left[ C_{Lip}\sqrt{\int_0^{T\wedge\tau} e^{\alpha s}\vert\delta Y^{\mathbb G}_s\vert^2 ds}  +\sqrt{\int_0^{T\wedge\tau} e^{\alpha s}\vert\delta f_s\vert^2 ds}   \ \bigg|\ {\cal G}_t\right].
\end{eqnarray}
{\rm{(c)}} If ($Y^{\mathbb{G}},Z^{\mathbb{G}},K^{\mathbb{G}}, M^{\mathbb{G}}$)  is a  solution to the RBSDE (\ref{nonLinear}) that corresponds to  $(f, S, \xi)$, then for any $\alpha>0$,  any $\mathbb F$-stopping time $\sigma\in {\cal T}_0^T(\mathbb F)$ and any $t\in [0,T]$, the following holds
\begin{align}\label{inequa4deltaYG2}
&\exp\bigl({{{\alpha}\over{2}}(t\wedge\tau)}\bigr)\vert{Y}^{\mathbb G}_t\vert{I}_{\{\sigma\leq\tau\}}I_{\{\sigma\leq {t}\}}\nonumber\\
&\leq{E}^{\widetilde Q}\left[\sup_{\sigma\wedge\tau\leq{u}\leq T\wedge\tau}e^{\alpha u/2} S_u^+ +e^{\alpha(T\wedge\tau)/2}\vert\xi\vert{I}_{\{\sigma\leq\tau\}} +{1\over{\sqrt{\alpha}}}\sqrt{\int_{\sigma\wedge\tau}^{T\wedge\tau} e^{\alpha s}\vert{f}(s,0,0)\vert^2 ds}  \ \Bigg|\ {\cal G}_t\right]\nonumber\\
&+{{C_{Lip}}\over{\sqrt{\alpha}}}E^{\widetilde Q}\left[\sqrt{\int_{\sigma\wedge\tau}^{T\wedge\tau} e^{\alpha s}\vert{ Y}^{\mathbb G}_s\vert^2 ds}+ \sqrt{\int_{\sigma\wedge\tau}^{T\wedge\tau} e^{\alpha s}( Z^{\mathbb G}_s)^2 ds} \ \bigg|\ {\cal G}_t\right].
\end{align}
\end{lemma}
\begin{proof} The proof of assertion (a) follows the same footsteps of the proof of (\ref{RBSDE2Snell}) in  Theorem \ref{abcde}. Thus, the rest of this proof focuses on proving assertions (b) and (c) in two parts.\\
{\bf Part 1.} This part proves assertion (b). To this end, we start by proving the following 
 \begin{eqnarray}\label{inequa4deltaYG1}
\vert \delta Y^{\mathbb G}_t\vert\leq E^{\widetilde Q}\left[\int_{t\wedge\tau}^{T\wedge\tau}\vert \Delta f_s\vert ds+\sup_{t\wedge\tau<s\leq T\wedge\tau}\vert \delta S_u\vert +\vert \delta\xi\vert\ \bigg|\ {\cal G}_t\right],\ \Delta f_t:=f_1(t, Y^{\mathbb G,1}_t, Z^{\mathbb G,1}_t)-f_2(t, Y^{\mathbb G,2}_t, Z^{\mathbb G,2}_t)
\end{eqnarray} 
Let $t\in [0,T]$ be arbitrary but fixed. Hence, by applying assertion (a) to each $Y^{\mathbb G, i}$, $i=1,2$, we deduce the existence of $\nu_i\in  {\cal T}_{t\wedge\tau}^{T\wedge\tau}(\mathbb {G})$ for $i=1,2$ such that the second equality of (\ref{YGessSup}) holds for each $Y^{\mathbb G, i}$, $i=1,2$, and in virtue of the first equality in (\ref{YGessSup}) we  have 
\begin{eqnarray*}
Y^{\mathbb G, 1}_t\geq E^{\widetilde Q}\left[\int_{t\wedge\tau}^{\nu_2\wedge\tau} f_1(s, Y^{\mathbb G,1}_s, Z^{\mathbb G,1}_s)ds +S_{\nu_2}^1I_{\{\nu_2<T\wedge\tau\}}+\xi^1 I_{\{\nu_2=T\wedge\tau\}}\ \Bigg|{\cal G}_t\ \right],\\
Y^{\mathbb G, 2}_t\geq  E^{\widetilde Q}\left[\int_{t\wedge\tau}^{\nu_1\wedge\tau} f_2(s, Y^{\mathbb G,2}_s, Z^{\mathbb G,2}_s)ds +S_{\nu_1}^2I_{\{\nu_1<T\wedge\tau\}}+\xi^2 I_{\{\nu_1=T\wedge\tau\}}\ \Bigg|{\cal G}_t\ \right].\end{eqnarray*}
By combining all these remarks, we derive 
\begin{eqnarray*} 
&&E^{\widetilde Q}\left[\int_{t\wedge\tau}^{\nu_2\wedge\tau} \Delta f_s ds +\delta S_{\nu_2}I_{\{\nu_2<T\wedge\tau\}}+\delta \xi I_{\{\nu_2=T\wedge\tau\}}\ \Bigg|{\cal G}_t\right]\leq\delta Y^{\mathbb G}_t\\
\rm{and}&&\delta Y^{\mathbb G}_t\leq E^{\widetilde Q}\left[\int_{t\wedge\tau}^{\nu_1\wedge\tau} \Delta f_s ds +\delta S_{\nu_1}I_{\{\nu_1<T\wedge\tau\}}+\delta\xi I_{\{\nu_1=T\wedge\tau\}}\ \Bigg|{\cal G}_t\right].\end{eqnarray*}
Therefore,  on the one hand, (\ref{inequa4deltaYG1}) follows immediately from these two inequalities. On the other hand, due to H\"older's inequality, for any nonnegative and progressively measurable process $h$, we have
\begin{eqnarray}\label{InequalitySchwartz}
\int_{t\wedge\tau}^{T\wedge\tau} h_s ds\leq \left({{{p_1}\over{\alpha' q_1}}}\right)^{{1\over{q_1}}}\exp\left(-{{\alpha'(t\wedge\tau)}\over{p_1}}\right)\left(\int_{t\wedge\tau}^{T\wedge\tau} e^{\alpha' s}h^{p_1}_s ds\right)^{{1\over{p_1}}},\  \alpha'>0,\  p_1>1,\ q_1:={{p_1}\over{p_1-1}}.
\end{eqnarray}
By using the fact that $\vert\Delta f_s\vert \leq \vert \delta f_s\vert +C_{Lip}\vert\delta Y^{\mathbb G}_s\vert +C_{Lip}\vert\delta Z^{\mathbb G}_s\vert $, and applying the above inequality repeatedly, we derive
\begin{eqnarray*}
\int_{t\wedge\tau}^{T\wedge\tau}\vert \Delta f_s\vert ds&&\leq {1\over{\sqrt{\alpha}}}\exp\left(-{{\alpha(t\wedge\tau)}\over{2}}\right)\left\{\left(\int_{t\wedge\tau}^{T\wedge\tau} e^{\alpha s}\vert\delta f_s\vert^{2} ds\right)^{{1\over{2}}}
+C_{Lip}\left(\int_{t\wedge\tau}^{T\wedge\tau} e^{\alpha s}\vert\delta Y^{\mathbb G}_s\vert^2 ds\right)^{{1\over{2}}}\right\}\\
&&\hskip 1cm +{{C_{Lip}}\over{\sqrt{\alpha}}}\exp\left(-{{\alpha(t\wedge\tau)}\over{2}}\right)\left(\int_{t\wedge\tau}^{T\wedge\tau} e^{\alpha s}(\delta Z^{\mathbb G}_s)^2 ds\right)^{{1\over{2}}}
\end{eqnarray*}
Thus, by combining this inequality with (\ref{inequa4deltaYG1}), we derive 
\begin{eqnarray*}
\exp\left({{\alpha(t\wedge\tau)}\over{ 2}}\right)\vert \delta Y^{\mathbb G}_t\vert&&\leq 
E^{\widetilde Q}\left[\sup_{t\wedge\tau<s\leq T\wedge\tau}e^{\alpha s/2}\vert \delta S_u\vert +e^{\alpha(T\wedge\tau)/2}\vert \delta\xi\vert+{{C_{Lip}}\over{\sqrt{\alpha}}} \sqrt{\int_{t\wedge\tau}^{T\wedge\tau} e^{\alpha s}(\delta Z^{\mathbb G}_s)^2 ds}\ \Bigg|\ {\cal G}_t\right]\\
&&+{1\over{\sqrt{\alpha}}}E^{\widetilde Q}\left[ C_{Lip}\left(\int_{t\wedge\tau}^{T\wedge\tau} e^{\alpha s}\vert\delta Y^{\mathbb G}_s\vert^2 ds\right)^{{1\over{2}}}  +\left(\int_{t\wedge\tau}^{T\wedge\tau} e^{\alpha s}\vert\delta f_s\vert^2 ds\right)^{{1\over{2}}}   \ \bigg|\ {\cal G}_t\right].
 \end{eqnarray*}
 Here $C_{Lip}$ is the Lipschitz' constant associated to the driver $f$ defined in (\ref{LipschitzAssumption}). Thus, (\ref{inequa4deltaYG2bis}) follows immediately from the inequality above, and hence part 1 is complete.\\
 {\bf Part 2.} Here we prove assertion (c). Thus, we consider $\alpha>0$ and an $\mathbb F$-stopping time $\sigma$. Similarly as in part 1, for any $t\in [0,T]$, thanks to (\ref{YGessSup}) we have
 \begin{align*}
 &Y^{\mathbb G}_t\geq-{E}^{\widetilde Q}\left[\int_{t\wedge\tau}^{T\wedge\tau} \left(f(s, Y^{\mathbb G}_s, Z^{\mathbb G}_s)\right)^-{d}s +\xi^- \Bigg|\ {\cal G}_t\right]\\
  &Y^{\mathbb G}_t\leq{E}^{\widetilde Q}\left[\int_{t\wedge\tau}^{T\wedge\tau} \left(f(s, Y^{\mathbb G}_s, Z^{\mathbb G}_s)\right)^+ds +\sup_{t\wedge\tau\leq\theta\leq{T}\wedge\tau}S_{\theta}^+ +\xi^+ \ \Bigg|\ {\cal G}_t\right].
  \end{align*}
  Thus, by combining these inequalities with $\vert{x}\vert=x^++x^-$, we obtain
  \begin{align}\label{MainInequality4Lemmac}
 \vert{Y}^{\mathbb G}_t\vert\leq {E}^{\widetilde Q}\left[\int_{t\wedge\tau}^{T\wedge\tau} \vert{f}(s, Y^{\mathbb G}_s, Z^{\mathbb G}_s)\vert{d}s +\sup_{t\wedge\tau\leq\theta\leq{T}\wedge\tau}S_{\theta}^{+} +\vert\xi\vert\quad \Bigg|\ {\cal G}_t\right]
   \end{align}
   Then the Lipschitz assumption of $f$ in (\ref{LipschitzAssumption}) implies that
    \begin{align*}
   \int_{\sigma\wedge\tau}^{T\wedge\tau} \vert{f}(s, Y^{\mathbb G}_s, Z^{\mathbb G}_s)\vert{d}s\leq  \int_{\sigma\wedge\tau}^{T\wedge\tau} \vert{f}(s, 0, 0)\vert{d}s+ C_{Lip}\int_{\sigma\wedge\tau}^{T\wedge\tau} \vert{Y}^{\mathbb G}_s\vert{d}s+ C_{Lip}\int_{\sigma\wedge\tau}^{T\wedge\tau} \vert{Z}^{\mathbb G}_s\vert{d}s.
     \end{align*}
     Hence, by applying (\ref{InequalitySchwartz}) to each term on the right-hand-side above for $p_1=q_1=2$,  and inserting the resulting inequality in (\ref{MainInequality4Lemmac}) afterwards, we obtain for any $t\in [0,T]$
      \begin{align*}
 \vert{Y}^{\mathbb G}_t\vert&\leq {{e^{-\alpha(t\wedge\tau)/2}}\over{\sqrt{\alpha}}}{E}^{\widetilde Q}\left[ \left(\int_{t\wedge\tau}^{T\wedge\tau} e^{\alpha s}(f(s,0,0))^2 ds\right)^{{1\over{2}}}  +C_{Lip} \left(\int_{t\wedge\tau}^{T\wedge\tau} e^{\alpha s}\vert\delta Y^{\mathbb G}_s\vert^2 ds\right)^{{1\over{2}}}\Bigg|\ {\cal G}_t\right]\\
 &+e^{-\alpha(t\wedge\tau)/2}{E}^{\widetilde Q}\left[{{C_{Lip}}\over{\sqrt{\alpha}}}\sqrt{\int_{t\wedge\tau}^{T\wedge\tau} e^{\alpha s}(\delta Z^{\mathbb G}_s)^2 ds} +\sup_{t\wedge\tau\leq\theta\leq{T}\wedge\tau}e^{\theta/2}S_{\theta}^{+} +e^{\alpha(T\wedge\tau)/2}\vert\xi\vert\quad \Bigg|\ {\cal G}_t\right].
   \end{align*}
Therefore, the inequality (\ref{inequa4deltaYG2}) follows immediately from combining the above inequality with the fact that $(\sigma\leq{t})\cap(\sigma\leq\tau)\in {\cal G}_t$. This proves assertion (c) and ends the proof  of the lemma.
\end{proof}
Throughout the rest of the paper, for any $ p\in (1,+\infty)$, $\alpha_0(p)$ and $\alpha_1(p)$ are given by 
\begin{eqnarray}\label{alphaZero}
\begin{cases}
\alpha_0(p):=\max\left(2C_{Lip}+2C_{Lip}^{2}+1,81\left\{1+{{9\sqrt{2}\kappa(1+C_{DB})}\over{3-\sqrt{8}}}\right\}^2 C_{DB}^2C_{Lip}^2\right)\\
\alpha_1(p):=\max\left( {{8}\over{9}}+4C_{Lip}+4C_{Lip}^2,{{81}\over{4}}C_{DB}^2C_{Lip}^2\left(1+{{9\sqrt{2}\kappa(1+C_{DB})}\over{3-\sqrt{8}}}\right)^2\right).
\label{alphaOne}
\end{cases}
\end{eqnarray}
Here $C_{DB}$ is the Doob's constant that depends on $p\in (1,+\infty)$ and $\kappa$ is the positive constant that depends on $p\in (1,+\infty)$ only given by Lemma \ref{Lemma4.8FromChoulliThesis}, and $C_{Lip}$ is the Lipschitz constant in (\ref{LipschitzAssumption}).

\begin{theorem}\label{WhyNot} For $p>1$ and $\alpha>\alpha_0(p)$, there exists $\widehat{C}>0$ that depend on $(\alpha,p)$ only such that for any $\mathbb F$-stopping time $\sigma\in {\cal T}_0^T(\mathbb F)$ and any solution to (\ref{nonLinear}), denoted by ($Y^{\mathbb{G}},Z^{\mathbb{G}},M^{\mathbb{G}},K^{\mathbb{G}}$), we have 
\begin{align*}
&\Vert {e^{\alpha\cdot/2}} Y^{\mathbb G}I_{\{\tau\geq\sigma\}}I_{\Lbrack\sigma,+\infty\Lbrack}\Vert_{\mathbb{D}_{\tau\wedge{T}}(\widetilde{Q},p)}+\Vert e^{ {\alpha}\cdot/2}Z^{\mathbb{G}}I_{\Rbrack\sigma,+\infty\Lbrack}\Vert_{\mathbb{S}_{\tau\wedge{T}}(\widetilde{Q}, p)}\nonumber\\
&+\Vert e^{ {\alpha}\cdot/2}Y^{\mathbb{G}}I_{\Rbrack\sigma,+\infty\Lbrack}\Vert_{\mathbb{S}_{\tau\wedge{T}}(\widetilde{Q}, p)}+\Vert (e^{ {\alpha'}(\tau\wedge\cdot)}I_{\Rbrack\sigma,+\infty\Lbrack}\is{K^{\mathbb{G}}})_{T\wedge\tau}\Vert_{L^p(\widetilde{Q})}+\Vert{e^{\alpha(\tau\wedge\cdot)/2}}I_{\Rbrack\sigma,+\infty\Lbrack} \is{M^{\mathbb{G}}}\Vert_{{\cal{M}}^{p}_T(\widetilde{Q})}\\
&\leq \widehat{C}\left\{  \Vert{e^{\alpha(T\wedge\tau)/2}}\xi{I}_{\{\tau\geq\sigma\}}\Vert_{L^p(\widetilde{Q})}+\Vert{e^{(\alpha-\alpha')\cdot}}S^{+}I_{\Lbrack\sigma,+\infty\Lbrack}\Vert_{\mathbb{D}_{\tau\wedge{T}}(\widetilde{Q}, p)}+\Vert e^{ {\alpha}\cdot/2}f(\cdot,0,0)I_{\Rbrack\sigma,+\infty\Lbrack}\Vert_{\mathbb{S}_{\tau\wedge{T}}(\widetilde{Q}, p)}\right\}.\end{align*}
\end{theorem}
\begin{proof} Let $\sigma\in {\cal T}_0^T(\mathbb F)$ be an $\mathbb F$-stopping time. Remark that, in virtue of (\ref{inequa4deltaYG2}) and Doob's inequality under $(\widetilde{Q}, \mathbb G)$, on the one hand, we have
\begin{align}
&\Vert {e^{\alpha\cdot/2}}Y^{\mathbb G}I_{\{\tau\geq\sigma\}}I_{\Lbrack\sigma,+\infty\Lbrack} \Vert_{\mathbb{D}_{T\wedge\tau}(\widetilde{Q},p)}\nonumber\\
&\leq C_{DB}\left\{
\Vert {e^{\alpha\cdot/2}}S^+I_{\Lbrack\sigma,+\infty\Lbrack}\Vert_{\mathbb{D}_{T\wedge\tau}(\widetilde{Q},p)}+\Vert{e^{\alpha(T\wedge\tau)/2}}\xi{I}_{\{\tau\geq\sigma\}}\Vert_{L^p(\widetilde{Q})}+{1\over{\sqrt{\alpha}}} \Vert{e^{\alpha\cdot/2}} f(\cdot,0,0)I_{\Rbrack\sigma,+\infty\Lbrack}\Vert_{\mathbb{S}_{T\wedge\tau}(\widetilde{Q},p)}\right\}\nonumber\\
&+{{C_{DB}C_{Lip}}\over{\sqrt{\alpha}}}\left\{ \Vert{e^{\alpha\cdot/2}} Z^{\mathbb{G}}I_{\Rbrack\sigma,+\infty\Lbrack}\Vert_{\mathbb{S}_{T\wedge\tau}(\widetilde{Q},p)}+\Vert{e^{\alpha\cdot/2}}Y^{\mathbb{G}}I_{\Rbrack\sigma,+\infty\Lbrack}\Vert_{\mathbb{S}_{T\wedge\tau}(\widetilde{Q},p)}\right\}.\label{equa445}
\end{align}
On the other hand, by combining It\^o applied to $e^{\alpha{t}}(Y^{\mathbb G}_t)^2$, $e^{\alpha(\sigma\wedge\tau)}(Y^{\mathbb G}_{\sigma\wedge\tau})^2\geq0$,  (\ref{nonLinear}), and Young's inequality (i.e. $2xy\leq \epsilon{x^2}+y^2/\epsilon$ for any $\epsilon>0$), we derive 
\begin{align}
&\overbrace{(\alpha-2C_{Lip}-2C_{Lip}^{2}-{\epsilon}^{-1})}^{C} \int_{\sigma\wedge\tau}^{T\wedge\tau}e^{\alpha s}(Y_{s}^{\mathbb{G}})^{2}ds+\frac{1}{2} \int_{\sigma\wedge\tau}^{T\wedge\tau}e^{\alpha s}(Z_{s}^{\mathbb{G}})^{2}ds+ \int_{\sigma\wedge\tau}^{T\wedge\tau}e^{\alpha t}d[M^{\mathbb{G}}, M^{\mathbb{G}}]_{s}\nonumber\\
&\leq e^{\alpha (T\wedge\tau)}\xi^{2}I_{\{\sigma\leq\tau\}}+\epsilon\int_{\sigma\wedge\tau}^{T\wedge\tau}e^{\alpha s}\vert{ f(s,0,0)}\vert^{2}ds+2\int_{\sigma\wedge\tau}^{T\wedge\tau}e^{\alpha s}Y_{s-}^{\mathbb{G}}dK_{s}^{\mathbb{G}}+(I_{\Rbrack\sigma,+\infty\Lbrack}\is{L}^{\mathbb G,1})_{T\wedge\tau}\nonumber\\
&\leq{e^{\alpha (T\wedge\tau)}}\xi^{2}I_{\{\sigma\leq\tau\}}+\epsilon\int_{\sigma\wedge\tau}^{T\wedge\tau}e^{\alpha s}\vert{ f(s,0,0)}\vert^{2}ds+2\int_{\sigma\wedge\tau}^{T\wedge\tau}e^{\alpha s}{S_{s-}^+}dK_{s}^{\mathbb{G}}+\sup_{0\leq{t\leq{T}\wedge\tau}}\vert(I_{\Rbrack\sigma,+\infty\Lbrack}\is{L}^{\mathbb G,1})_t\vert, \label{equa446}\end{align}
where the last equality is due to the Skorokhod's condition and  $L^{\mathbb G,1}\in {\cal M}_{loc}(\mathbb G)$ is given by 
\begin{eqnarray}\label{LG1}
L^{\mathbb G,1}:=2e^{\alpha (\tau\wedge\cdot)}(Y_{-}^{\mathbb{G}}-\Delta{K_{s}^{\mathbb{G}}}))\is{M^{\mathbb{G}}}+2e^{\alpha (\tau\wedge\cdot)}(Y_{-}^{\mathbb{G}})Z^{\mathbb{G}}\is W^{\tau}.\end{eqnarray}
Thus, by applying Lemma \ref{Lemma4.8FromChoulliThesis} to $I_{\Rbrack\sigma,+\infty\Lbrack}\is{L}^{\mathbb G,1}$ with $a=b=p$, and using Doob's inequality afterwards to the martingale $E^{\widetilde{Q}}[\sup_{0\leq{s}\leq{T}\wedge\tau}\vert{Y}^{\mathbb G}_s\vert{I}_{\{\sigma\leq{s}\wedge\tau\}}\ \big|{\cal{G}}_t]$, we derive 
\begin{align}\label{Control4LG1}
&\Vert\sqrt{\vert{I}_{\Rbrack\sigma,+\infty\Lbrack}\is{L}^{\mathbb G,1}\vert}\Vert_{\mathbb{D}_{T\wedge\tau}(\widetilde{Q},p)}\nonumber\\
&\leq 2\sqrt{\left\{\kappa(1+C_{DB})\Vert {e}^{\alpha(\cdot\wedge\tau)/2}I_{\Rbrack\sigma,+\infty\Lbrack}\is{M}^{\mathbb{G}}\Vert_{{\cal{M}}^{p}_T(\widetilde{Q})}+\Vert{e}^{\alpha\cdot/2}Z_{s}^{\mathbb{G}}I_{\Rbrack\sigma,+\infty\Lbrack}\Vert _{\mathbb{S}_{T\wedge\tau}(\widetilde{Q},p)}\right\} \Vert {e^{\alpha\cdot/2}}Y^{\mathbb G}{I}_{\Lbrack\sigma,+\infty\Lbrack}\Vert_{\mathbb{D}_{T\wedge\tau}(\widetilde{Q},p)} } \nonumber\\
&\leq \epsilon_1\left\{\Vert{e^{\alpha(\tau\wedge\cdot)/2}}{I}_{\Rbrack\sigma,+\infty\Lbrack} \is {M^{\mathbb{G}}}\Vert_{{\cal{M}}^{p}_T(\widetilde{Q})} +\Vert{e^{\alpha\cdot/2 }}Z^{\mathbb{G}}\Vert_{\mathbb{S}_{T\wedge\tau}(\widetilde{Q},p)} \right\}+{{\kappa(1+C_{DB})}\over{\epsilon_1}}\Vert {e^{\alpha\cdot/2}}Y^{\mathbb G}{I}_{\Lbrack\sigma,+\infty\Lbrack}\Vert_{\mathbb{D}_{T\wedge\tau}(\widetilde{Q},p)}.
\end{align}
Therefore, by combining (\ref{equa445}), (\ref{equa446}), and (\ref{Control4LG1}),  and the fact that $\Vert\sqrt{\sum_{i=1}^n X_i}\Vert_{L^p(\widetilde{Q})}\geq n^{-1}\sum_{i=1}^n \Vert \sqrt{X_i}\Vert_{L^p(\widetilde{Q})}$ for any nonnegative random  variables $(X_i)_{i=1,...,n}$, we derive the following inequality.\\
\begin{align}\label{Major1} 
&\Vert {e^{\alpha\cdot/2}}Y^{\mathbb G}{I}_{\Lbrack\sigma,+\infty\Lbrack} \Vert_{\mathbb{D}_{T\wedge\tau}(\widetilde{Q},p)}+C_1\Vert{e^{\alpha\cdot/2}} Y^{\mathbb{G}}{I}_{\Rbrack\sigma,+\infty\Lbrack}\Vert_{\mathbb{S}_{T\wedge\tau}(\widetilde{Q},p)}\nonumber\\
&+C_2\Vert{e^{\alpha\cdot/2}}Z^{\mathbb{G}}{I}_{\Rbrack\sigma,+\infty\Lbrack}\Vert_{\mathbb{S}_{T\wedge\tau}(\widetilde{Q},p)}+C_3\Vert{e^{\alpha\cdot/2}} {I}_{\Rbrack\sigma,+\infty\Lbrack}\is {M^{\mathbb{G}}}\Vert_{{\cal M}^{p}_{T}(\widetilde{Q})}\nonumber\\
& \leq C_4 \Vert{e^{\alpha\cdot/2}} f(\cdot,0,0){I}_{\Rbrack\sigma,+\infty\Lbrack}\Vert_{\mathbb{S}_{T\wedge\tau}(\widetilde{Q},p)}+C_5 \Vert {e^{\alpha(T\wedge\tau)/2}}\xi{I}_{\{\sigma\leq\tau\}}\Vert_{\mathbb{L}^{p}(\widetilde{Q})}+C_6\Vert {e^{\alpha\cdot/2}} S^+{I}_{\Lbrack\sigma,+\infty\Lbrack}\Vert_{\mathbb{D}_{T\wedge\tau}(\widetilde{Q},p)}\nonumber\\
&+C_7\Vert {e^{(\alpha-\alpha')\cdot}}  S^+{I}_{\Lbrack\sigma,+\infty\Lbrack}\Vert_{\mathbb{D}_{T\wedge\tau}(\widetilde{Q},p)}^{1/2}\Vert{e}^{\alpha'(\tau\wedge\cdot)}{I}_{\Rbrack\sigma,+\infty\Lbrack}\is{K}^{\mathbb G}_T \Vert_{\mathbb{L}^{p}(\widetilde{Q})}^{1/2},\quad \mbox{for}\quad \alpha'<\alpha/2,
\end{align}
where $C_i$, $i=1,...,7$ are given by 
\begin{align}\label{Constants4estimates}
\begin{cases}
C_1:={\sqrt{C}\over{3}}-\left(1+{{\kappa(1+C_{DB})}\over{\epsilon_1}} \right){{ C_{DB}C_{Lip}}\over{\sqrt{\alpha}}},\quad{C}_2:={1\over{3\sqrt 2}}-\epsilon_1-\left(1+{{\kappa(1+C_{DB})}\over{\epsilon_1}} \right){{ C_{DB}C_{Lip}}\over{\sqrt{\alpha}}},\\
\\
C_3:={1\over {3}}-\epsilon_1,\quad C_4:=\sqrt{\epsilon}+{{ C_{DB}}\over{\sqrt{\alpha}}}\left(1+{{\kappa(1+C_{DB})}\over{\epsilon_1}} \right),\\
\\
C_5:=1+C_{DB}\left(1+{{\kappa(1+C_{DB})}\over{\epsilon_1}} \right),\quad C_6:=C_{DB}\left(1+{{\kappa(1+C_{DB})}\over{\epsilon_1}} \right),\quad C_7:=\sqrt{2}.\end{cases}
\end{align}
Thus, the next step consists of controlling the norm of $K^{\mathbb{G}}$. To this end, we use the RBSDE (\ref{nonLinear}) and Ito's formula,  and derive for any $\alpha'>0$ and $t>\sigma$
 \begin{align*}
    &\int_{t\wedge\tau}^{T\wedge\tau}e^{\alpha' {s}}dK_{s}^{\mathbb{G}}=-   \int_{t\wedge\tau}^{T\wedge\tau} e^{\alpha' {s}}dY^{\mathbb{G}}_{s}  -\int_{s\wedge\tau}^{T\wedge\tau}e^{\alpha' {s}}f(s,Y^{\mathbb{G}}_{s},Z^{\mathbb{G}}_{s})ds-  \int_{t\wedge\tau}^{T\wedge\tau}e^{\alpha' {s}}dM^{\mathbb{G}}_{t}+  \int_{t\wedge\tau}^{T\wedge\tau}e^{\alpha' {s}}Z^{\mathbb{G}}_{s}dW_{s}^{\tau}\\
    &\overset{Ito}{=}-  e^{\alpha' {T\wedge\tau}}Y^{\mathbb{G}}_{T\wedge\tau}+ e^{\alpha'{t\wedge\tau}}Y^{\mathbb{G}}_{t\wedge\tau}+ \int_{t\wedge\tau}^{T\wedge\tau} e^{\alpha' {s}}({\alpha'}Y^{\mathbb{G}}_{s}-f(s,Y^{\mathbb{G}}_{s},Z^{\mathbb{G}}_{s}))ds-  \int_{t\wedge\tau}^{T\wedge\tau}e^{\alpha' {s}}dM^{\mathbb{G}}_{s}+  \int_{t\wedge\tau}^{T\wedge\tau}e^{\alpha' {s}}Z^{\mathbb{G}}_{s}dW_{s}^{\tau}.
    \end{align*}
    Therefore, by using this latter equality together with (\ref{LipschitzAssumption}), we derive 
    
       \begin{align*}
& E^{\widetilde{Q}} \left[\int_{t\wedge\tau}^{T\wedge\tau}e^{ {\alpha'}s}dK_{s}^{\mathbb{G}}\ \big|{\cal G}_{t\wedge\tau}\right]\leq E^{\widetilde{Q}}\left[2\sup_{t\wedge\tau<u\leq T\wedge\tau}e^{ {\alpha'}s}\vert {Y^{\mathbb{G}}_s}\vert + \int_{t\wedge\tau}^{T\wedge\tau}e^{ {\alpha'}s}\vert {\alpha'}Y^{\mathbb{G}}_{s}+ f(s,Y^{\mathbb{G}}_{s},Z^{\mathbb{G}}_{s})\vert ds\ \big|{\cal G}_{t\wedge\tau}\right],\\
&\leq
E^{\widetilde{Q}}\left[2\sup_{t\wedge\tau<u\leq T\wedge\tau}e^{ {\alpha'}s}\vert {Y^{\mathbb{G}}_s}\vert + \int_{t\wedge\tau}^{T\wedge\tau}e^{ {\alpha'}s}( {\alpha'}+C_{Lip})\vert{Y^{\mathbb{G}}_s}\vert ds+ \int_{t\wedge\tau}^{T\wedge\tau}e^{ {\alpha'}s}\vert {f(s,0,0)}\vert ds\ \big|{\cal G}_{t\wedge\tau}\right]\\
&+ C_{Lip}E^{\widetilde{Q}}\left[\int_{t\wedge\tau}^{T\wedge\tau}e^{ {\alpha'}s}\vert {Z^{\mathbb{G}}_{s}}\vert ds\ \big|{\cal G}_{t\wedge\tau}\right].
    \end{align*}
    Then by applying (\ref{InequalitySchwartz}) for each term above, and choosing $\alpha'<\alpha/2$, we get for $t>\sigma$
    \begin{align*}
    & E^{\widetilde{Q}} \left[\int_{t\wedge\tau}^{T\wedge\tau}e^{ {\alpha'}s}dK_{s}^{\mathbb{G}}\ \big|{\cal G}_{t\wedge\tau}\right]\\
    &\leq E^{\widetilde{Q}}\left[2\sup_{\sigma\wedge\tau\leq{u}\leq T\wedge\tau}e^{ {\alpha'}s}\vert {Y^{\mathbb{G}}_s}\vert +{{\alpha'+C_{Lip}}\over{\sqrt{\alpha-2\alpha'}}} \sqrt{\int_{\sigma\wedge\tau}^{T\wedge\tau}e^{ {\alpha}s}(Y^{\mathbb{G}}_s)^2 ds}+ {1\over{\sqrt{\alpha-2\alpha'}}} \sqrt{\int_{\sigma\wedge\tau}^{T\wedge\tau}e^{ {\alpha}s}(f(s,0,0))^2 ds}\ \big|{\cal G}_{t\wedge\tau}\right]\\
&+ {{C_{Lip}}\over{\sqrt{\alpha-2\alpha'}}} E^{\widetilde{Q}}\left[\sqrt{\int_{\sigma\wedge\tau}^{T\wedge\tau}e^{ {\alpha}s}(Z^{\mathbb{G}}_s)^2 ds}\ \big|{\cal G}_{t\wedge\tau}\right].
    \end{align*}
Therefore, thanks to Theorem \ref{DellacherieAndMeyer}, we deduce that  for any $p>1$ and $\alpha'<\alpha/2$, we have 
 \begin{align}
&\Vert (e^{ {\alpha'}\cdot}{I}_{\Rbrack\sigma,+\infty\Lbrack}\is{K^{\mathbb{G}}})_{T\wedge\tau}\Vert_{L^p(\widetilde{Q})}\nonumber\\
&\leq C'\left\{\Vert {e^{\alpha\cdot/2}} Y^{\mathbb G}{I}_{\Lbrack\sigma,+\infty\Lbrack}\Vert_{\mathbb{D}_{T\wedge\tau}(\widetilde{Q},p)}+\Vert e^{ {\alpha}\cdot/2} Y^{\mathbb{G}}{I}_{\Rbrack\sigma,+\infty\Lbrack}\Vert_{\mathbb{S}_{T\wedge\tau}(\widetilde{Q}, p)}+\Vert e^{ {\alpha}\cdot/2}Z^{\mathbb{G}}{I}_{\Rbrack\sigma,+\infty\Lbrack}\Vert_{\mathbb{S}_{T\wedge\tau}(\widetilde{Q}, p)}\right\}\nonumber\\
&+C'\Vert e^{ {\alpha}\cdot/2}f(\cdot,0,0){I}_{\Rbrack\sigma,+\infty\Lbrack}\Vert_{\mathbb{S}_{T\wedge\tau}(\widetilde{Q}, p)},\label{Equa444}
    \end{align}
    where the constant $C'$ is given by 
    \begin{align*}
C':=p\max\left(2, \frac{\alpha'+C_{Lip}}{\sqrt{\alpha-2\alpha'}}\right).
\end{align*}
 Remark that for $\alpha>\alpha_0(p)$, and by choosing $\epsilon=9/5$ and $\epsilon_1=(3-\sqrt{8})/9\sqrt{2}$, we get $1/9<C_2\leq \min(C_1,C_3)$. Then by inserting (\ref{Equa444}) in (\ref{Major1}) and using Young's inequality afterwards, we get 
 \begin{align*}
 &\Vert {e}^{{{\alpha\cdot}\over{2}}} Y^{\mathbb G}{I}_{\Lbrack\sigma,+\infty\Lbrack}\Vert_{\mathbb{D}_{T\wedge\tau}(\widetilde{Q},p)}+\Vert e^{{{\alpha\cdot}\over{2}}} Y^{\mathbb{G}}{I}_{\Rbrack\sigma,+\infty\Lbrack}\Vert_{\mathbb{S}_{T\wedge\tau}(\widetilde{Q}, p)}+\Vert{e}^{{{\alpha\cdot}\over{2}}}Z^{\mathbb{G}}{I}_{\Rbrack\sigma,+\infty\Lbrack}\Vert_{\mathbb{S}_{T\wedge\tau}(\widetilde{Q},p)}+\Vert{e^{{{\alpha\cdot}\over{2}}}}{I}_{\Rbrack\sigma,+\infty\Lbrack} \is {M^{\mathbb{G}}}\Vert_{{\cal M}^{p}_{T}(\widetilde{Q})}\nonumber\\
& \leq \overline{C}\left\{ \Vert{e^{\alpha(\tau\wedge\cdot)/2}} f(\cdot,0,0){I}_{\Lbrack\sigma,+\infty\Lbrack}\Vert_{\mathbb{S}^{p}_{T}(\widetilde{Q})}+ \Vert {e^{\alpha(T\wedge\tau)/2}}\xi{I}_{\{\sigma\leq\tau\}}\Vert_{\mathbb{L}^{p}(\widetilde{Q})}+\Vert {e^{(\alpha-\alpha')\cdot}}  S^+{I}_{\Rbrack\sigma,+\infty\Lbrack}\Vert_{{\mathbb{D}}_{T\wedge\tau}(\widetilde{Q},p)}\right\},\end{align*}
where $\overline{C}:=(20(C')^2+C_6)/(C_2-(1/9))$. Therefore, the proof of the theorem follows immediately from combining the above inequality with (\ref{Equa444}) and choosing $\widehat{C}=\overline{C}(1+C')+C'$. This ends the proof of the theorem.
\end{proof}
\begin{theorem}\label{uniquenessNonlinear}
If ($Y^{\mathbb{G},i},Z^{\mathbb{G},i},K^{\mathbb{G},i}, M^{\mathbb{G},i}$)  is a  solution to the RBSDE (\ref{nonLinear}) that correspond to  $(f_{i}, S_{i}, \xi^i)$, $i=1,2$ respectively, then for any $p>1$ and $\alpha>\max(\alpha_1(p), \alpha_0(p))$ given in (\ref{alphaZero}), there exist positive $\widehat{C}_j$, $j=1,2,3,4$, that depend on $(\alpha, p)$ only such that $\lim_{\alpha\to\infty}\widehat{C}_1=0$ and 
\begin{align}\label{MainDeltaInequality}
&\Vert {e^{\alpha\cdot/2}}\delta Y^{\mathbb G} \Vert_{\mathbb{D}_{T\wedge\tau}(\widetilde{Q},\mathbb{G})}+\Vert{e^{\alpha\cdot/2}} \delta Y^{\mathbb{G}}\Vert_{\mathbb{S}_{T\wedge\tau}(\widetilde{Q},p)}+\Vert{e^{\alpha\cdot/2}} \delta Z^{\mathbb{G}}\Vert_{\mathbb{S}_{T\wedge\tau}(\widetilde{Q},p)}+\Vert{e^{\alpha(\tau\wedge\cdot)/2}} \is \delta {M^{\mathbb{G}}}\Vert_{{\cal M}^{p}(\widetilde{Q})}\nonumber\\
& \leq \widehat{C}_1 \Vert{e^{\alpha\cdot/2}} \delta f\Vert_{\mathbb{S}_{T\wedge\tau}(\widetilde{Q},p)}+\widehat{C}_2\Vert {e^{\alpha(T\wedge\tau)/2}} \delta \xi\Vert_{\mathbb{L}^{p}(\widetilde{Q})}+\widehat{C}_3\Vert {e^{\alpha\cdot/2}} \delta S\Vert_{\mathbb{D}_{T\wedge\tau}(\widetilde{Q},p)}\nonumber\\
&+\widehat{C}_4\sqrt{\Vert {e^{\alpha\cdot/2}} \delta S\Vert_{\mathbb{D}_{T\wedge\tau}(\widetilde{Q},p)}\left\{\sum_{i=1}^2\Delta(f^{(i)}, S^{(i)}, \xi^{(i)})\right\}}.
\end{align}
Here $\Delta(f^{(i)}, S^{(i)}, \xi^{(i)})$ is 
\begin{eqnarray}\label{Delta(i)}
\Delta(f^{(i)}, S^{(i)}, \xi^{(i)}):= \Vert{e^{\alpha(T\wedge\tau)/2}}\xi^{(i)}\Vert_{L^p(\widetilde{Q})}+\Vert{e^{\alpha(\tau\wedge\cdot)}}(S^{(i)})^{+}\Vert_{\mathbb{S}_T(\widetilde{Q}, p)}+\Vert e^{ {\alpha}(\tau\wedge\cdot)/2}f^{(i)}(\cdot,0,0)\Vert_{\mathbb{S}_T(\widetilde{Q}, p)},
\end{eqnarray}
and $(\delta Y^{\mathbb{G}}, \delta Z^{\mathbb{G}}, \delta M^{\mathbb{G}}, \delta K^{\mathbb{G}})$ and $(\delta f,\delta S, \delta\xi)$ are given by 
\begin{align*}
\delta Y^{\mathbb{G}}&:=Y^{\mathbb{G},1}-Y^{\mathbb{G},2},\ \delta Z^{\mathbb{G}}:=Z^{\mathbb{G},1}-Z^{\mathbb{G},2}, \delta M^{\mathbb{G}}:=M^{\mathbb{G},1}-M_{2}^{\mathbb{G},2}, \delta K^{\mathbb{G}}:=K^{\mathbb{G},1}-K^{\mathbb{G},2}, \\
\delta S&:=S^1-S^2,\quad \delta \xi:=\xi^1-\xi^2,\quad \delta f_t:= f_{1}(t,Y^{\mathbb{G},1}_t,Z^{\mathbb{G},1}_t)- f_{2}(t,Y^{\mathbb{G},1}_t,Z^{\mathbb{G},1}_t).
\end{align*}
\end{theorem}

\begin{proof} On the one hand, due to the Lipschitz assumption on $f$, we have 
\begin{align}
&\vert \Delta f_t\vert:=\vert f_{1}(t,Y^{\mathbb{G},1},Z^{\mathbb{G},1})- f_{2}(t,Y^{\mathbb{G},2},Z^{\mathbb{G},2})\vert \leq \vert \delta f_t\vert+C_{Lip}\vert\delta Y^{\mathbb{G}}_t\vert+C_{Lip}\vert \delta Z^{\mathbb{G}}_t\vert.\label{Lipschitz1}
\end{align}
On the other hand, in virtue of Lemma \ref{Estimation4DeltaY}-(b) and Doob's inequality, we get 
\begin{align}\label{Control4supYG}
\Vert {e^{\alpha\cdot/2}}\delta Y^{\mathbb G} \Vert_{\mathbb{D}_{T\wedge\tau}(\widetilde{Q},p)}&\leq C_{DB}\left\{
\Vert {e^{\alpha\cdot/2}}\delta S_u\Vert_{\mathbb{D}_{T\wedge\tau}(\widetilde{Q},p)}+\Vert{e^{\alpha(T\wedge\tau)/2}}\delta\xi\Vert_{L^p(\widetilde{Q})}+{1\over{\sqrt{\alpha}}} \Vert{e^{\alpha\cdot/2}} \delta f\Vert_{\mathbb{S}_{T\wedge\tau}(\widetilde{Q},p)}\right\}\nonumber\\
&+{{C_{DB}C_{Lip}}\over{\sqrt{\alpha}}}\left\{ \Vert{e^{\alpha\cdot/2}} \delta Z^{\mathbb{G}}\Vert_{\mathbb{S}_{T\wedge\tau}(\widetilde{Q},p)}+\Vert{e^{\alpha\cdot/2}} \delta Y^{\mathbb{G}}\Vert_{\mathbb{S}_{T\wedge\tau}(\widetilde{Q},p)}\right\}.
\end{align}
By combining It\^o applied to $e^{\alpha t}(\delta Y^{\mathbb{G}})^{2}$,  $(\delta Y_{0}^{\mathbb{G}})^{2}\geq0$ and (\ref{Lipschitz1}), and putting \begin{eqnarray}\label{LG}
L^{\mathbb G}:=e^{\alpha (\tau\wedge\cdot)}(\delta Y_{-}^{\mathbb{G}}-2\Delta(\delta K_{s}^{\mathbb{G}}))\is \delta M^{\mathbb{G}}+e^{\alpha (\tau\wedge\cdot)}(\delta Y_{-}^{\mathbb{G}})\delta Z^{\mathbb{G}}\is W^{\tau}\in {\cal M}_{loc}(\widetilde Q, \mathbb G),
\end{eqnarray}
 we derive 

\begin{align*}
&\alpha \int_{0}^{T\wedge\tau}e^{\alpha s}(\delta Y_{s}^{\mathbb{G}})^{2}ds+ \int_{0}^{T\wedge\tau}e^{\alpha s}(\delta Z_{s}^{\mathbb{G}})^{2}ds+ \int_{0}^{T\wedge\tau}e^{\alpha t}d[ \delta M^{\mathbb{G}}, \delta M^{\mathbb{G}}]_{s}\nonumber\\
&\leq e^{\alpha (T\wedge\tau)}(\delta \xi)^{2}+
2\int_{0}^{T\wedge\tau}e^{\alpha s}(\delta Y_{s}^{\mathbb{G}})\Delta f ds+2\int_{0}^{T\wedge\tau}e^{\alpha s}(\delta Y_{s-}^{\mathbb{G}})d\delta K_{s}^{\mathbb{G}}+L^{\mathbb G}_T,\\
&\leq  e^{\alpha (T\wedge\tau)}(\delta \xi)^{2}+2\int_{0}^{T\wedge\tau}e^{\alpha s}|\delta Y_{s}^{\mathbb{G}}|(|\delta f_s|+C_{Lip}(\vert\delta Y^{\mathbb{G}}\vert+\vert\delta Z^{\mathbb{G}}\vert))ds+2(e^{\alpha(\tau\wedge\cdot)}(\delta Y_{-}^{\mathbb{G}})\is \delta K^{\mathbb{G}})_{T\wedge\tau}+L^{\mathbb G}_T\\
&=e^{\alpha (T\wedge\tau)}(\delta \xi)^{2}+2\int_{0}^{T\wedge\tau}e^{\alpha s}\vert\delta Y_{s}^{\mathbb{G}}\vert\vert\delta f_s\vert ds+2\int_{0}^{T\wedge\tau}e^{\alpha s}C_{Lip}\vert\delta Y_{s}^{\mathbb{G}}\vert^{2}ds \\
&+2\int_{0}^{T\wedge\tau}e^{\alpha s}C_{Lip}\vert\delta Y_{s}^{\mathbb{G}}\vert\vert\delta Z^{\mathbb{G}}\vert ds+2\int_{0}^{T\wedge\tau}e^{\alpha s}(\delta Y_{s-}^{\mathbb{G}})d\delta K_{s}^{\mathbb{G}}+L^{\mathbb G}_T\\
&\overset{Young}{\leq}e^{\alpha (T\wedge\tau)}(\delta \xi)^{2}+\frac{1}{\epsilon}\int_{0}^{T\wedge\tau}e^{\alpha s}\vert\delta Y_{s}^{\mathbb{G}}\vert^{2}ds+\epsilon\int_{0}^{T\wedge\tau}e^{\alpha s}\vert\delta f_s\vert^{2} ds+2\int_{0}^{T\wedge\tau}e^{\alpha s}C_{Lip}\vert\delta Y_{s}^{\mathbb{G}}\vert^{2}ds \\&+2C_{Lip}^{2}\int_{0}^{T\wedge\tau}e^{\alpha s}\vert\delta Y_{s}^{\mathbb{G}}\vert^{2}ds+\frac{1}{2}\int_{0}^{T\wedge\tau}e^{\alpha s}\vert\delta Z^{\mathbb{G}}\vert^{2}ds+2\int_{0}^{T\wedge\tau}e^{\alpha s}(\delta Y_{s-}^{\mathbb{G}})d\delta K_{s}^{\mathbb{G}}+L^{\mathbb G}_T.
\end{align*}
Therefore, by arranging terms in the last inequality, we obtain 
\begin{align}
&\overbrace{(\alpha-2C_{Lip}-2C_{Lip}^{2}-{\epsilon}^{-1})}^{C} \int_{0}^{T\wedge\tau}e^{\alpha s}(\delta Y_{s}^{\mathbb{G}})^{2}ds+\frac{1}{2} \int_{0}^{T\wedge\tau}e^{\alpha s}(\delta Z_{s}^{\mathbb{G}})^{2}ds+ \int_{0}^{T\wedge\tau}e^{\alpha t}d[ \delta M^{\mathbb{G}},  \delta M^{\mathbb{G}}]_{s}\nonumber\\
&\leq e^{\alpha (T\wedge\tau)}(\delta \xi)^{2}+\epsilon\int_{0}^{T\wedge\tau}e^{\alpha s}\vert\delta f_s\vert^{2}ds+2\int_{0}^{T\wedge\tau}e^{\alpha s}(\delta Y_{s-}^{\mathbb{G}})d\delta K_{s}^{\mathbb{G}}+L^{\mathbb G}_T\label{equa399}\\
&\leq  e^{\alpha (T\wedge\tau)}(\delta \xi)^{2}+\epsilon\int_{0}^{T\wedge\tau}e^{\alpha s}\vert\delta f_s\vert^{2}ds+2\int_{0}^{T\wedge\tau}e^{\alpha s}\vert\delta S_{s-}\vert d\mbox{Var}_s(\delta{K^{\mathbb{G}}})+L^{\mathbb G}_T.\label{equa400}\end{align}
The last inequality is due to $e^{\alpha (\tau\wedge\cdot)}(\delta Y_{-}^{\mathbb{G}})\is \delta K^{\mathbb{G}}\leq e^{\alpha (\tau\wedge\cdot)}(\delta S_{-}^{\mathbb{G}})\is \delta K^{\mathbb{G}}\leq e^{\alpha (\tau\wedge\cdot)}\vert\delta S_{-}^{\mathbb{G}}\vert\is \mbox{Var}(\delta K^{\mathbb{G}})$ that follows from Skorokhod's condition. Furthermore, by applying Lemma \ref{Lemma4.8FromChoulliThesis} to $L^{\mathbb G}$ given in (\ref{LG}) with $a=b=p$ and Doob's inequality afterwards, there exists a constant $\kappa=\kappa(p)>0$ that depends on $p$ only such that
\begin{align}\label{Control4LG}
 &\Vert \vert L^{\mathbb G}\vert^{1/2}\Vert_{\mathbb{D}_{T\wedge\tau}(\widetilde{Q},p)}\nonumber\\
 &\leq \sqrt{\kappa(1+C_{DB})}\left\{\Vert{e}^{\alpha\cdot/2}\is\delta M^{\mathbb{G}}\Vert_{{\cal{M}}^p_T(\widetilde{Q})} +\Vert{e}^{\alpha\cdot/2}\delta Z^{\mathbb{G}}\Vert_{\mathbb{S}_{T\wedge\tau}(\widetilde{Q},p)}\right\}^{1/2}\Vert\delta Y^{\mathbb G}\Vert_{\mathbb{D}_{T\wedge\tau}(\widetilde{Q},p)}^{1/2}\nonumber\\
&\leq \epsilon_1\Vert{e}^{\alpha\cdot/2}\is\delta M^{\mathbb{G}}\Vert_{{\cal{M}}^p_T(\widetilde{Q})} +\epsilon_1\Vert{e}^{\alpha\cdot/2}\delta Z^{\mathbb{G}}\Vert_{\mathbb{S}_{T\wedge\tau}(\widetilde{Q},p)}+{{\kappa(1+C_{DB})}\over{\epsilon_1}}\Vert {e^{\alpha\cdot/2}}\delta Y^{\mathbb G}\Vert_{\mathbb{D}_{T\wedge\tau}(\widetilde{Q},p)}.
\end{align}
Therefore, by combining (\ref{equa400}), (\ref{Control4LG}) and (\ref{Control4supYG}) and choosing adequately $\alpha, \epsilon, \epsilon_1$ and using $n^{-1}\sum_{i=1}^n x_i^{p/2}\leq (\sum_{i=1}^n x_i)^{p/2}\leq n^{p/2}\sum_{i=1}^n x_i^{p/2}$ for any positive integer and any sequence of nonnegative number $x_i$, we derive 
\begin{align*}
&{1\over{3}}\left\{\sqrt{C}\Vert{e^{\alpha\cdot/2}} \delta Y^{\mathbb{G}}\Vert_{\mathbb{S}_{T\wedge\tau}(\widetilde{Q},p)} +2^{-1} \Vert{e^{\alpha\cdot/2}} \delta Z^{\mathbb{G}}\Vert_{\mathbb{S}_{T\wedge\tau}(\widetilde{Q},p)}+\Vert{e^{\alpha(\tau\wedge\cdot)/2}} \is \delta {M^{\mathbb{G}}}\Vert_{{\cal{M}}^{p}_T(\widetilde{Q})}\right\}\\
&\leq  \epsilon\Vert{e^{\alpha\cdot/2}} \delta f\Vert_{\mathbb{S}_{T\wedge\tau}(\widetilde{Q},p)}+\Vert {e^{\alpha(T\wedge\tau)/2}} \delta \xi\Vert_{\mathbb{L}^{p}(\widetilde{Q})}+\sqrt{2}\Vert {e^{\alpha\cdot/2}} \delta S\Vert_{\mathbb{D}_{T\wedge\tau}(\widetilde{Q},p)}^{1/2}\Vert \rm{Var}_T(e^{\alpha\cdot/2}\is\delta K^{\mathbb G}) \Vert_{\mathbb{L}^{p}(\widetilde{Q})}^{1/2}\\
&+\Vert\vert L^{\mathbb G}\vert^{1/2}\Vert_{\mathbb{D}_{T\wedge\tau}(\widetilde{Q},p)},\\
&\leq  \epsilon\Vert{e^{\alpha(\tau\wedge\cdot)/2}} \delta f\Vert_{\mathbb{S}_{T\wedge\tau}(\widetilde{Q})}+\Vert {e^{\alpha(T\wedge\tau)/2}} \delta \xi\Vert_{\mathbb{L}^{p}(\widetilde{Q})}+\sqrt{2}\Vert {e^{\alpha\cdot/2}} \delta S\Vert_{\mathbb{D}_{T\wedge\tau}(\widetilde{Q},p)}^{1/2}\Vert \rm{Var}_T(e^{\alpha\cdot/2}\is\delta K^{\mathbb G}) \Vert_{\mathbb{L}^{p}(\widetilde{Q})}^{1/2}\\
&+{\epsilon_1}\Vert{e^{\alpha(\tau\wedge\cdot)/2}} \is \delta {M^{\mathbb{G}}}\Vert_{{\cal{M}}_T^p(\widetilde{Q})}+{\epsilon_1}\Vert{e^{\alpha\cdot/2}} \delta Z^{\mathbb{G}}\Vert_{\mathbb{S}_{T\wedge\tau}(\widetilde{Q},p)}+{{\kappa(1+C_{DB})}\over{\epsilon_1}}\Vert {e^{\alpha\cdot/2}}\delta Y^{\mathbb G}\Vert_{\mathbb{D}_{T\wedge\tau}(\widetilde{Q},p)}.
\end{align*}
Then by combining this equality with (\ref{Control4supYG}) and (\ref{equa400}) we obtain
\begin{align}\label{BeforeLast}
&\Vert {e^{\alpha\cdot/2}}\delta Y^{\mathbb G} \Vert_{\mathbb{D}_{T\wedge\tau}(\widetilde{Q},\mathbb{G})}+C_1\Vert{e^{\alpha\cdot/2}} \delta Y^{\mathbb{G}}\Vert_{\mathbb{S}_{T\wedge\tau}(\widetilde{Q},p)}+C_2\Vert{e^{\alpha\cdot/2}} \delta Z^{\mathbb{G}}\Vert_{\mathbb{S}_{T\wedge\tau}(\widetilde{Q},p)}+C_3\Vert{e^{\alpha(\tau\wedge\cdot)/2}} \is \delta {M^{\mathbb{G}}}\Vert_{{\cal M}^{p}(\widetilde{Q})}\nonumber\\
& \leq {C}_4 \Vert{e^{\alpha\cdot/2}} \delta f\Vert_{\mathbb{S}_{T\wedge\tau}(\widetilde{Q},p)}+C_5\Vert {e^{\alpha(T\wedge\tau)/2}} \delta \xi\Vert_{\mathbb{L}^{p}(\widetilde{Q})}+{C}_6\Vert {e^{\alpha\cdot/2}} \delta S\Vert_{\mathbb{D}_{T\wedge\tau}(\widetilde{Q},p)}\nonumber\\
&+{C}_7\sqrt{\Vert {e^{\alpha\cdot/2}} \delta S\Vert_{\mathbb{D}_{T\wedge\tau}(\widetilde{Q},p)}\Vert \rm{Var}_T(e^{\alpha\cdot/2}\is\delta K^{\mathbb G}) \Vert_{\mathbb{L}^{p}(\widetilde{Q})}},
\end{align}
where ${C}_i$, $i=1,...,7$ are given by  (\ref{Constants4estimates}). 
Then here we take $\epsilon=2/{\alpha}$,  $\epsilon_1=(3-\sqrt{8})/9\sqrt{2}$ and $\alpha>\alpha_1(p)$,  and remark that $0<{C}_2\leq\min({C}_1,{C}_3)$. Furthermore, in virtue of Theorem \ref{WhyNot} with $\sigma=0$, we get 
$$\Vert \rm{Var}_T(e^{\alpha\cdot/2}\is\delta K^{\mathbb G}) \Vert_{\mathbb{L}^{p}(\widetilde{Q})}\leq \Vert(e^{\alpha\cdot/2}\is K^{\mathbb G,1})_T\Vert_{\mathbb{L}^{p}(\widetilde{Q})}+\Vert({e}^{\alpha\cdot/2}\is{K}^{\mathbb G,2})_T\Vert_{\mathbb{L}^{p}(\widetilde{Q})}\leq \widehat{C}\sum_{i=1}^2\Delta(f^{(i)}, S^{(i)}, \xi^{(i)}).$$
Therefore, by plugging this inequality in (\ref{BeforeLast}), the inequality (\ref{MainDeltaInequality}) follows immediately with
$$
\widehat{C}_1={{ {C}_4}\over{ {C}_2}},\quad  \widehat{C}_2={{ {C}_5}\over{{C}_2}},\quad \widehat{C}_3={{ {C}_6}\over{ {C}_2}},\quad  \widehat{C}_4={{ {C}_7\sqrt{\widehat{C}}}\over{ {C}_2}}.
$$
It is also clear that $\widehat{C}_1$ goes to zero when $\alpha$ goes to infinity. This ends the proof of the theorem. 
\end{proof}
\subsection{Existence, uniqueness and relationship to $\mathbb F$-RBSDEs}\label{GeneralRBSDEfromG2F}
In this subsection, we elaborate our results on the existence and uniqueness of the solution to (\ref{nonLinear}), and describe the form of its $\mathbb F$-RBSDE counter part. To this end, we assume that there exists $\alpha>\max(\alpha_0(p),\alpha_1(p))$ such that 
\begin{eqnarray}\label{MainAssumption4NonlinearBounded}
&&E\left[{\widetilde{\cal E}}_T{\cal K}_T(f,S,h)+\int_0^T {\cal K}_s(f,S,h) dV^{\mathbb F}_s\right]<+\infty,\end{eqnarray}
where 
\begin{eqnarray}\label{CalKprocess}
{\cal K}_t(f,S,h):=\vert{h}_t\vert^p+\left(\int_0^{t}\vert f(s,0,0)\vert^2 ds\right)^{p/2}+\sup_{0\leq u\leq{t}}(S_u^+)^p,\quad{t}\geq 0.\end{eqnarray}

One of the main obstacles, herein, lies in guessing the form of the $\mathbb F$-RBSDE that corresponds to (\ref{nonLinear}).  To overcome this challenge, we appeal to the linear case and the known method of approximating the solution to the general RBSDE (\ref{nonLinear}) by the sequence of solutions to linear RBSDEs --as it is adopted in \cite{Bouchard} and the references therein--. This is the aim of the following remark. 
 \begin{remark} Following the footsteps of \cite{Bouchard} and the main stream of BSDE literature, we define the sequence of linear RBSDEs under $\mathbb G$, whose solutions approximate the solution to the general RBSDE (\ref{nonLinear}). Thus, we consider the sequence $\left(Y^{\mathbb{G},n},Z^{\mathbb{G},n},M^{\mathbb{G},n},K^{\mathbb{G},n}\right)$ defined recursively as follows.
 \begin{eqnarray*}
 && (Y^{\mathbb{G},0},Z^{\mathbb{G},0},M^{\mathbb{G},0},K^{\mathbb{G},0}):=(0,0,0,0),\\
 &&\mbox{for any}\ n\geq 1,\quad  \left(Y^{\mathbb{G},n},Z^{\mathbb{G},n},M^{\mathbb{G},n},K^{\mathbb{G},n}\right)\ \mbox{is the unique solution to}:\\
 &&\begin{cases}
 Y_{t}=\xi+\displaystyle\int_{t\wedge\tau}^{T\wedge\tau}f(s,Y_{s}^{\mathbb{G},n-1}, Z_{s}^{\mathbb{G},n-1})ds+\int_{t\wedge\tau}^{T\wedge\tau}dM_s+\int_{t\wedge\tau}^{T\wedge\tau}dK_{s}-\int_{t\wedge\tau}^{T\wedge\tau}Z_{s}dW_{s},\\
 Y\geq S\quad \mbox{on}\ \Rbrack0,\tau\Lbrack,\quad\displaystyle\int_{0}^{T\wedge\tau}(Y_{t-}-S_{t-})dK_{t}=0.\end{cases}
 \end{eqnarray*}
 Thus, from this recursive sequence of solutions, and thanks to the linear part fully analyzed in Sections \ref{LinearboundedSection} and \ref{LinearUnboundedSection}, we obtain a sequence of RBSDEs under $\mathbb F$ and their solutions. This can be achieved by determining $\left(Y^{\mathbb{F},n},Z^{\mathbb{F},n},K^{\mathbb{F},n}\right)$ associated to $\left(Y^{\mathbb{G},n},Z^{\mathbb{G},n},M^{\mathbb{G},n},K^{\mathbb{G},n}\right)$ for each $n\geq 0$ as follows.
 \begin{enumerate}
 \item As ($Y^{\mathbb{G},0},Z^{\mathbb{G},0},M^{\mathbb{G},0},K^{\mathbb{G},0}$):=($0,0,0,0$), then we get ($Y^{\mathbb{F},0},Z^{\mathbb{F},0},K^{\mathbb{F},0}$):=($0,0,0$).
 \item For $n=1$,  $\left(Y^{\mathbb{G},1},Z^{\mathbb{G},1},M^{\mathbb{G},1},K^{\mathbb{G},1}\right)$ is the solution to 
 \begin{eqnarray*}
Y_{t}=\xi+\int_{t\wedge\tau}^{T\wedge\tau}f(s,0,0)ds+\int_{t\wedge\tau}^{T\wedge\tau}dM_{s}+\int_{t\wedge\tau}^{T\wedge\tau}dK_{s}-\int_{t\wedge\tau}^{T\wedge\tau}Z_{s}dW_{s}.
\end{eqnarray*}
Here, the generator/driver is constant in ($Y,Z, M, K$), and hence in virtue of Theorem \ref{abcde} there exists a unique $(Y^{\mathbb{F},1},Z^{\mathbb{F},1},K^{\mathbb{F},1})$ solution to the RBSDE (\ref{RBSDEF}) with generator/driver $f^{\mathbb{F},1}(s):={\widetilde{\cal E}}_{s}f(s,0,0)$ and
 \begin{eqnarray}\label{EqualityYG2YF1}
   Y^{\mathbb{G},1}= \frac{Y^{\mathbb{F},1}}{{\widetilde{\cal E}}}I_{\Lbrack0,\tau\Lbrack}+\xi1_{\Lbrack\tau,+\infty\Lbrack},\ 
  Z^{\mathbb{G},1}=\frac{Z^{\mathbb{F},1}}{{\widetilde{\cal E}}_{-}},\ 
   K^{\mathbb{G},1}=\frac{1}{{\widetilde{\cal E}}_{-}}\is K ^{\mathbb{F},1},\ 
      M^{\mathbb{G},1}=\left(h-\frac{Y^{\mathbb{F},1}}{{\widetilde{\cal E}}}\right)\is N^{\mathbb{G}}.
       \end{eqnarray}
\item For $n=2$, $\left(Y^{\mathbb{G},2},Z^{\mathbb{G},2},M^{\mathbb{G},2},K^{\mathbb{G},2}\right)$ is the solution to 
 \begin{align*}
Y_{t}&=\xi+\int_{t\wedge\tau}^{T\wedge\tau}f(s,Y_{s}^{\mathbb{G},1}, Z_{s}^{\mathbb{G},1})ds+\int_{t\wedge\tau}^{T\wedge\tau}dM_{s}+\int_{t\wedge\tau}^{T\wedge\tau}dK_{s}-\int_{t\wedge\tau}^{T\wedge\tau}Z_{s}dW_{s}.
\end{align*}
Thus, by plugging (\ref{EqualityYG2YF1}) in this equation, we obtain
 \begin{align*}
&Y_{t}=\xi+\int_{t\wedge\tau}^{T\wedge\tau}f(s,\frac{Y_{s}^{\mathbb{F},1}}{{\widetilde{\cal E}}_{s}}, \frac{Z_{s}^{\mathbb{F},1}}{{\widetilde{\cal E}}_{s-}})ds+\int_{t\wedge\tau}^{T\wedge\tau}dM_{s}+\int_{t\wedge\tau}^{T\wedge\tau}dK_{s}-\int_{t\wedge\tau}^{T\wedge\tau}Z_{s}dW_{s}.
\end{align*}
The generator here does not depend on $(Y,Z, M, K)$. Hence, again, Theorem \ref{abcde} yields the existence of a unique $(Y^{\mathbb{F},2},Z^{\mathbb{F},2},K^{\mathbb{F},2})$ solution to the RBSDE (\ref{RBSDEF}) under $\mathbb F$ with generator/driver
 $
 f^{\mathbb{F},2}(s):={\widetilde{\cal E}}_{s}f\left(s,{Y_{s}^{\mathbb{F},1}/{\widetilde{\cal E}}_{s}}, {Z_{s}^{\mathbb{F},1}/{\widetilde{\cal E}}_{s-}}\right),$
 and 
 \begin{eqnarray*}
   Y_{t}^{\mathbb{G},2}= \frac{Y_{t}^{\mathbb{F},2}}{{\widetilde{\cal E}}_{t}}1_{\{t\ <\tau\}}+\xi1_{\{t\geq\tau\}},\ 
  Z^{\mathbb{G},2}=\frac{Z^{\mathbb{F},2}}{{\widetilde{\cal E}}_{-}},\quad 
   K^{\mathbb{G},2}=\frac{1}{{\widetilde{\cal E}}_{-}}\is K ^{\mathbb{F},2},\ \mbox{and}\ 
      M^{\mathbb{G},2}=\left(h-\frac{Y^{\mathbb{F},2}}{{\widetilde{\cal E}}}\right)\is N^{\mathbb{G}}.
       \end{eqnarray*}
\item By iterating this procedure, we get the sequence $\left(Y^{\mathbb{F},n},Z^{\mathbb{F},n},K^{\mathbb{F},n}\right)$ defined recursively as follows.
 \begin{eqnarray*}
 && (Y^{\mathbb{F},0},Z^{\mathbb{F},0},K^{\mathbb{F},0}):=(0,0,0,0),\\
 &&
Y^{\mathbb{F},n}_{t}= \displaystyle\xi^{\mathbb{F}}+\int_{t}^{T}f^{\mathbb{F}}(s,Y^{{\mathbb{F}},n-1}_{s},Z^{{\mathbb{F}},n-1}_{s})ds+\int_{t}^{T}h_{s}dV^{\mathbb{F}}_{s}+K^{\mathbb{F},n}_{T}-K^{\mathbb{F},n}_{t}-\int_{t}^{T}Z^{\mathbb{F},n}_{s}dW_{s},\\
&&Y^{\mathbb{F},n}_{t}\geq S_{t}^{\mathbb{F}}1_{\{t\ <T\}}+\xi^{\mathbb{F}}1_{\{t\ =T\}},\quad\displaystyle\int_{0}^{T}(Y^{\mathbb{F},n}_{t-}-S_{t-}^{\mathbb{F}})dK^{\mathbb{F},n}_{t}=0,\end{eqnarray*}
 \end{enumerate}
where $f^{\mathbb{F}}(s,y,z):={\widetilde{\cal E}_s}f\left(s,y({\widetilde{\cal E}_s})^{-1},z({\widetilde{\cal E}_s})^{-1}\right)$. Thus, thanks to the convergence of $\left(Y^{\mathbb{G},n},Z^{\mathbb{G},n},M^{\mathbb{G},n},K^{\mathbb{G},n}\right)$ and the relationship (\ref{secondrelation}), we deduce that $\left(Y^{\mathbb{F},n},Z^{\mathbb{F},n},K^{\mathbb{F},n}\right)$ should also converge to $ \left(Y^{\mathbb{F}},Z^{\mathbb{F}},K^{\mathbb{F}}\right)$, and this triplet satisfies \begin{eqnarray*}\label{RBSDEFsequence}
\begin{cases}
Y_{t}= \displaystyle\xi^{\mathbb{F}}+\int_{t}^{T}f^{\mathbb F}(s,Y_s, Z_{s})ds+\int_{t}^{T}h_{s}dV^{\mathbb{F}}_{s}+K_{T}-K_{t}-\int_{t}^{T}Z_{s}dW_{s},\\
Y_{t}\geq S_{t}^{\mathbb{F}}1_{\{t\ <T\}}+\xi^{\mathbb{F}}1_{\{t\ =T\}},\quad\displaystyle\int_{0}^{T}(Y_{t-}-S_{t-}^{\mathbb{F}})dK_{t}=0.\end{cases}\end{eqnarray*}
 This gives us the RBSDE under $\mathbb F$ that we are looking for, and this also shows the importance of analyzing the linear case separately besides its own importance. 
\end{remark}
In the following, we elaborate our main result on how to connect RBSDE in $\mathbb G$ with those in $\mathbb F$.

\begin{theorem}\label{alkd} Suppose $G>0$ and both (\ref{LipschitzAssumption})and (\ref{MainAssumption4NonlinearBounded}) hold. Then the following assertions hold.\\
{\rm{(a)}} The following RBSDE under $\mathbb F$, associated to the triplet $(S^{\mathbb{F}}, \xi^{\mathbb{F}},f^{\mathbb F})$,
 \begin{eqnarray}\label{RBSDEFGENERAL}
\begin{cases}
Y_{t}= \displaystyle\xi^{\mathbb{F}}+\int_{t}^{T}f^{\mathbb{F}}(s,Y_{s},Z_{s})ds+\int_{t}^{T}h_{s}dV^{\mathbb{F}}_{s}+K_{T}-K_{t}-\int_{t}^{T}Z_{s}dW_{s},\\
\\
Y_{t}\geq S_{t}^{\mathbb{F}},\quad t\in[0,T),\quad \displaystyle\int_{0}^{T}(Y_{t-}-S_{t-}^{\mathbb{F}})dK_{t}=0,
\end{cases}
\end{eqnarray} 
  has a unique $L^p(P,\mathbb F)$-solution that we denote by $(Y^{\mathbb{F}},Z^{{\mathbb{F}}},K^{\mathbb{F}})$, where
\begin{eqnarray}\label{Data4RBSDE(F)}
f^{\mathbb{F}}(s,y,z):={\widetilde{\cal E}}_{s}f\left(s,y{\widetilde{\cal E}}_{s}^{-1},z{\widetilde{\cal E}}_{s}^{-1}\right),\quad S^{\mathbb{F}}:={\widetilde{\cal E}}S,\quad \xi^{\mathbb{F}}:={\widetilde{\cal E}}_{T}h_{T},\quad\mbox{and}\quad {\widetilde{\cal E}}:={\cal E}(-{\widetilde{G}}^{-1}\is D^{o,\mathbb F}).
\end{eqnarray}
{\rm{(b)}} There exists a unique solution to (\ref{nonLinear}), denoted by ($Y^{\mathbb{G}},Z^{\mathbb{G}},M^{\mathbb{G}},K^{\mathbb{G}}$), and is given by  
\begin{eqnarray}
   Y^{\mathbb{G}}= \frac{Y^{\mathbb{F}}}{{\widetilde{\cal E}}}I_{\Lbrack0, \tau\Lbrack}+\xi I_{\Lbrack\tau,+\infty\Lbrack},\quad
  Z^{\mathbb{G}}=\frac{Z^{\mathbb{F}}}{{\widetilde{\cal E}}},\quad K^{\mathbb{G}}=\frac{1}{{\widetilde{\cal E}_{-}}}\is (K ^{\mathbb{F}})^{\tau}\quad\mbox{and}\quad 
      M^{\mathbb{G}}=\left(h-\frac{Y^{\mathbb{F}}}{{\widetilde{\cal E}}}\right)\is N^{\mathbb{G}}.\label{secondrelationn}
       \end{eqnarray}
     \end{theorem}

\begin{proof} This is divided into two steps, where we prove assertions (a) and (b) respectively.\\
{\bf Step 1.} On the one hand, put 
\begin{eqnarray*}
\widetilde{f}^{\mathbb F}(t,y,z):=f^{\mathbb{F}}(t,y-(h\is{V}^{\mathbb{F}})_t,z),\quad \widetilde{S}^{\mathbb F}:=S^{\mathbb F}+h\is{V}^{\mathbb{F}},\quad\mbox{and}\quad \widetilde{\xi}^{\mathbb F}:={\xi}^{\mathbb F}+(h\is{V}^{\mathbb{F}})_T,
\end{eqnarray*}
and remark that $(\overline{Y}, \overline{Z}, \overline{K})$ is a solution to (\ref{RBSDEFGENERAL}) if and only if $(Y', Z', K'):=(\overline{Y}+h\is{V}^{\mathbb{F}}, \overline{Z}, \overline{K})$ is a solution to the following RBSDE
\begin{eqnarray}\label{RBSDEftilde}
\begin{cases}
Y_{t}= \displaystyle\widetilde{\xi}^{\mathbb{F}}+\int_{t}^{T}\widetilde{f}^{\mathbb{F}}(s,Y_{s},Z_s)ds+K_{T}-K_{t}-\int_{t}^{T}Z_{s}dW_{s},\\
Y_{t}\geq \widetilde{S}_{t}^{\mathbb{F}},\quad t\in[0,T),\quad \displaystyle\int_{0}^{T}(Y_{t-}-\widetilde{S}_{t-}^{\mathbb{F}})dK_{t}=0.
\end{cases}\end{eqnarray}
On the other hand, thanks to (\ref{MainAssumption4NonlinearBounded}), we derive 
\begin{align*}
&\Vert\widetilde{\xi}^{\mathbb F}\Vert_{L^p(P)}\leq \Vert{\xi}^{\mathbb F}\Vert_{L^p(P)}+\Vert(\vert{h}\vert\is{V}^{\mathbb{F}})_T\Vert_{L^p(P)}\leq  \Vert{\xi}^{\mathbb F}\Vert_{L^p(P)}+E\left[(\vert{h}\vert^p\is{V}^{\mathbb{F}})_T\right]^{1/p}<+\infty,\\
&\Vert \widetilde{f}^{\mathbb F}(\cdot,0,0)\Vert_{\mathbb{S}_{T}(P,p)}\leq \Vert {f}^{\mathbb F}(\cdot,0,0)\Vert_{\mathbb{S}_{T}(P,p)}+C_{Lip}\Vert(\vert{h}\vert\is{V}^{\mathbb{F}})_T\Vert_{L^p(P)}<+\infty, \\
&\Vert (\widetilde{S}^{\mathbb F})^+\Vert_{\mathbb{D}_{T}(P,p)}\leq \Vert ({S}^{\mathbb F})^+\Vert_{\mathbb{D}_{T}(P,p)}+\Vert(\vert{h}\vert\is{V}^{\mathbb{F}})_T\Vert_{L^p(P)}<+\infty.
\end{align*}
 Therefore,  by combining these inequalities  and \cite[Theorem 3.1]{Bouchard}, we conclude that (\ref{RBSDEftilde}) has a unique solution.This ends the first part. \\
{\bf Step 2.} Here we prove assertion (b). To this end, we remark that due to Theorem \ref{uniquenessNonlinear} the RBSDE (\ref{nonLinear}) has at most one solution. Thus, the proof of assertion (b) will follows immediately as soon as we prove that the quadruplet $(\overline{Y},\overline{Z}, \overline{K}, \overline{M})$, give by 
\begin{eqnarray*}
\overline{Y}:= \frac{Y^{\mathbb{F}}}{{\widetilde{\cal E}}}I_{\Lbrack0, \tau\Lbrack}+\xi I_{\Lbrack\tau,+\infty\Lbrack},\quad
  \overline{Z}:=\frac{Z^{\mathbb{F}}}{{\widetilde{\cal E}}},\quad \overline{K}:=\frac{1}{{\widetilde{\cal E}_{-}}}\is (K ^{\mathbb{F}})^{\tau}\quad\mbox{and}\quad 
      \overline{M}:=\left(h-\frac{Y^{\mathbb{F}}}{{\widetilde{\cal E}}}\right)\is N^{\mathbb{G}},
\end{eqnarray*}
is  in fact a solution to (\ref{nonLinear}). The proof of this latter fact mimics exactly Step 2 in the proof of Theorem \ref{Relationship4InfiniteBSDE}, and will be omitted.   This ends the proof the theorem.
\end{proof}
\section{Stopped general RBSDE: The case of unbounded  horizon}
In this section, we study the following RBSDE, 
\begin{eqnarray}\label{nonLinearINFINITE}
\begin{cases}
dY_{t}=-f(t,Y_{t},Z_{t})d(t\wedge\tau)-d(K_{t\wedge\tau}+M_{t\wedge\tau})+Z_{t}dW_{t}^{\tau},\\
Y_{\tau}=\xi,\quad Y_{t}\geq S_{t},\quad 0\leq t< \tau,\quad{E}\left[\displaystyle\int_{0}^{\tau}(Y_{t-}-S_{t-})dK_{t}\right]=0,
\end{cases}
\end{eqnarray}
where $(\xi, S, f)$ is such that $S$ is an $\mathbb F$-adapted and RCLL process, $f(t,y,z)$ is a $\mbox{Prog}(\mathbb F)\times {\cal B}(\mathbb R)\times {\cal B}(\mathbb R)$-measurable functional satisfying (\ref{LipschitzAssumption}) and $\xi\in L^2({\cal G}_{T\wedge\tau})$ such that there exists an $\mathbb F$-optional process $h$ such that $\xi=h_{\tau}$. This section has three subsections. The first subsection derives estimates and stability inequalities that controls the solutions under the probability $P$ instead of $\widetilde Q$.  The second subsection introduces the RBSDE under $\mathbb F$ and discusses the existence and uniqueness of its solution, while the third subsection elaborate our principal results that solves (\ref{nonLinearINFINITE}) and discusses its properties. 
\subsection{Estimate under $P$ for the solution of (\ref{nonLinear})}
This subsection extends Theorem \ref{estimates} and \ref{estimates1} to the case of general driver/generator $f$. These theorems, that give estimates for the solutions under $P$ instead, are based essentially on Theorems \ref{WhyNot} and \ref{uniquenessNonlinear} respectively, and represent an important step towards solving (\ref{nonLinearINFINITE}).
\begin{theorem}\label{estimates4GeneralUnbounded} If ($Y^{\mathbb{G}},Z^{\mathbb{G}},M^{\mathbb{G}},K^{\mathbb{G}}$)  is a  solution to the RBSDE (\ref{nonLinear}), associated to $\left(f, S, \xi\right)$, then for $p>1$ and $\alpha$ large there exists a constant $C({\alpha},p)>0$ that depends on $\alpha$ and $p$ only such that 


\begin{eqnarray*}
&&\Vert {e^{\alpha\cdot/2}}(\widetilde{\cal E})^{1/p}Y^{\mathbb G} \Vert_{\mathbb{D}_{T\wedge\tau}(P,p)}+\Vert{e^{\alpha(\tau\wedge\cdot)/2}} (\widetilde{\cal E})^{1/p}Y^{\mathbb{G}}\Vert_{\mathbb{S}_{T\wedge\tau}(P,p)}+\Vert{e^{\alpha(\tau\wedge\cdot)/2}} (\widetilde{\cal E}_{-})^{1/p}Z^{\mathbb{G}}\Vert_{\mathbb{S}_{T\wedge\tau}(P,p)}\\
&&+\Vert{e^{\alpha\cdot/2}} I_{\Rbrack0,T\wedge\tau\Rbrack}(\widetilde{\cal E}_{-})^{1/p}\is{M^{\mathbb{G}}}\Vert_{{\cal M}^{p}(P,\mathbb{G})}+\Vert \int_{0}^{T\wedge\tau}e^{\frac{\alpha}{2} s}(\widetilde{\cal E}_{s-})^{1/p}d K_{s}^{{\mathbb{G}}}\Vert_{L^p(P)}\\
&&\leq C({\alpha},p)\left\{\Vert {e^{\alpha(T\wedge\tau)/2}}\xi\Vert_{\mathbb{L}^{p}(\widetilde{Q})}+\Vert{e^{\alpha\cdot/2}} {f}(\cdot,0,0)\Vert_{\mathbb{S}_{T\wedge\tau}(\widetilde{Q},p)}+\Vert\sup_{0\leq t\leq\cdot}e^{\alpha t/p}S^{+}_{t}\Vert_{\mathbb{S}_{T\wedge\tau}(\widetilde{Q},p)}\right\},
\end{eqnarray*}
where $\widetilde{\cal E}$ is defined in (\ref{ProcessVFandXiF}) and that we recall herein
\begin{eqnarray}\label{Xitilde}
\widetilde{\cal E}_{t}:={\cal E}_{t}(-{\widetilde{G}}^{-1}\is D^{o,\mathbb{F}}).\end{eqnarray}\end{theorem}
\begin{proof} The proof relies essentially on Lemma \ref{technicallemma1} and Theorem \ref{WhyNot}.\\
In fact, a direct application of Lemma \ref{technicallemma1}-(a)  to $Y:=e^{\alpha(\cdot\wedge T\wedge\tau)/p}Y^{\mathbb{G}}$ yields
\begin{eqnarray}\label{Control4YGgeneralCase}
E\left[\sup_{0\leq s\leq{T}\wedge\tau }e^{p\alpha{s}/2}\widetilde{\cal E}_{s}\vert{Y}^{{\mathbb{G}}}_{s}\vert^p\right]\leq G_0^{-1} E^{\widetilde{Q}}\left[\sup_{0\leq s\leq{T}\wedge\tau }e^{p\alpha{s}/2}\vert{Y}^{{\mathbb{G}}}_{s}\vert^p\right].
\end{eqnarray}
By applying Lemma \ref{technicallemma1}-(b) to both cases when $K=\int_{0}^{\cdot}e^{\alpha s}\vert Z^{{\mathbb{G}}}_{s}\vert^{2}ds$ and when $K=\int_{0}^{\cdot}e^{\alpha s}\vert Y^{{\mathbb{G}}}_{s}\vert^{2}ds$ afterwards  with $a=2/p$, we get 
\begin{eqnarray}\label{Control4ZGgeneralCase}
\begin{cases}
E\left[\left(\int_{0}^{T\wedge\tau}e^{\alpha s}(\widetilde{\cal E}_{s})^{2/p}\vert Z^{{\mathbb{G}}}_{s}\vert^{2}ds\right)^{p/2}\right]\leq {{\kappa}\over{G_0}} E^{\widetilde{Q}}\left[\left(\int_{0}^{T\wedge\tau}e^{\alpha s}\vert Z^{{\mathbb{G}}}_{s}\vert^{2}ds\right)^{p/2}\right],\\
\\
E\left[\left(\int_{0}^{T\wedge\tau}e^{\alpha s}(\widetilde{\cal E}_{s})^{2/p}\vert Z^{{\mathbb{G}}}_{s}\vert^{2}ds\right)^{p/2}\right]\leq {{\kappa}\over{G_0}} E^{\widetilde{Q}}\left[\left(\int_{0}^{T\wedge\tau}e^{\alpha s}\vert Y^{{\mathbb{G}}}_{s}\vert^{2}ds\right)^{p/2}\right].
\end{cases}
\end{eqnarray}
Similarly, we apply  Lemma \ref{technicallemma1}-(b) to $K=e^{\alpha\cdot/2}\is {K}^{\mathbb{G}}$ with $a=1/p$, we get 
\begin{align}\label{Control4KGgeneralCase}
E\left[\left(\int_{0}^{T\wedge\tau}e^{\alpha s/2}(\widetilde{\cal E}_{s})^{1/p}\vert Z^{{\mathbb{G}}}_{s}\vert^{2}d{K}^{\mathbb{G}}_s\right)^{p}\right]&\leq {{\kappa}\over{G_0}} E^{\widetilde{Q}}\left[\left(\int_{0}^{T\wedge\tau}e^{\alpha s/2}d{K}^{\mathbb{G}}_s\right)^{p}+\sum_{0<s\leq{T}\wedge\tau}\widetilde{G}_s(e^{\alpha s/2}\Delta{K}^{\mathbb{G}}_s)^p\right]\nonumber\\
&\leq {{2\kappa}\over{G_0}} E^{\widetilde{Q}}\left[\left(\int_{0}^{T\wedge\tau}e^{\alpha s/2}d{K}^{\mathbb{G}}_s\right)^{p}\right].
\end{align}
The last inequality follows from the easy facts that $\widetilde{G}\leq 1$ and $\sum_{0<s\leq{T}}(\Delta V_s)^p\leq V_T^p$ for any nondecreasing process $V$ with $V_0=0$ and any $p\geq 1$.\\
 The rest of the proof will address the term that involves the $\mathbb G$-martingale $M^{\mathbb G}$. Thus, thanks to Theorem \ref{alkd}, we know that $[M^{\mathbb G}, M^{\mathbb G}]=H'\is [N^{\mathbb G},N^{\mathbb G}]$ where $H':=(h-Y^{\mathbb F}/{\widetilde{\cal E}})^2$, which is $\mathbb F$-optional. Thus, an application of  Lemma \ref{technicallemma1}-(d) to $H:=e^{\alpha\cdot}(h-Y^{\mathbb F}/{\widetilde{\cal E}})^2$ that is $\mathbb F$-optional, we get 
\begin{align}\label{Control4NGgeneralCase}
E\left[\left((\widetilde{\cal E})^{2/p}H\is[ N^{\mathbb{G}}, N^{\mathbb{G}}]_{T\wedge\tau}\right)^{p/2}\right]&\leq {{\kappa}\over{G_0}} E^{\widetilde{Q}}\left[\left(H\is [ N^{\mathbb{G}}, N^{\mathbb{G}}]_T\right)^{p/2}+2(H^{p/2}I_{\Rbrack0,\tau\Lbrack}\is D^{o,\mathbb F})_T\right]\nonumber\\
&={{\kappa}\over{G_0}} E^{\widetilde{Q}}\left[\left(e^{\alpha\cdot} \is [M^{\mathbb{G}}, M^{\mathbb{G}}]_T\right)^{p/2}+2(H^{p/2}I_{\Rbrack0,\tau\Lbrack}\is D^{o,\mathbb F})_T\right].\end{align}
Thus, we need to control the second term in the right-hand-side of this inequality. To this end, we remark that $(H^{p/2}I_{\Rbrack0,\tau\Lbrack}\is D^{o,\mathbb F})\leq 2^{p-1}(\vert{h}e^{\alpha\cdot/2}\vert^p+\vert{Y}^{\mathbb G}e^{\alpha\cdot/2}\vert^pI_{\Rbrack0,\tau\Lbrack}\is D^{o,\mathbb F})$. Thus, by using this, we derive 
\begin{align}\label{Control4MG1}
2E^{\widetilde{Q}}\left[(H^{p/2}I_{\Rbrack0,\tau\Lbrack}\is D^{o,\mathbb F})_T\right]\leq 2^p{G_0}E^{\widetilde{Q}}\left[e^{p\alpha\tau/2}\vert{h}_{\tau}\vert^p I_{\{\tau\leq{T}\}}\right]+2^pE^{\widetilde{Q}}\left[\sup_{0\leq{t}\leq\tau\wedge{T}}e^{p\alpha{s}/2}\vert{Y}^{\mathbb G}_t\vert^p\right].
\end{align}
 Therefore, by combining this inequality with ${h}_{\tau}I_{\{\tau\leq{T}\}}=\xi I_{\{\tau\leq{T}\}}$, (\ref{Control4NGgeneralCase}), (\ref{Control4KGgeneralCase}), (\ref{Control4ZGgeneralCase}), (\ref{Control4YGgeneralCase}) and Theorem \ref{WhyNot} with $\sigma=0$, the proof of the theorem follows immediately.
 
  \end{proof}
\begin{theorem}\label{estimatesdelta} If ($Y^{\mathbb{G},i},Z^{\mathbb{G},i},K^{\mathbb{G},i}, M^{\mathbb{G},i}$)  is a  solution to the RBSDE (\ref{nonLinear}) that correspond to  $(f^{(i)}, S^{(i)}, \xi^{(i)})$, $i=1,2$ respectively, then for any $p>1$ and  $\alpha>\max(\alpha_0(p),\alpha_1(p))$, there exist positive $C_i$, $i=1,2,3,$ that depend on $(\alpha,p)$ only such that $\lim_{\alpha\to\infty}C_1=0$ and 
\begin{align}\label{Estimate4P12}
&\Vert {e^{\alpha\cdot/2}}(\widetilde{\cal E})^{1/p}\delta{Y}^{\mathbb G} \Vert_{\mathbb{D}_{T\wedge\tau}(P,p)}+\Vert{e^{\alpha\cdot/2}} (\widetilde{\cal E})^{1/p}\delta{Y}^{\mathbb{G}}\Vert_{\mathbb{S}_{T\wedge\tau}(P,p)}\nonumber\\
&+\Vert{e^{\alpha\cdot/2}} (\widetilde{\cal E}_{-})^{1/p}\delta{Z}^{\mathbb{G}}\Vert_{\mathbb{S}_{T\wedge\tau}(P,p)}+\Vert{e^{\alpha\cdot/2}} I_{\Rbrack0,T\wedge\tau\Rbrack}(\widetilde{\cal E}_{-})^{1/p}\is\delta{M}^{\mathbb{G}}\Vert_{{\cal M}^{p}(P,\mathbb{G})}\nonumber\\
& \leq C_1\Vert{e^{\alpha\cdot/2}} \delta f\Vert_{\mathbb{S}_{T\wedge\tau}(\widetilde{Q},p)}+ C_2\Vert {e^{\alpha(T\wedge\tau)/p}} \delta \xi\Vert_{\mathbb{L}^p(\widetilde{Q})}+C_3\sqrt{\Vert {e^{\alpha\cdot/2}} \delta S\Vert_{\mathbb{D}_{T\wedge\tau}(\widetilde{Q},p)}\sum_{i=1}^2 \Delta(\xi^{(i)}, f^{(i)}, S^{(i)})}.
\end{align}
Here $\Delta(\xi^{(i)}, f^{(i)}, S^{(i)})$ is 
\begin{eqnarray}\label{Deltaxi(i)}
\Delta(\xi^{(i)}, f^{(i)}, S^{(i)}):=\Vert {e^{\alpha(T\wedge\tau)/p}}\xi^{(i)}\Vert_{\mathbb{L}^p(\widetilde{Q})}+\Vert{e^{\alpha\cdot/2}}{f}^{(i)}(s,0,0)\Vert_{\mathbb{S}_{T\wedge\tau}(\widetilde{Q},p)}+\Vert\sup_{0\leq{t}\leq\cdot}e^{\alpha t/p}(S^{(i)}_{t})^{+}\Vert_{\mathbb{S}_{T\wedge\tau}(\widetilde{Q},p)}
\end{eqnarray}
and $(\delta Y^{\mathbb{G}}, \delta Z^{\mathbb{G}}, \delta M^{\mathbb{G}}, \delta K^{\mathbb{G}})$ and $(\delta f,\delta S, \delta\xi)$ are given by 
\begin{align}\label{DeltaProcesses4GeneralInfinity}
\begin{cases}
\delta Y^{\mathbb{G}}:=Y^{\mathbb{G},1}-Y^{\mathbb{G},2},\ \delta Z^{\mathbb{G}}:=Z^{\mathbb{G},1}-Z^{\mathbb{G},2}, \delta M^{\mathbb{G}}:=M^{\mathbb{G},1}-M^{\mathbb{G},2}, \delta K^{\mathbb{G}}:=K^{\mathbb{G},1}-K^{\mathbb{G},2}, \\
\delta S:=S^{(1)}-S^{(2)},\quad \delta \xi:=\xi^{(1)}-\xi^{(2)},\quad \delta f_t:= f^{(1)}(t,Y^{\mathbb{G},1}_t,Z^{\mathbb{G},1}_t)- f^{(2)}(t,Y^{\mathbb{G},1}_t,Z^{\mathbb{G},1}_t).
\end{cases}\end{align}
\end{theorem}
\begin{proof} By applying  Lemma \ref{technicallemma1}-(a) to $Y_s=e^{\alpha s/2}\delta Y^{{\mathbb{G}}}_{s}$ and $a=p$, we obtain
\begin{eqnarray}\label{Control4YGinfiniteCaseDifference}
E\left[\sup_{0\leq s\leq{T}\wedge\tau }e^{p\alpha{s}/2}\widetilde{\cal E}_{s}\vert\delta{Y}^{{\mathbb{G}}}_{s}\vert^p\right]\leq G_0^{-1} E^{\widetilde{Q}}\left[\sup_{0\leq s\leq{T}\wedge\tau }e^{p\alpha{s}/2}\vert\delta{Y}^{{\mathbb{G}}}_{s}\vert^p\right].
\end{eqnarray}
By applying  Lemma \ref{technicallemma1}-(b) to  both cases when $K=\int_{0}^{\cdot}e^{\alpha s}\vert \delta{Z}^{{\mathbb{G}}}_{s}\vert^{2}ds$ and when  $K=\int_{0}^{\cdot}e^{\alpha s}\vert \delta{Y}^{{\mathbb{G}}}_{s}\vert^{2}ds$ afterwards with $a=2/p$, we get 
\begin{eqnarray}\label{Control4ZGinfiniteCaseDifference}
\begin{cases}
E\left[\left(\int_{0}^{T\wedge\tau}e^{\alpha s}(\widetilde{\cal E}_{s})^{2/p}\vert \delta{Z}^{{\mathbb{G}}}_{s}\vert^{2}ds\right)^{p/2}\right]\leq {{\kappa}\over{G_0}} E^{\widetilde{Q}}\left[\left(\int_{0}^{T\wedge\tau}e^{\alpha s}\vert \delta{Z}^{{\mathbb{G}}}_{s}\vert^{2}ds\right)^{p/2}\right],\\
\\
E\left[\left(\int_{0}^{T\wedge\tau}e^{\alpha s}(\widetilde{\cal E}_{s})^{2/p}\vert \delta{Y}^{{\mathbb{G}}}_{s}\vert^{2}ds\right)^{p/2}\right]\leq {{\kappa}\over{G_0}} E^{\widetilde{Q}}\left[\left(\int_{0}^{T\wedge\tau}e^{\alpha s}\vert \delta{Y}^{{\mathbb{G}}}_{s}\vert^{2}ds\right)^{p/2}\right].
\end{cases}
\end{eqnarray} 
Thanks to Theorem \ref{alkd}, we know that $[\delta{M}^{\mathbb G}, \delta{M}^{\mathbb G}]=H'\is [N^{\mathbb G},N^{\mathbb G}]$ where $H':=(\delta{h}-\delta{Y}^{\mathbb F}/{\widetilde{\cal E}})^2$, which is $\mathbb F$-optional. Thus, an application of  Lemma \ref{technicallemma1}-(d) to $H_s:=e^{\alpha s}(\delta{h}-\delta{Y}^{\mathbb F}/{\widetilde{\cal E}})^2$ that is $\mathbb F$-optional, and similar argument as in (\ref{Control4MG1}), we get 
\begin{align}\label{Control4MGinfiniteCaseDifference}
&E\left[\left((\widetilde{\cal E})^{2/p}H\is[ N^{\mathbb{G}}, N^{\mathbb{G}}]_{T}\right)^{p/2}\right]\nonumber\\
&\leq {{\kappa}\over{G_0}} E^{\widetilde{Q}}\left[\left(H\is [ N^{\mathbb{G}}, N^{\mathbb{G}}]_T\right)^{p/2}+2(H^{p/2}I_{\Rbrack0,\tau\Lbrack}\is D^{o,\mathbb F})_T\right]\nonumber\\
&={{\kappa}\over{G_0}} E^{\widetilde{Q}}\left[\left(e^{\alpha\cdot} \is [\delta{M}^{\mathbb{G}}, \delta{M}^{\mathbb{G}}]_T\right)^{p/2}+2(H^{p/2}I_{\Rbrack0,\tau\Lbrack}\is D^{o,\mathbb F})_T\right]\nonumber\\
&\leq  {{\kappa}\over{G_0}} E^{\widetilde{Q}}\left[\left(e^{\alpha\cdot} \is [\delta{M}^{\mathbb{G}}, \delta{M}^{\mathbb{G}}]_T\right)^{p/2}\right]+ {{2^p\kappa}\over{G_0}} \left\{E^{\widetilde{Q}}\left[\vert\delta{h}_{\tau}\vert^p I_{\{\tau\leq{T}\}}\right]+E^{\widetilde{Q}}\left[\sup_{0\leq{t}\leq\tau\wedge{T}}e^{p\alpha{s}/2}\vert\delta{Y}^{\mathbb G}_t\vert^p\right]\right\}\end{align} 
Hence, by combining (\ref{Control4YGinfiniteCaseDifference}), (\ref{Control4ZGinfiniteCaseDifference}), (\ref{Control4MGinfiniteCaseDifference}) and Theorem  \ref{uniquenessNonlinear}, the proof of the theorem follows. \end{proof}
\subsection{Existence, uniqueness, and estimates}
This subsection elaborates our first main result of this section that proves the existence  and uniqueness of the solution to (\ref{nonLinearINFINITE}), and estimates it. 
\begin{theorem}\label{alkdINFINITE} Let $p\in (1,+\infty)$ and suppose $G>0$ and there exists $\alpha>\max(\alpha_0(p),\alpha_1(p))$ such that 
\begin{eqnarray}\label{MainAssumption4InfiniteHorizonNonlinear}
E\left[\int_0^{\infty}\left\{e^{{p\alpha}t/2}\vert{h}_t\vert^p+(F^{(\alpha)}_t)^{p}+\sup_{0\leq {u}\leq t}e^{{\alpha}u}(S_u^+)^p\right\}dV^{\mathbb F}_t\right]<+\infty,\ F^{(\alpha)}_t:=\sqrt{\int_0^te^{{\alpha}s}\vert f(s,0,0)\vert ^2 ds}.
\end{eqnarray}
Then the following assertions hold.\\
{\rm{(a)}} There exists a unique solution  ($Y^{\mathbb{G}},Z^{\mathbb{G}},M^{\mathbb{G}},K^{\mathbb{G}}$) to the RBSDE (\ref{nonLinearINFINITE}).  \\
{\rm{(b)}}  There exists $C(\alpha,p)>0$ that depends on $\alpha$ and $p$ only such that     
   \begin{align*}
&E\left[\sup_{0\leq s\leq\tau }e^{{p\alpha}s/2}\widetilde{\cal E}_{s}\vert{Y}^{{\mathbb{G}}}_{s}\vert^p+\left(\int_{0}^{\tau}e^{\alpha s}(\widetilde{\cal E}_{s})^{2/p}\vert Z^{{\mathbb{G}}}_{s}\vert^{2}ds\right)^{p/2}+\left(\int_{0}^{\tau}e^{\alpha s}(\widetilde{\cal E}_{s})^{2/p}\vert Y^{{\mathbb{G}}}_{s}\vert^{2}ds\right)^{p/2}\right]\\
&+E\left[\left(\int_{0}^{\tau}e^{{\alpha'}s}(\widetilde{\cal E}_{s-})^{1/p}d K_{s}^{{\mathbb{G}}}\right)^p+\left(\int_{0}^{\tau}e^{\alpha s}(\widetilde{\cal E}_{s-})^{2/p}d[ M^{\mathbb{G}}, M^{\mathbb{G}}]_s\right)^{p/2}\right]\\
&\leq C(\alpha,p)E\left[\int_0^{\infty} \left\{e^{{p\alpha}t/2}\vert{h}_t\vert^p +(F^{(\alpha)}_t)^{p}+\sup_{0\leq u\leq t}e^{\alpha s}(S^{+}_{u})^p\right\}dV^{\mathbb F}_t\right].
\end{align*}
  {\rm{(c)}} Let $(f, h^{(i)}, S^{(i)})$, $i=1,2$, be two triplets satisfying (\ref{MainAssumption4InfiniteHorizonNonlinear}), and $(Y^{\mathbb{G},i},Z^{\mathbb{G},i},K^{\mathbb{G},i},M^{\mathbb{G},i})$  be the solutions to their corresponding RBSDE (\ref{nonLinearINFINITE}). There exist $C_1$ and $C_2$ that depend on $\alpha$ and $p$ only such that   
   \begin{align*}
&E\left[\sup_{0\leq s\leq\tau }e^{{p\alpha}s/2}\widetilde{\cal E}_{s}\vert{\delta}Y^{{\mathbb{G}}}_{s}\vert^p+\left(\int_{0}^{\tau}e^{\alpha s}(\widetilde{\cal E}_{s-})^{2/p}\vert \delta{Z}^{{\mathbb{G}}}_{s}\vert^{2}ds\right)^{p/2}+\left(\int_{0}^{\tau}e^{\alpha s}(\widetilde{\cal E}_{s-})^{2/p}d[ {\delta}M^{\mathbb{G}},{\delta} M^{\mathbb{G}}]_s\right)^{p/2}\right]\\
&\leq C_1 E\left[\int_0^{\infty} \left\{e^{{\alpha}t}\vert{\delta}h_t\vert^p +\sup_{0\leq{u}\leq{t}}e^{\alpha{u}}\vert\delta{S}_u\vert^p\right\}dV^{\mathbb F}_t\right]\nonumber\\
&+C_2\sqrt{E\left[\int_0^{\infty}\sup_{0\leq{u}\leq{t}}e^{\alpha{u}}\vert\delta{S}_u\vert^pdV^{\mathbb F}_t\right]}\sqrt{\sum_{i=1}^2E\left[\int_0^{\infty} \left\{e^{{p\alpha}t/2}\vert{h}^{(i)}_t\vert^p +(F^{(\alpha)}_t)^p+\sup_{0\leq{u}\leq{t}}e^{\alpha{u}}\vert({S}_u^{(i)})^+\vert^p\right\}dV^{\mathbb F}_t\right]}.
\end{align*}
Here $({\delta}Y^{{\mathbb{G}}},{\delta}Z^{{\mathbb{G}}}, {\delta}K^{{\mathbb{G}}},{\delta}M^{{\mathbb{G}}})$  and $({\delta}f,{\delta}S,{\delta}h)$ are given by  (\ref{DeltaProcesses4GeneralInfinity}).
     \end{theorem}
\begin{proof} On the one hand, in virtue of assertion (c), it is clear that (\ref{nonLinearINFINITE}) has at most one solution. Thus, the rest of this proof focuses on proving the existence of the solution and assertion (b) and (c). To this end, we divid the rest of the prof into two parts.\\
{\bf Part 1.} In this part, we comnsider $(f,S,h)$  and suppose that there exists a constant $C$ such that
\begin{eqnarray}\label{BoundednessAssumpionInfinite}
\max\left(e^{p\alpha\cdot/2}\vert{h}\vert, (F^{(\alpha)})^p, \sup_{0\leq{t}\leq\cdot}e^{p\alpha\cdot}(S^+_t)^p\right)\leq C{\cal E}(G_{-}^{-1}\is{m}). 
\end{eqnarray}
The rest of this part is divided into three steps.\\
{\bf Step 1.} To the triplet $(f,S,h)$ satisfying (\ref{BoundednessAssumpionInfinite}), we associate the sequence $(\overline{f}^{(n)}, \overline{S}^{(n)},\overline{h}^{(n)})$  given by 
\begin{eqnarray}\label{overlineFn}
\overline{f}^{(n)}:=fI_{\Lbrack0,n\Rbrack},\quad  \overline{S}^{(n)}_t:=S_{n\wedge{t}},\quad \overline{h}^{(n)}_t:=h_{n\wedge{t}},\quad \overline{\xi}^{(n)}:=h_{n\wedge\tau},\quad n\geq 0.
\end{eqnarray}
Then thanks to Theorem \ref{alkd}, we deduce that for each triplet $(\overline{f}^{(n)}, \overline{S}^{(n)},\overline{\xi}^{(n)})$, the RBSDE  (\ref{nonLinearINFINITE}) has a unique solution $(\overline{Y}^{(n)}, \overline{Z}^{(n)},\overline{M}^{(n)},\overline{K}^{(n)} )$. Then by applying Theorem \ref{estimatesdelta} to the difference of solutions $(\delta{Y}, \delta{Z},\delta{M},\delta{K} ):=(\overline{Y}^{(n+m)}, \overline{Z}^{(n+m)},\overline{M}^{(n+m)},\overline{K}^{(n+m)} )-(\overline{Y}^{(n)}, \overline{Z}^{(n)},\overline{M}^{(n)},\overline{K}^{(n)} )$, and the horizon $T=n+m$, we get 
\begin{align}\label{Estimate4GenreralPoof}
&\Vert{e}^{\alpha\cdot/2}{\widetilde{\cal E}}^{1/p}\delta{Y}\Vert_{\mathbb{D}_{T}(\widetilde{P},p)}+\Vert{e}^{\alpha\cdot/2}(\widetilde{\cal E}_{-})^{1/p}\vert \delta{Z}\Vert_{\mathbb{S}^p_{T}(\widetilde{P},p)}+\Vert{e}^{\alpha\cdot/2}(\widetilde{\cal E}_{-})^{1/p}\is \delta{M}\Vert_{\mathbb{M}_{T}^p(\widetilde{P})}\nonumber\\
& \leq C_1\Vert{e^{\alpha(\tau\wedge\cdot)/p}} \delta f\Vert_{\mathbb{S}^p_{T}(\widetilde{Q})}+ C_2\Vert {e^{\alpha(T\wedge\tau)/p}} \delta \xi\Vert_{\mathbb{L}^p(\widetilde{Q})}+C_3\Vert {e^{\alpha(\tau\wedge\cdot)/2}} \delta S\Vert_{\mathbb{D}_{T}(\widetilde{Q})}^{1/2}\sup_{k\geq n}\sqrt{ \Delta(\xi^{(k)}, f^{(k)}, S^{(k)})}.
\end{align}
Here $\Delta(\xi^{(k)}, f^{(k)}, S^{(k)})$ is given by 
\begin{eqnarray}\label{Deltaxi(i)4proof}
\Delta(\xi^{(k)}, f^{(k)}, S^{(k)}):=\Vert {e^{\alpha(T\wedge\tau)/p}} \overline{\xi}^{(k)}\Vert_{\mathbb{L}^p(\widetilde{Q})}+\Vert{e^{\alpha(\tau\wedge\cdot)/p}}  \overline{f}^{(k)}(\cdot,0,0)\Vert_{\mathbb{S}_{T}(\widetilde{Q},p)} +\Vert {e^{\alpha(\tau\wedge\cdot)/2}} ( \overline{S}^{(k)})^{+}\Vert_{\mathbb{D}_{T}(\widetilde{Q},p)},
\end{eqnarray}
{\bf Step 2.} It is clear that, due to the assumption \ref{BoundednessAssumpionInfinite} and  in virtue of Lemma \ref{ExpecationQtilde2P}, we have 
\begin{eqnarray}\label{limites}
\begin{cases}
\displaystyle\lim_{n\to\infty}\sup_{m\geq 0}\Vert {e^{\alpha(\tau\wedge\cdot)/2}} \delta S\Vert_{\mathbb{D}_{T}(\widetilde{Q},p)}\leq 2\lim_{n\to\infty}\sup_{m} \Vert {e^{\alpha(\tau\wedge\cdot)/2}} SI_{\Lbrack{n},n+m\Rbrack}\Vert_{\mathbb{D}_{T}(\widetilde{Q},p)}=0\\
\displaystyle\lim_{n\to\infty}\sup_{m\geq 0}\Vert {e^{\alpha(T\wedge\tau)/p}} \delta \xi\Vert_{\mathbb{L}^p(\widetilde{Q})}=\lim_{n\to\infty}\sup_{m\geq 0}\Vert {e^{\alpha(T\wedge\tau)/p}}(h_{(n+m)\wedge\tau}-h_{n\wedge\tau})\Vert_{\mathbb{L}^p(\widetilde{Q})}=0,\\
\displaystyle\lim_{n\to\infty}\sup_{m\geq 0}\Vert {e^{\alpha(\tau\wedge\cdot)/2}} \delta f(\cdot,0,0)\Vert_{\mathbb{D}_{T}(\widetilde{Q},p)}=0,\\
\sup_{k\geq 0}\sqrt{ \Delta(\xi^{(k)}, f^{(k)}, S^{(k)})}<+\infty,
\end{cases}
\end{eqnarray}
 and 
\begin{eqnarray}\label{limitesBIS}
\begin{cases}
\displaystyle\lim_{n\to\infty}\Vert {e^{\alpha(n\wedge\tau)/p}} \overline{\xi}^{(n)}\Vert_{\mathbb{L}^p(\widetilde{Q})}^p=E\left[\int_0^{\infty} e^{\alpha{s}}\vert{h}_s\vert^p dV^{\mathbb F}_s \right]\\
\displaystyle\lim_{n\to\infty}\Vert {e^{\alpha(\tau\wedge\cdot)/p}}(\overline{S}^{(n)})^+\Vert_{\mathbb{D}_{T}(\widetilde{Q},p)}^p=E\left[\int_0^{\infty} \sup_{0\leq{s}\leq{t}}e^{\alpha{s}}(S^+_s)^p dV^{\mathbb F}_s \right]\\
\displaystyle\lim_{n\to\infty}\Vert{e^{\alpha(\tau\wedge\cdot)/p}}  \overline{f}^{(n)}(\cdot,0,0)\Vert_{\mathbb{S}_{T}(\widetilde{Q},p)}^p=E\left[\int_0^{\infty}(F^{(\alpha)}_t)^p dV^{\mathbb F}_t \right].
\end{cases}
\end{eqnarray}
Now we deal with the first term in the right-hand-side of the above inequality. To this end, on the one hand, we remark that 
\begin{eqnarray}\label{Control4deltaF}
&&\vert \delta f_t\vert=\vert (\overline{f}^{(n+m)}-\overline{f}^{(n)})(t,\overline{Y}^{(n+m)}_t, \overline{Z}^{(n+m)}_t)\vert=\vert {f}(t,\overline{Y}^{(n+m)}_t, \overline{Z}^{(n+m)}_t)\vert{I}_{\{n<t\leq n+m\}}\nonumber\\
&&\leq \vert{f}(t,0,0)\vert{I}_{\{n<t\leq n+m\}}+C_{lip}(\vert\overline{Y}^{(n+m)}_t\vert+\overline{Z}^{(n+m)}_t\vert)I_{\{n<t\leq{n+m}\}}.
\end{eqnarray}
On the other hand, thanks to Theorem \ref{WhyNot}, applied to $(\overline{Y}^{(n+m)}, \overline{Z}^{(n+m)},\overline{M}^{(n+m)},\overline{K}^{(n+m)} )$  and $\sigma=n$, we deduce that 
\begin{eqnarray*}
&&\Vert{e}^{\alpha(\tau\wedge\cdot)/2}\overline{Z}^{(n+m)}I_{\Rbrack{n},n+m\Rbrack}\Vert_{\mathbb{S}_{T}(\widetilde{Q},p)}+\Vert{e}^{\alpha\cdot/2} \overline{Y}^{(n+m)}I_{\Rbrack{n},n+m\Rbrack}\Vert_{\mathbb{S}_{T}(\widetilde{Q},p)}\\
&&\leq \widehat{C}\left\{  \Vert{e^{\alpha(T\wedge\tau)/2}}\overline{\xi}^{(n+m)}I_{\{\tau>n\}}\Vert_{L^p(\widetilde{Q})}+\Vert{e^{\alpha(\tau\wedge\cdot)}}S^{+}I_{\Rbrack{n},n+m\Rbrack}\Vert_{\mathbb{D}_T(\widetilde{Q}, p)}+\Vert e^{ {\alpha}(\tau\wedge\cdot)/2}f(\cdot,0,0)I_{\Rbrack{n},n+m\Rbrack}\Vert_{\mathbb{S}_T(\widetilde{Q}, p)}\right\}
\end{eqnarray*}
Therefore, by combining this inequality with (\ref{limites}) and (\ref{Control4deltaF}), we deduce that 
\begin{eqnarray*}
\lim_{n\to\infty}\sup_{m\geq 0} \Vert{e^{\alpha(\tau\wedge\cdot)/p}} \delta f\Vert_{\mathbb{S}_{(n+m)\wedge\tau}(\widetilde{Q},p)}=0.
\end{eqnarray*}
Therefore, a combination of this equality with (\ref{Estimate4GenreralPoof}) and (\ref{limites}), we conclude that the sequence $(\overline{Y}^{(n)}, \overline{Z}^{(n)},\overline{M}^{(n)},\overline{K}^{(n)} )$ is a Cauchy sequence in norm, and hence it converges in norm and almost surely for a subsequence, and its limite is a solution to (\ref{nonLinearINFINITE}). This proves assertion (a) of the theorem provided that assertion (c) is true. Furthermore, by applying Theorem \ref{estimates4GeneralUnbounded} to $(\overline{Y}^{(n)}, \overline{Z}^{(n)},\overline{M}^{(n)},\overline{K}^{(n)} )$, we get 
\begin{align*}
& \Vert{e}^{\alpha\cdot/2}{\widetilde{\cal E}}^{1/p}\overline{Y}^{(n)}\Vert_{\mathbb{D}({P},p)}+ \Vert{e}^{\alpha\cdot/2}(\widetilde{\cal E}_{-})^{1/p}\vert \overline{Z}^{(n)}\Vert_{\mathbb{S}({P},p)}+\Vert{e}^{\alpha\cdot/2}(\widetilde{\cal E}_{-})^{1/p}\is\overline{M}^{(n)}\Vert_{{\cal{M}}^p(P,\mathbb{G})}+\Vert{e}^{{\alpha}\cdot/2}(\widetilde{\cal E}_{-})^{1/p}\is\overline{K}^{(n)}_{T\wedge\tau}\Vert_{L^p(P)}\\
&\leq C({\alpha},p)  \left\{\Vert{e^{\alpha\cdot/p}}\overline{f}^{(n)}(\cdot,0,0)\Vert_{\mathbb{S}_{T\wedge\tau}(\widetilde{Q},p)}+\Vert {e^{\alpha(T\wedge\tau)/p}} \overline{\xi}^{(n)}\Vert_{\mathbb{L}^p(\widetilde{Q})}+\Vert {e^{\alpha\cdot/p}}(\overline{S}^{(n)})^{+}\Vert_{\mathbb{D}_{T\wedge\tau}(\widetilde{Q},p)}\right\}.
\end{align*}
Hence, by using (the convergence in norm of the sequence $(\overline{Y}^{(n)}, \overline{Z}^{(n)},\overline{M}^{(n)},\overline{K}^{(n)} )$ or the convergence almost surely and Fatou for the left-hand-side term of the above inequality and using (\ref{limites}) for its right-hand-side term, the proof of assertion (b) follows immediately.\\
{\bf Step 3.} Here we prove assertion (c) under the assumption (\ref{BoundednessAssumpionInfinite}). Consider two triplets $(f, S^{(i)}, h^{(i)})$, $i=1,2$ satisfying (\ref{BoundednessAssumpionInfinite}). Then for each triplet we associate to it a sequence $(\overline{f}^{(n)}, \overline{S}^{(n,i)}, \overline{h}^{(n,i)})$ defined via (\ref{overlineFn}). Thus, there exists two sequences  $(\overline{Y}^{(n,i)}, \overline{Z}^{(n,i)},\overline{M}^{(n,i)},\overline{K}^{(n,i)} )$, $i=1,2$, that converge in norm and almost surely for  subsequences to $(\overline{Y}^{\mathbb G,i}, \overline{Z}^{\mathbb G, i},\overline{M}^{\mathbb G, i},\overline{K}^{\mathbb G, i} )$ which is solution to (\ref{nonLinearINFINITE}) associated to  $(f, S^{(i)}, h^{(i)})$.\ Then by applying Theorem \ref{estimatesdelta} to the difference of solutions 
$$(\delta{Y}, \delta{Z},\delta{M},\delta{K} ):=(\overline{Y}^{(n,1)}, \overline{Z}^{(n,1)},\overline{M}^{(n,1)},\overline{K}^{(n,1)} )-(\overline{Y}^{(n,2)}, \overline{Z}^{(n,2)},\overline{M}^{(n,2)},\overline{K}^{(n,2)} ),$$ and the horizon $T=n$, we get 
\begin{align}\label{Estimate4GenreralPoofBis}
&\Vert{e}^{\alpha\cdot/2}{\widetilde{\cal E}}^{1/p}\delta{Y}\Vert_{\mathbb{D}_{T}(\widetilde{P},p)}+\Vert{e}^{\alpha\cdot/2}(\widetilde{\cal E}_{-})^{1/p}\vert \delta{Z}\Vert_{\mathbb{S}^p_{T}(\widetilde{P},p)}+\Vert{e}^{\alpha\cdot/2}(\widetilde{\cal E}_{-})^{1/p}\is \delta{M}\Vert_{{\cal{M}}^p(P,\mathbb{G})}\nonumber\\
& \leq C_2\Vert {e^{\alpha(T\wedge\tau)/p}} \delta \xi\Vert_{\mathbb{L}^p(\widetilde{Q})}+C_3\Vert {e^{\alpha(\tau\wedge\cdot)/2}} \delta S\Vert_{\mathbb{D}_{T}(\widetilde{Q})}^{1/2}\sup_{k\geq n}\sqrt{ \Delta(\xi^{(k)}, f^{(k)}, S^{(k)})}.
\end{align}
Here $\Delta(\xi^{(k)}, f^{(k)}, S^{(k)})$ is given by (\ref{Deltaxi(i)4proof}). Thus, by using (\ref{limitesBIS}) for each term in the right-hand-side term of the above inequality and Fatou and almost convergence for its left-hand-side term, we conclude that assertion (c) holds, and this ends the first part.\\
{\bf Part 2.}  Here we drop the assumption (\ref{BoundednessAssumpionInfinite}). Let $(f,S,h)$ be a triplet and consider 
\begin{eqnarray*}
T_n:=\inf\left\{t\geq 0\ :\ e^{\alpha{t}}\max(F^{(\alpha)}_t, \sup_{0\leq{u}\leq t}\vert S_u\vert^p)>n {\cal E}(G_{-}^{-1}\is m)\right\}.
\end{eqnarray*}
It is clear that $T_n$ is an $\mathbb F$-stopping time that converges to infinity almost surely. Then we associate  a sequence, to the triplet $(f,S,h)$, denoted by $(f^{(n)},S^{(n)},h^{(n)})$, given by 
\begin{eqnarray}\label{SequenceTn}
f^{(n)}:=fI_{\Lbrack0,T_n\Lbrack},\quad S^{(n)}:=SI_{\Lbrack0,T_n\Lbrack},\quad h^{(n)}_t:=h_tI_{\{e^{\alpha{t}}\vert{h}_t\vert^p\leq n{\cal E}_t(G_{-}^{-1}\is m)\}}I_{\Lbrack0,T_n\Lbrack}(t).
\end{eqnarray}
{\bf Step 1.} Here we prove that assertion (c) holds. In fact, we consider two triplets $(f,S_i,h_i)$, $i=1,2$, and for each tripler we associate a sequence $(f^{(n)},S^{(n,i)},h^{(n,i)})$, via (\ref{SequenceTn}), for each $i=1,2$. Therefore, for each $i=1,2$ and any $n\geq 1$,  the triplet  $(f^{(n)},S^{(n,i)},h^{(n,i)})$ fulfills (\ref{BoundednessAssumpionInfinite}), and hence due to Part 1, there exists a unique solution $(\overline{Y}^{(n,i)}, \overline{Z}^{(n,i)},\overline{M}^{(n,i)},\overline{K}^{(n,i)} )$ that converges in norma and almost surely for a subsequence to $({Y}^{\mathbb G,i}, Z^{\mathbb G,i},M^{\mathbb G,i},K^{\mathbb G,i})$. Furthermore,  we apply assertion (c)  to the difference of solutions
$$
\left({\delta}Y^{\mathbb{G},n}, \delta{Z}^{\mathbb{G},n},{\delta}M^{\mathbb{G},n} \right)=(\overline{Y}^{(n,1)}, \overline{Z}^{(n,1)},\overline{M}^{(n,1)},\overline{K}^{(n,1)} )-(\overline{Y}^{(n,2)}, \overline{Z}^{(n,2)},\overline{M}^{(n,2)},\overline{K}^{(n,2)} ),$$
 and get 
   \begin{align*}
&E\left[\sup_{0\leq s\leq\tau }e^{{p\alpha}s/2}\widetilde{\cal E}_{s}\vert{\delta}Y^{{\mathbb{G},n}}_{s}\vert^p+\left(\int_{0}^{\tau}e^{\alpha s}(\widetilde{\cal E}_{s-})^{2/p}\vert \delta{Z}^{{\mathbb{G},n}}_{s}\vert^{2}ds\right)^{p/2}+\left(\int_{0}^{\tau}e^{\alpha s}(\widetilde{\cal E}_{s-})^{2/p}d[ {\delta}M^{\mathbb{G},n},{\delta} M^{\mathbb{G},n}]_s\right)^{p/2}\right]\\
&\leq C_1 E\left[\int_0^{T_n} \left\{e^{{\alpha}t}\vert{\delta}h_t\vert^p +\sup_{0\leq{u}\leq{t}}e^{\alpha{u}}\vert\delta{S}_u\vert^p\right\}dV^{\mathbb F}_t\right]\nonumber\\
&+C_2\sqrt{E\left[\int_0^{T_n}\sup_{0\leq{u}\leq{t}}e^{\alpha{u}}\vert\delta{S}_u\vert^pdV^{\mathbb F}_t\right]}\sqrt{\sum_{i=1}^2E\left[\int_0^{T_n} \left\{e^{{p\alpha}t/2}\vert{h}_i(t)\vert^p +(F^{(\alpha)}_t)^p+\sup_{0\leq{u}\leq{t}}e^{\alpha{u}}\vert({S}_i)(u))^+\vert^p\right\}dV^{\mathbb F}_t\right]}.
\end{align*}
Therefore, by taking the limite on both sides, we deduce that assertion (c) holds. \\
{\bf Step 2.}  This step proves assertion (a) and (b) of the theorem.\\
Furthermore, for any $n\geq 0$, we associate the double-sequence as follows 
\begin{eqnarray*}
\overline{f}^{(n,k)}:=f^{(n)}I_{\Lbrack 0,k\Rbrack},\overline{S}^{(n,k)}_t:=S^{(n)}_{t\wedge{k}},\quad\overline{h}^{(n,k)}:=h^{(n)} I_{\Lbrack 0,k\Rbrack},\quad \overline{\xi}^{(n,k)}:=\overline{h}^{(n,k)}_{\tau}.
\end{eqnarray*}
Thus, on the one hand, Theorem \ref{alkd} yields the existence of $(\overline{Y}^{(n,k)}, \overline{Z}^{(n,k)},\overline{M}^{(n,k)},\overline{K}^{(n,k)} )$  solution to (\ref{nonLinearINFINITE}) associated to $(\overline{f}^{(n,k)},\overline{S}^{(n,k)},\overline{h}^{(n,k)})$. On the other hand, part 1 of this proof implies the existence of $(\overline{Y}^{(n)}, \overline{Z}^{(n)},\overline{M}^{(n)},\overline{K}^{(n)} )$  solution to (\ref{nonLinearINFINITE}) associated to $({f}^{(n)},{S}^{(n)},{h}^{(n)})$, and which is the limite of  $(\overline{Y}^{(n,k)}, \overline{Z}^{(n,k)},\overline{M}^{(n,k)},\overline{K}^{(n,k)} )$ when $k$ goes infinity in norm and almost surely for subsequence. Now, we consider nonengative integers $n,m$ and $k$ and we follow Step 2 of part 1, and apply Theorem \ref{estimatesdelta} to the solutions difference
\begin{eqnarray*}
\left(\delta Y^{(k)}, \delta Z^{(k)},\delta M^{(k)},\delta K^{(k)}\right):=(\overline{Y}^{(n+m,k)}, \overline{Z}^{(n+m,k)},\overline{M}^{(n+m,k)},\overline{K}^{(n+m,k)} )-(\overline{Y}^{(n,k)}, \overline{Z}^{(n,k)},\overline{M}^{(n,k)},\overline{K}^{(n,k)} ),
\end{eqnarray*}
and get 
\begin{align}\label{Estimate4GeneralPoofBis}
&\Vert{e}^{\alpha\cdot/2}{\widetilde{\cal E}}^{1/p}{\delta}Y^{(k)}\Vert_{\mathbb{D}_{T}(\widetilde{P},p)}+\Vert{e}^{\alpha\cdot/2}(\widetilde{\cal E}_{-})^{1/p} \delta{Z}^{(k)}\Vert_{\mathbb{S}^p_{T}(\widetilde{P},p)}+\Vert{e}^{\alpha\cdot/2}(\widetilde{\cal E}_{-})^{1/p}\is{\delta}M^{(k)}\Vert_{{\cal{M}}^p(P,\mathbb{G})}\nonumber\\
& \leq C_1\Vert{e^{\alpha(\tau\wedge\cdot)/p}} \delta f^{(k)}\Vert_{\mathbb{S}^p_{T}(\widetilde{Q})}+ C_2\Vert {e^{\alpha(T\wedge\tau)/p}} \delta \xi^{(k)}\Vert_{\mathbb{L}^p(\widetilde{Q})}+C_3\Vert {e^{\alpha(\tau\wedge\cdot)/2}} \delta S^{(k)}\Vert_{\mathbb{D}_{T}(\widetilde{Q})}^{1/2}\sup_{l\geq n}\sqrt{ \Delta(l,k)}.
\end{align}
Here $\Delta(l,k)$ is given by 
\begin{eqnarray}\label{Deltaxi(i)4proofBis}
\Delta(l,k):=\Vert {e^{\alpha(T\wedge\tau)/p}} \overline{\xi}^{(l,k)}\Vert_{\mathbb{L}^p(\widetilde{Q})}+\Vert{e^{\alpha(\tau\wedge\cdot)/p}}  \overline{f}^{(l,k)}(\cdot,0,0)\Vert_{\mathbb{S}_{T}(\widetilde{Q},p)} +\Vert {e^{\alpha(\tau\wedge\cdot)/2}} ( \overline{S}^{(l,k)})^{+}\Vert_{\mathbb{D}_{T}(\widetilde{Q},p)},
\end{eqnarray}
Here, we have 
\begin{eqnarray}\label{Control4deltaFBis}
&&\vert \delta f_t^{(k)}\vert:=\vert (\overline{f}^{(n+m,k)}-\overline{f}^{(n,k)})(t,\overline{Y}^{(n+m,k)}_t, \overline{Z}^{(n+m,k)}_t)\vert=\vert {f}(t,\overline{Y}^{(n+m,k)}_t, \overline{Z}^{(n+m,k)}_t)\vert{I}_{\{T_n<t\leq T_{n+m}\}}I_{\{t\leq k\}}\nonumber\\
&&\leq \vert{f}(t,0,0)\vert{I}_{\{T_n<t\leq T_{n+m}\}}I_{\{t\leq k\}}+C_{lip}(\vert\overline{Y}^{(n+m,k)}_t\vert+\overline{Z}^{(n+m,k)}_t\vert){I}_{\{T_n<t\leq T_{n+m}\}}I_{\{t\leq k\}}
\end{eqnarray}
On the other hand, thanks to Theorem \ref{WhyNot} applied to $(\overline{Y}^{(n+m,k)}, \overline{Z}^{(n+m,k)},\overline{M}^{(n+m,k)},\overline{K}^{(n+m,k)} )$  and $\sigma=T_n$, we deduce that 
\begin{eqnarray*}
&&\Vert{e}^{\alpha\cdot/2}\overline{Z}^{(n+m)}I_{\Rbrack{T_n},T_{n+m}\Rbrack}\Vert_{\mathbb{S}_{k\wedge\tau}(\widetilde{Q},p)}+\Vert{e}^{\alpha\cdot/2} \overline{Y}^{(n+m,k)}I_{\Rbrack{T_n},T_{n+m}\Rbrack}\Vert_{\mathbb{S}_{k\wedge\tau}(\widetilde{Q},p)}\\
&&\leq \widehat{C}\left\{  \Vert{e^{\alpha(T\wedge\tau)/2}}\overline{\xi}^{(n+m,k)}I_{\{\tau>T_n\}}\Vert_{L^p(\widetilde{Q})}+\Vert{e^{\alpha(\tau\wedge\cdot)}}S^{+}I_{\Rbrack{T_n},T_{n+m}\Rbrack}\Vert_{\mathbb{D}{k\wedge\tau}(\widetilde{Q}, p)}\right\}\\
&&+ \widehat{C}\Vert e^{ {\alpha}(\tau\wedge\cdot)/2}f(\cdot,0,0)I_{\Rbrack{T_n},T_{n+m}\Rbrack}\Vert_{\mathbb{S}_{k\wedge\tau}(\widetilde{Q}, p)}.
\end{eqnarray*}
Therefore, by combining this inequality with (\ref{Control4deltaFBis}) and Lemma \ref{ExpecationQtilde2P}, we deduce the existence of $\widehat{C}_i$, $i=1,2,3,$ that depend on $p$ and $\alpha$ only such that
\begin{align*}
&\displaystyle\lim\inf_{k\to\infty}\Vert{e^{\alpha(\tau\wedge\cdot)/p}} \delta f^{(k)}\Vert_{\mathbb{S}_{k\wedge\tau}(\widetilde{Q},p)}\\
&\leq \widehat{C}_1\lim\inf_{k\to\infty}  \Vert{e^{\alpha(\tau\wedge\cdot)/2}}{f}(\cdot,0,0){I}_{\Rbrack{T}_n,T_{n+m}\Rbrack}\Vert_{\mathbb{S}_{k\wedge\tau}(\widetilde{Q},p)}+\displaystyle\widehat{C}_2\lim\inf_{k\to\infty}\Vert{e^{\alpha(k\wedge\tau)/2}}\overline{\xi}^{(n+m,k)}I_{\{\tau>T_n\}}\Vert_{L^p(\widetilde{Q})}\\
&\displaystyle+\widehat{C}_3\lim\inf_{k\to\infty}\Vert{e^{\alpha(\tau\wedge\cdot)}}S^{+}I_{\Rbrack{T_n},T_{n+m}\Rbrack}\Vert_{\mathbb{D}_{k\wedge\tau}(\widetilde{Q}, p)}\\
&\leq \max(\widehat{C}_1,\widehat{C}_2,\widehat{C}_3) E\left[\int_{T_n}^{T_{n+m}}\left\{(F^{(\alpha)}_t)^p+\sup_{0\leq{u}\leq{t}}e^{\alpha{t}}(S_u^+)^p+e^{\alpha{t}}\vert{h_t}\vert^p\right\}dV^{\mathbb F}_t\right].
\end{align*}
As a result, in virtue of (\ref{MainAssumption4InfiniteHorizonNonlinear}), the above inequality implies that 
\begin{eqnarray}\label{LimitZeroGeneral}
\lim_{n\to\infty}\sup_{m\geq 1}\lim\inf_{k\to\infty}\Vert{e^{\alpha(\tau\wedge\cdot)/p}} \delta f^{(k)}\Vert_{\mathbb{S}_{k\wedge\tau}(\widetilde{Q},p)}=0.
\end{eqnarray}
Then by taking the limit when  $k$ goes to infinity in (\ref{Estimate4GeneralPoofBis}) and using Lemma \ref{ExpecationQtilde2P} again, we get 
\begin{align*}
&\Vert{e}^{\alpha\cdot/2}{\widetilde{\cal E}}^{1/p}(\overline{Y}^{(n+m)}-\overline{Y}^{(n)})\Vert_{\mathbb{D}(P,p)}+\Vert{e}^{\alpha\cdot/2}(\widetilde{\cal E}_{-})^{1/p}(\overline{Z}^{(n+m)}-\overline{Z}^{(n)})\Vert_{\mathbb{S}^p(P,p)}\\
&+\Vert{e}^{\alpha\cdot/2}(\widetilde{\cal E}_{-})^{1/p}\is(\overline{M}^{(n+m)}-\overline{M}^{(n)})\Vert_{{\cal{M}}^p(P,\mathbb{G})}\nonumber\\
& \leq C_1\lim\inf_{k\to\infty}\Vert{e^{\alpha(\tau\wedge\cdot)/p}} \delta f^{(k)}\Vert_{\mathbb{S}_{k\wedge\tau}(\widetilde{Q},p)}+ C_2E\left[\int_{T_n}^{T_{n+m}}e^{\alpha{t}}\vert{h_t}\vert^pdV^{\mathbb F}_t\right]\\
&+C_3\sqrt{E\left[\int_{T_n}^{T_{n+m}}\sup_{0\leq{u}\leq{t}}e^{\alpha{t}}\vert{S}_t\vert^pdV^{\mathbb F}_t\right]}\sqrt{E\left[\int_0^{\infty}\left\{(F^{(\alpha)}_t)^p+\sup_{0\leq{u}\leq{t}}e^{\alpha{t}}(S_t^+)^p+e^{\alpha{t}}\vert{h_t}\vert^p\right\}dV^{\mathbb F}_t\right]}.\end{align*}
Therefore, by combining this with (\ref{LimitZero}), we conclude that $(\overline{Y}^{(n)},\overline{Z}^{(n)},\overline{M}^{(n)})$ is a Cauchy sequence, and hence it converges to $(Y^{\mathbb G}, Z^{\mathbb G},M^{\mathbb G})$ in norm and almost surely for a subsequence. Then the convergence of $\overline{K}^{(n)}$ to $K^{\mathbb G}$ follows immediately, and hence  $(Y^{\mathbb G}, Z^{\mathbb G},M^{\mathbb G},K^{\mathbb G})$ is a solution (\ref{nonLinearINFINITE}). This proves assertion (a). Assertion (b) follows from from part 1 of this proof,  that guarantees that assertion (b) holds for each $(\overline{Y}^{(n)},\overline{Z}^{(n)},\overline{M}^{(n)},\overline{K}^{(n)})$ ($n\geq 0$),  and from taking the limit afterwards in the obtained inequality. This ends the proof of theorem.\end{proof}
\subsection{An RBSDE under $\mathbb F$ with infinite horizon and its relationship to (\ref{nonLinearINFINITE}) }
In this subsection, we derive our second main result of this section that  addresses the RBSDE under $\mathbb F$ given below, and connects to (\ref{nonLinearINFINITE}).
 \begin{eqnarray}\label{RBSDEFGENERALInfinite}
\begin{cases}
Y_{t}=\displaystyle\int_{t}^{\infty}f^{\mathbb{F}}(s,Y_s,Z_s)ds+\int_{t}^{\infty}h_{s}dV^{\mathbb{F}}_{s}+K^{\mathbb{F}}_{\infty}-K_t-\int_{t}^{\infty}Z_{s}dW_{s},\\
Y^{\mathbb{F}}_{t}\geq S_{t}^{\mathbb{F}},\quad t\geq 0,\quad \displaystyle{E}\left[\int_{0}^{\infty}(Y_{t-}-S_{t-}^{\mathbb{F}})dK^{\mathbb{F}}_{t}\right]=0.
\end{cases}
\end{eqnarray} 
Here  $(f^{\mathbb{F}},S^{\mathbb{F}}, {\widetilde{\cal E}})$ denote the functionals defined via (\ref{Data4RBSDE(F)}). First of all, remark that a solution to the above RBSDE is any triplet $(Y, Z, K)$ such that $ \lim_{t\to\infty}Y_t$ exists almost surely and is null, and 
\begin{eqnarray*}
\begin{cases}
dY_{t}=f^{\mathbb{F}}(t,Y_t,Z_t)dt-h_tdV^{\mathbb{F}}_t-dK_t+Z_tdW_t,\\
Y^{\mathbb{F}}_{t}\geq S_{t}^{\mathbb{F}},\quad t\geq 0,\quad \displaystyle{E}\left[\int_{0}^{\infty}(Y_{t-}-S_{t-}^{\mathbb{F}})dK^{\mathbb{F}}_{t}\right]=0.
\end{cases}
\end{eqnarray*} 
This RBSDE generalizes Hamad\`ene et al . \cite{Hamadane1} in many aspects. First of all, our obstacle process $S^{\mathbb{F}}$ is arbitrary RCLL and might not be continuous at all. Furthermore, we do not exige that the part $(Y,K)$ of the solution to be continuous. Besides these, our RBSDE  has an additional term, $\int_0^{\cdot}h_{s}dV^{\mathbb{F}}_{s}$ that might not be absolutely continuous with respect to the Lebesgue measure.
\begin{theorem}\label{RBSDEinfite4F} Let $p\in (1,+\infty)$, $(h,S)$ be a pair of $\mathbb F$-optional processes, $f$ is a functional satisfying (\ref{LipschitzAssumption}), and  $(f^{\mathbb{F}},S^{\mathbb{F}}, {\widetilde{\cal E}})$ is given by (\ref{Data4RBSDE(F)}). Suppose $G>0$  and there exists $\alpha>\max(\alpha_0(p),\alpha_1(p))$  
\begin{eqnarray}\label{additiona4GenertalINfinity}
\mbox{(\ref{MainAssumption4InfiniteHorizonNonlinear}) holds and}\quad E\left[\left(\widetilde{\cal E}_{\infty}F^{(\alpha)}_{\infty}\right)^p\right]<+\infty.
\end{eqnarray}
Then the RBSDE (\ref{RBSDEFGENERALInfinite}) has a unique $L^p(P,\mathbb F)$-solution $(Y^{\mathbb{F}}, Z^{\mathbb{F}},K^{\mathbb F})$.
\end{theorem}

The proof of this theorem is based on the following lemma
\begin{lemma}\label{Inequality4BSDEunderF} For $\alpha>\max(\alpha_0(P),\alpha_1(p))$, there exist $C_i$, $i=1,2,3,4$ that depend on $\alpha$ and $p$ only such that $C_1C_{lip}<1$ and the following assertions hold.\\
{\rm{(a)}} If $(Y^{i}, Z^i, K^i)$ is a solution to the RBSDE (\ref{RBSDEsaid}) associated to $(f^i,S^i, \xi^i)$, $i=1,2$, then 
\begin{eqnarray}
\Vert\delta{Y}\Vert_{\mathbb{D}(P,p)}+\Vert\delta{Z}\Vert_{\mathbb{S}(P,p)}\leq C_1\Vert\delta{f}\Vert_{\mathbb{S}(P,p)}+C_2\Vert\delta\xi\Vert_{L^p(P)}+C_3\sqrt{\Vert\delta{S}\Vert_{\mathbb{S}(P,p)}}\sqrt{\Vert\mbox{Var}_T(\delta{K})\Vert_{L^p(P)}}.\end{eqnarray}
{\rm{(b)}} If $(Y, Z, K)$ is a solution to the RBSDE (\ref{RBSDEsaid}), then 
\begin{eqnarray}
\Vert{Y}\Vert_{\mathbb{D}(P,p)}+\Vert{Z}\Vert_{\mathbb{S}(P,p)}+\Vert{Y}\Vert_{\mathbb{S}(P,p)}+\Vert{K}_T\Vert_{L^p(P)}\leq C_4\left\{\Vert{f}(\cdot, 0,0)\Vert_{\mathbb{S}(P,p)}+\Vert\xi\Vert_{L^p(P)}+\Vert{S^+}\Vert_{\mathbb{S}(P,p)}\right\}.\end{eqnarray}
\end{lemma}
The proof of this lemma mimics those of Theorems \ref{WhyNot} and \ref{uniquenessNonlinear}, and will be omitted here.

\begin{proof}[Proof of Theorem \ref{RBSDEinfite4F}] Remark that due to the assumption (\ref{MainAssumption4InfiniteHorizonNonlinear}), the nondecreasing process $U:=\int_0^{\cdot}{h_s}{d}V^{\mathbb F}_s$ has a limit at infinity. Put
\begin{eqnarray*}
\widetilde{f}^{\mathbb{F}}(s,y,z)=f^{\mathbb{F}}(s,y-U_s,z),\quad \widetilde{S}^{\mathbb{F}}:=S^{\mathbb{F}}+U,\quad\mbox{and}\quad \widehat{\xi}:=U_{\infty}=\int_0^{\infty}h_sdV^{\mathbb F}_s.
\end{eqnarray*}
Then $(\overline{Y}, \overline{Z},\overline{K})$ is a solution to (\ref{RBSDEFGENERALInfinite}) if and only if  $(Y',Z',K'):=(\overline{Y}+U, \overline{Z},\overline{K})$ is a solution to 
 \begin{eqnarray}\label{RBSDEsaid}
\begin{cases}
Y_{t}= \widehat{\xi}+\displaystyle\int_{t}^{\infty}\widetilde{f}^{\mathbb{F}}(s,Y_s,Z_s)ds+K^{\mathbb{F}}_{\infty}-K_t-\int_{t}^{\infty}Z_{s}dW_{s},\\
\\
Y^{\mathbb{F}}_{t}\geq \widetilde{S}_{t}^{\mathbb{F}},\quad t\geq 0,\quad \displaystyle{E}\left[\int_{0}^{\infty}(Y_{t-}-S_{t-}^{\mathbb{F}})dK^{\mathbb{F}}_{t}\right]=0,
\end{cases}
\end{eqnarray}

Now, we define the sequence $(Y^{(n)}, Z^{(n)}, K^{(n)})$ as follows: 
$(Y^{(0)}, Z^{(0)}, K^{(0)}):=(0, 0, 0)$, and $(Y^{(n)}, Z^{(n)}, K^{(n)})$ is the unique solution to
\begin{eqnarray}\label{linear4n}
Y^{(n)}_t=\xi+\int_t^{\infty}{\widetilde{ f}}^{\mathbb F}(s, Y^{(n-1)}_s, Z^{(n-1)}_s)ds +\int_t^{\infty} Z^{(n)}_s dW_s+K^{(n)}_{\infty}-K^{(n)}_t.
\end{eqnarray} 
The existence and uniqueness of this solution is guaranteed by Theorem \ref{Relationship4InfiniteBSDE}. Thus, by applying Lemma \ref{Inequality4BSDEunderF}  to  $(Y^{(i)}, Z^{(i)}, K^{(i)})$ and  $(\delta{Y},\delta{Z},\delta{K}):=\left(Y^{(n+m)}-Y^{(n)}, Z^{(n+m)}-Z^{(n)}, K^{(n+m)}-K^{(n)}\right)$, we deduce that 
\begin{eqnarray*}
&&\sup_{i\geq 0}\Vert(Y^{(i)}, Z^{(i)}, K^{(i)})\Vert:=\sup_{i\geq 0} \left\{\Vert{Y^{(i)}}\Vert_{\mathbb{D}(P,p)}+\Vert{Z^{(i)}}\Vert_{\mathbb{S}(P,p)}+\Vert{Y}\Vert_{\mathbb{S}(P,p)}+\Vert{K^{(i)}}_T\Vert_{L^p(P)}\right\}<+\infty\quad\mbox{and}\\
&&\Vert\delta{Y}\Vert_{\mathbb{D}(P,p)}+\Vert\delta{Z}\Vert_{\mathbb{S}(P,p)}=:\Vert(\delta{Y},\delta{ Z})\Vert\leq C_1\Vert\delta{f}\Vert\leq C_1C_{lip}\Vert(Y^{(n+m-1)}-Y^{(n-1)}, Z^{(n+m-1)}- Z^{(n-1)})\Vert.
\end{eqnarray*} 
Thus, by iterating this inequality, we get 
\begin{eqnarray*}
\Vert(Y^{(n+m)}-Y^{(n)}, Z^{(n+m)}- Z^{(n)})\Vert\leq (C_1C_{lip})^n \Vert(Y^{(m)}, Z^{(m)})\Vert\leq (C_1C_{lip})^n \sup_{i\geq 0}\Vert(Y^{(i)}, Z^{(i)}, K^{(i)})\Vert.\end{eqnarray*}
This proves that the sequence $(Y^{(n)}, Z^{(n)})$ is a Cauchy, and hence it convergences in norm and almost surely for a subsequence to $(Y, Z)$. Then the convergence of $ K^{(i)})$ to some $K$ follows immediately from the RBSDE (\ref{RBSDEsaid}), and therefore the triplet $(Y, Z, K)$ is a solution to (\ref{RBSDEsaid}). The uniqueness of the solution to (\ref{RBSDEsaid}) is a direct consequence of Lemma \ref{Inequality4BSDEunderF}-(a). This ends the proof of the theorem.
\end{proof}
Below, we establish the relationship between the solution of (\ref{nonLinearINFINITE})   and that of (\ref{RBSDEFGENERALInfinite}). 
\begin{theorem}\label{Relatiuonship4GeneralINifity}
Suppose that the assumptions of Theorem \ref{RBSDEinfite4F} hold. Then both RBSDEs (\ref{RBSDEFGENERALInfinite})  and  (\ref{nonLinearINFINITE}) have unique solutions $(Y^{\mathbb{F}}, Z^{\mathbb{F}},K^{\mathbb F})$ and $(Y^{\mathbb{G}}, Z^{\mathbb{G}},K^{\mathbb G},M^{\mathbb G})$ respectively, and they satisfy  
\begin{eqnarray}
   Y^{\mathbb{G}}= \frac{Y^{\mathbb{F}}}{{\widetilde{\cal E}}}I_{\Lbrack0,\tau\Lbrack}+\xi I_{\Lbrack\tau,+\infty\Lbrack},\ 
  Z^{\mathbb{G}}=\frac{Z^{\mathbb{F}}}{{\widetilde{\cal E}}}I_{\Rbrack0,\tau\Rbrack},\  K^{\mathbb{G}}=\frac{1}{{\widetilde{\cal E}_{-}}}\is (K ^{\mathbb{F}})^{\tau}\quad\mbox{and}\quad 
      M^{\mathbb{G}}=\left(h-\frac{Y^{\mathbb{F}}}{{\widetilde{\cal E}}}\right)\is N^{\mathbb{G}}.\label{secondrelation4Generalinfinity}
       \end{eqnarray}
\end{theorem}
\begin{proof}  Thanks to Theorems \ref{RBSDEinfite4F}  and \ref{alkdINFINITE}, it is clear that both RBSDEs (\ref{RBSDEFGENERALInfinite})  and  (\ref{nonLinearINFINITE}) have unique solutions. This proves the first claim of the theorem, while the proof of (\ref{secondrelation4Generalinfinity}) follows immediately as soon as we prove that $(\overline{Y}, \overline{Z},\overline{M},\overline{K})$ given by 
\begin{eqnarray*}
 \overline{Y}:= \frac{Y^{\mathbb{F}}}{{\widetilde{\cal E}}}I_{\Lbrack0,\tau\Lbrack}+\xi I_{\Lbrack\tau,+\infty\Lbrack},\ 
  \overline{Z}:=\frac{Z^{\mathbb{F}}}{{\widetilde{\cal E}}}I_{\Rbrack0,\tau\Rbrack},\  \overline{K}:=\frac{1}{{\widetilde{\cal E}_{-}}}\is (K ^{\mathbb{F}})^{\tau}\quad\mbox{and}\quad 
      \overline{M}:=\left(h-\frac{Y^{\mathbb{F}}}{{\widetilde{\cal E}}}\right)\is N^{\mathbb{G}}.\end{eqnarray*}
is a solution to (\ref{nonLinearINFINITE}). This latter fact can be proved by following exactly the footsteps of Step 2 in the proof of Theorem \ref{Relationship4InfiniteBSDE}. This ends the proof of theorem.
\end{proof}
We end this section with elaborating the BSDE version of this general case with unbounded horizon.
\begin{theorem}\label{BSDEinfini} Let $p\in (1,+\infty)$, $h$ an $\mathbb F$-optional process and $f(t,y,z)$ be a functional satisfying (\ref{LipschitzAssumption}). Suppose that $G>0$ and there exists $\alpha>\max(\alpha_0(p),\alpha_1(p))$ such that 
\begin{eqnarray}\label{MainAssumption4BSDEinfini}
E\left[\left(\widetilde{\cal E}_{\infty}F^{(\alpha)}_{\infty}\right)^p+\int_0^{\infty}\left\{e^{{p\alpha}t/2}\vert{h}_t\vert^p+(F^{(\alpha)}_t)^p\right\}dV^{\mathbb F}_t\right]<+\infty.
\end{eqnarray}
Then the following assertions hold.\\
{\rm{(a)}} There exists a unique solution  ($Y^{\mathbb{G}},Z^{\mathbb{G}},M^{\mathbb{G}},K^{\mathbb{G}}$) to the following BSDE
\begin{eqnarray}\label{BSDEinifiteDynamics}
dY_{t}=-f(t,Y_{t},Z_{t})d(t\wedge\tau)-dM_{t\wedge\tau}+Z_{t}dW_{t\wedge\tau},\quad{Y}_{\tau}=\xi=h_{\tau}.
\end{eqnarray}
{\rm{(b)}} For $(f^{\mathbb{F}},S^{\mathbb{F}}, {\widetilde{\cal E}})$ are defined via (\ref{Data4RBSDE(F)}), the following BSDE under $\mathbb F$ \begin{eqnarray}\label{BSDEFInfinite}
Y_{t}=\int_{t}^{\infty}f^{\mathbb{F}}(s,Y_s,Z_s)ds+\int_{t}^{\infty}h_{s}dV^{\mathbb{F}}_{s}-\int_{t}^{\infty}Z_{s}dW_{s},\end{eqnarray} 
has a unique solution that we denote by $(Y^{\mathbb{F}}, Z^{\mathbb{F}})$.\\
{\rm{(c)}} The two solutions $(Y^{\mathbb{G}},Z^{\mathbb{G}},M^{\mathbb{G}})$ and ($Y^{\mathbb{F}},Z^{\mathbb{F}}$) satisfy 
\begin{eqnarray}
   Y^{\mathbb{G}}= \frac{Y^{\mathbb{F}}}{{\widetilde{\cal E}}}I_{\Lbrack0,\tau\Lbrack}+\xi I_{\Lbrack\tau,+\infty\Lbrack},\ 
  Z^{\mathbb{G}}=\frac{Z^{\mathbb{F}}}{{\widetilde{\cal E}}}I_{\Rbrack0,\tau\Rbrack},\quad\mbox{and}\quad 
      M^{\mathbb{G}}=\left(h-\frac{Y^{\mathbb{F}}}{{\widetilde{\cal E}}}\right)\is N^{\mathbb{G}}.\label{secondrelationn}
       \end{eqnarray}
  {\rm{(d)}} There exists $C(\alpha,p)>0$ that depends on $\alpha$ and $p$ only such that     
   \begin{align*}
&E\left[\sup_{0\leq s\leq\tau }e^{{p\alpha}s/2}\widetilde{\cal E}_{s}\vert{Y}^{{\mathbb{G}}}_{s}\vert^p+\left(\int_{0}^{\tau}e^{\alpha s}(\widetilde{\cal E}_{s})^{2/p}\vert Z^{{\mathbb{G}}}_{s}\vert^{2}ds\right)^{p/2}+\left(\int_{0}^{\tau}e^{\alpha s}(\widetilde{\cal E}_{s-})^{2/p}d[ M^{\mathbb{G}}, M^{\mathbb{G}}]_s\right)^{p/2}\right]\\
&\leq C(\alpha,p) E\left[\int_0^{\infty} \left\{e^{{p\alpha}t/2}\vert{h}\vert^p_t +(F^{(\alpha)}_t)^p\right\}dV^{\mathbb F}_s\right].
\end{align*}
  {\rm{(e)}} Let $(f^{(i)}, h^{(i)})$, $i=1,2$, be two pairs satisfying (\ref{MainAssumption4BSDEinfini}), and $(Y^{\mathbb{G},i},Z^{\mathbb{G},i},M^{\mathbb{G},i})$  be the solutions to their corresponding BSDE (\ref{BSDEinifiteDynamics}). There exist $C(\alpha, p)$ that depends on $\alpha$ and $p$ only such that     
   \begin{align*}
&E\left[\sup_{0\leq s\leq\tau }e^{{p\alpha}s/2}\widetilde{\cal E}_{s}\vert{\delta}Y^{{\mathbb{G}}}_{s}\vert^p+\left(\int_{0}^{\tau}e^{\alpha s}(\widetilde{\cal E}_{s})^{2/p}\vert \delta{Z}^{{\mathbb{G}}}_{s}\vert^{2}ds\right)^{p/2}+\left(\int_{0}^{\tau}e^{\alpha s}(\widetilde{\cal E}_{s-})^{2/p}d[ {\delta}M^{\mathbb{G}},{\delta} M^{\mathbb{G}}]_s\right)^{p/2}\right]\\
&\leq C(\alpha, p) E\left[\int_0^{\infty} \left\{e^{{p\alpha}t/2}\vert{\delta}h_t\vert^p +(\delta{F}^{(\alpha)}_t)^p+\sup_{0\leq{u}\leq{t}}e^{\alpha{u}}\vert\delta{S}_u\vert^p\right\}dV^{\mathbb F}_t\right].
\end{align*}
     \end{theorem}
\begin{proof} Similarly, as in the proof of Theorem \ref{LinearBSDEcase}, we notice that in general a BSDE is a particular case of an RBSDE. In fact a BSDE is an RBSDE with $S\equiv -\infty$, and  this yields to having the predictable part with finite variation, in its solution, $K\equiv 0$.  Thus, by taking these into account, we remark that when $(S, K)\equiv (-\infty, 0)$, then (\ref{MainAssumption4InfiniteHorizonNonlinear}) reduces to (\ref{MainAssumption4BSDEinfini}), and the theorem follows from combining Theorem \ref{RBSDEinfite4F}  and Theorem \ref{alkdINFINITE} .\end{proof}
\begin{appendices}
\section{Some martingale inequalities}
We start this section by recalling an important theorem form martingale inequalities that goes back to Dellacherie and Meyer, see \cite[Th\'eor\`eme 99, Chapter VI]{DELLACHERIE}. 
\begin{theorem}\label{DellacherieAndMeyer}
Consider a complete filtered probability space $\left(\Omega, {\cal F}, \mathbb H=({\cal H}_t)_{0\leq t\leq T}, P\right)$. Let $A$ be predictable (optional) increasing process whose potential (left potential) Z is bounded above by a cadlag martingale $M_{t}=E[M_{\infty}|\mathcal{H}_{t}]$. Then 
\begin{align}
\Vert{ A}_{\infty}\Vert_{\Phi}\leq p_{\Phi}\Vert{M}_{\infty}\Vert_{\Phi},
\end{align}
where $p_{\Phi}$ is the constant associated with $\Phi$ and $\Phi$ is increasing  convex function defined as the following;
\begin{equation*}
\Phi(t):=\int_{0}^{t}\phi(s)ds,\quad p_{\Phi}:=\sup_{t}{{t\phi(t)}\over{\Phi(t)}}.
\end{equation*}
for some right continuous increasing function $\phi$ which is positive on $\mathbb{R}^{+}$.
\end{theorem}
The following lemma, that plays crucial role in our estimations, ia interesting in itself and  generalizes \cite[Lemma 4.8]{Choulli4} . 
\begin{lemma}\label{Lemma4.8FromChoulliThesis}
If $r^{-1}=a^{-1}+b^{-1},$ where $ a>1$ and $b>1$, then there exists a positive constant $\kappa=\kappa(a,b)$ depending only on $a$ and $b$  such that the following assertion holds.\\
For any triplet $(H, X, M)$ such that $H$ is predictable, $X$ is RCLL and adapted process, $M$ is a martingale,  and $\vert{H}\vert \leq \vert{X_{-}}\vert$, the following inequality holds. 
\begin{equation*}
\Vert \sup_{0\leq{t}\leq{T}}\vert(H\is M)_t\vert\Vert_r\leq \kappa\Vert\sup_{0\leq{t}\leq{T}}\vert{X}_t\vert\Vert_a\Vert[M]_T^{\frac{1}{2}}\Vert_b.
\end{equation*}
\end{lemma}
\begin{proof} When $H=X_{-}$, the assertion can be found in \cite[Lemma 4.8]{Choulli4}. To prove the general case, we remark that, there is no loss of generality in assuming $\vert X_{-}\vert>0$, and hence the process $H/X_{-}$ is a well defined process that is predictable and is bounded by one. Therefore, we put 
\begin{eqnarray*}
\overline{M}:={{H}\over{X_{-}}}\is M,\end{eqnarray*}
and remark that $[\overline{M},\overline{M}]=(H/X_{-})^2\is [M, M]\leq [M, M]$. Thus, we derive 
\begin{align*}
\Vert\sup_{0\leq{t}\leq{T}}\vert({H}\is M)_t\vert\Vert_r&=\Vert\sup_{0\leq{t}\leq{T}}\vert({X_{-}}\is \overline{M})_t\vert\Vert_r\leq \kappa \Vert\sup_{0\leq{t}\leq{T}}\vert{X}_t\vert\Vert_a\Vert[\overline{M}]_T^{\frac{1}{2}}\Vert_b\\
&\leq \kappa\Vert\sup_{0\leq{t}\leq{T}}\vert{X}_t\vert\Vert_a\Vert[M, M]_T^{\frac{1}{2}}\Vert_b.\end{align*}
This ends the proof of the lemma.
\end{proof}
\section{Proof of Lemmas \ref{stoppingTimeLemma}, \ref{Solution2SnellEnvelop}, \ref{L/EpsilonTilde}, \ref{Lemma4.11}, \ref{ExpecationQtilde2P} and \ref{technicallemma1} 
}\label{Appendix4Proofs}
\begin{proof}[Proof of Lemma \ref{stoppingTimeLemma}]
Thanks to \cite{DellacherieMeyer92}, for our $\mathbb{G}$-stopping time $\sigma^{\mathbb{G}}$, there exists an  $\mathbb{F}$-stopping time $\sigma$ such that 
\begin{equation*}
\sigma^{\mathbb{G}}=\sigma^{\mathbb{G}}\wedge\tau=\sigma\wedge\tau.
\end{equation*}

 Put 
 \begin{equation}\label{sigmaFdefinition}
 \sigma^{\mathbb{F}}:=\min\left(\max(\sigma,\sigma_{1}),\sigma_{2}\right),
  \end{equation}
  and on the one hand remark that $ \sigma^{\mathbb{F}}$ is an $\mathbb{F}$- stopping time satisfying the first condition in (\ref{sigmaF}). On the other hand, it is clear that 
  \begin{eqnarray*}
  \min(\tau, \max(\sigma, \sigma_1))=(\tau\wedge\sigma_1)I_{\{\sigma_1>\sigma\}}+(\tau\wedge\sigma)I_{\{\sigma_1\leq\sigma\}}=\max(\tau\wedge\tau, \sigma_1\wedge\tau).\end{eqnarray*}
  Thus, by using this equality, we derive
  \begin{eqnarray*}
   \sigma^{\mathbb{F}}\wedge\tau&&= \tau\wedge\sigma_2\wedge\max(\sigma,\sigma_1)=(\tau\wedge\sigma_2)\wedge(\tau\wedge\max(\sigma,\sigma_1))\\
   &&=(\tau\wedge\sigma_2)\wedge\max(\tau\wedge\tau, \sigma_1\wedge\tau)=\sigma\wedge\tau=\sigma^{\mathbb G}.
   \end{eqnarray*}
   This ends the proof of the lemma. \end{proof}
\begin{proof}[Proof of Lemma \ref{Solution2SnellEnvelop}]
Let $\nu\in\mathcal{J}_{t\wedge\tau}^{T\wedge\tau}(\mathbb{G})$. By using (\ref{RBSDEG}) and by taking the conditional expectation under $\widetilde{Q}$ afterwards, we get
\begin{align}
Y_{t\wedge\tau}&=E^{\widetilde{Q}}\left[\int_{t\wedge\tau}^{\nu\wedge\tau}f(s)ds+Y_{\nu\wedge\tau}+K_{\nu\wedge\tau}-K_{t\wedge\tau}\ \Big|\ \mathcal{G}_{t}\right]\nonumber\\
&\geq E^{\widetilde{Q}}\left[\int_{t\wedge\tau}^{\nu\wedge\tau}f(s)ds+S_{\nu\wedge\tau}1_{\{\nu\ <\tau\wedge T\}}+\xi1_{\{\nu\ =\tau\wedge T\}}\ \Big|\ \mathcal{G}_{t}\right]\nonumber\\
&\geq\mbox{ess}\sup_ {\theta\in \mathcal{J}_{t\wedge\tau,T\wedge\tau}(\mathbb{G})}E^{\widetilde{Q}}\left[\int_{t\wedge\tau}^{\theta}f(s)ds+S_{\theta}1_{\{\theta<\tau\wedge T\}}+\xi{I}_{\{\theta=\tau\wedge T\}}\ \Big|\ \mathcal{G}_{t}\right]\label{sameUntilHere}
   \end{align}
   To prove the reverse inequality, we consider  the following sequence of stopping times 
   \begin{eqnarray*}
   \theta_{n}:=\inf\left\{t\wedge\tau\leq u\leq T\wedge\tau;\quad Y_{u}<S_{u}+\frac{1}{n}\right\}\wedge (T\wedge\tau),\quad n\geq 1.
 \end{eqnarray*}
 Then it is clear that $\theta_{n}\in\mathcal{J}_{t\wedge\tau}^{T\wedge\tau}(\mathbb{G})$, and 
  \begin{eqnarray*}
 Y-S\geq\frac{1}{n}\quad  \mbox{on }\quad  \Lbrack t\wedge\tau ,\theta_{n}\Lbrack,\quad\mbox{and}\quad 
Y_{-}-S_{-}\geq\frac{1}{n}\quad  \mbox{on }\quad \Rbrack t\wedge\tau ,\theta_{n}\Rbrack. 
   \end{eqnarray*}
As a result, we get $I_{ \Rbrack t\wedge\tau ,\theta_{n}\Rbrack}\is K\equiv 0$, and hence using (\ref{RBSDEG}) again we deduce that
      \begin{align*}
   Y_{t\wedge\tau}&=Y_{\theta_{n}}+\int_{t\wedge\tau}^{\theta_{n}}f(s)ds+\int_{t\wedge\tau}^{\theta_{n}}d(K+M)_{t\wedge\tau}-\int_{t\wedge\tau}^{\theta_{n}}Z_{s}dW_{t}^{\tau}\\
   &=Y_{\theta_{n}}+\int_{t\wedge\tau}^{\theta_{n}}f(s)ds+\int_{t\wedge\tau}^{\theta_{n}}dM_{t\wedge\tau}-\int_{t\wedge\tau}^{\theta_{n}}Z_{s}dW_{t}^{\tau}.
   \end{align*}
  By taking conditional expectation under $\widetilde{Q}$, we get   $Y_{t\wedge\tau}=E^{\widetilde{Q}}[Y_{\theta_{n}}+\int_{t\wedge\tau}^{\theta_{n}}f(s)ds|\mathcal{G}_{t}]$ that implies 
     \begin{align*}
  &  \mbox{ess}\sup_{\theta\in \mathcal{J}_{t\wedge\tau,T\wedge\tau}(\mathbb{G})}E^{\widetilde{Q}}\left[\int_{t\wedge\tau}^{\theta}f(s)ds+S_{\theta}1_{\{\theta\ <\tau\wedge T\}}+\xi{I}_{\{\theta=\tau\wedge T\}}\ \Big|\ \mathcal{G}_{t}\right]\\
     &\geq E^{\widetilde{Q}}\left[\int_{t\wedge\tau}^{\theta_{n}}f(s)ds+S_{\theta_{n}}1_{\{\theta_{n}<\tau\wedge T\}}+\xi1_{\{\theta_{n}=\tau\wedge T\}}\ \Big|\ \mathcal{G}_{t}\right]       =    Y_{t\wedge\tau}+E^{\widetilde{Q}}\left[(S_{\theta_{n}}-Y_{\theta_{n}})1_{\{\theta_{n}<\tau\wedge T\}} \Big|\ \mathcal{G}_{t}\right]\\
   &    \geq Y_{t\wedge\tau}-\frac{1}{n}\widetilde{Q}(\theta_{n}<\tau\wedge T|\mathcal{G}_{t}).
    \end{align*} 
    Thus, by letting $n$ to go to infinity and due $\widetilde{Q}(\theta_{n}<\tau\wedge T\big| {\cal G}_{t})\leq 1$, we get
    \begin{equation*}
     \underset{\nu\in \mathcal{J}_{t,T}(\mathbb{G})}{\mbox{esssup}}\hspace{2mm}E^{\widetilde{Q}}\left[\int_{t\wedge\tau}^{\nu\wedge\tau}f(s)ds+S_{\nu\wedge\tau}1_{\{\nu\ <\tau\wedge T\}}+\xi1_{\{\nu=\tau\wedge T\}}\ \Big|\ \mathcal{G}_{t}\right]\geq  Y_{t\wedge\tau}.
    \end{equation*}
By combining this inequality with (\ref{sameUntilHere}), we get (\ref{RBSDE2Snell}), and the proof of the lemma completed.\end{proof}
\begin{proof}[Proof of Lemma \ref{L/EpsilonTilde}] 
 Let $L$ be an $\mathbb F$-semimartingale. Then we derive
 \begin{align*}
& {{L}\over{\widetilde{\cal E}}}I_{\Lbrack0,\tau\Lbrack}= {{L^{\tau}}\over{\widetilde{\cal E}^{\tau}}}- {{L}\over{\widetilde{\cal E}}}\is D=L\is {1\over{\widetilde{\cal E}^{\tau}}}+{1\over{\widetilde{\cal E}_{-}}}\is L^{\tau}- {{L}\over{\widetilde{\cal E}}}\is D= {{L}\over{G\widetilde{\cal E}_{-}}}I_{\Rbrack0,\tau\Rbrack}\is{D}^{o,\mathbb F} +{1\over{\widetilde{\cal E}_{-}}}\is L^{\tau}- {{L}\over{\widetilde{\cal E}}}\is D\\
&=  {{L}\over{\widetilde{G}\widetilde{\cal E}}}I_{\Rbrack0,\tau\Rbrack}\is{D}^{o,\mathbb F} +{1\over{\widetilde{\cal E}_{-}}}\is L^{\tau}- {{L}\over{\widetilde{\cal E}}}\is D=- {{L}\over{\widetilde{\cal E}}}\is{N}^{\mathbb G}+{1\over{\widetilde{\cal E}_{-}}}\is L^{\tau}.
\end{align*}
 The fourth equality follows from the fact that $\widetilde{\cal E}=\widetilde{\cal E}_{-}G/\widetilde{G}$. This ends the proof of the lemma.
\end{proof}

\begin{proof}[Proof of Lemma \ref{Lemma4.11}] Recall that $\Delta m={\widetilde G}-G_{-}\leq 1$, and $m$ is a BMO $(\mathbb F, P)$-martingale. Furthermore, we have 
       \begin{align*}
       &E^{\widetilde Q}\left[[m,m]_{T\wedge\tau}-[m,m]_{t\wedge\tau}\big|{\cal G}_t\right]= E\left[\int_{t\wedge\tau}^{T\wedge\tau}{\cal E}_s(G_{-}^{-1}\is m)^{-1}d[m,m]_s\big|{\cal G}_t\right]{\cal E}_{t\wedge\tau}(G_{-}^{-1}\is m)\\
       &=E\left[\int_{t\wedge\tau}^{T\wedge\tau}{\cal E}_s(G_{-}^{-1}\is m)^{-1}d[m,m]_s\big|{\cal F}_t\right]{{{\cal E}_t(G_{-}^{-1}\is m)}\over{G_t}}I_{\{\tau>t\}}\\
       &=E\left[\int_{t}^T{\widetilde{\cal E}}_sd[m,m]_s\big|{\cal F}_t\right]{1\over{{\widetilde{\cal E}}_t}}I_{\{\tau>t\}}\leq \Vert m\Vert_{BMO(P)}.
          \end{align*}
Hence, assertion (a) follows  from this latter inequality. Thanks to Lemma \ref{G-projection},  on $(\tau>s)$ we derive 
 \begin{align*}
     E^{\widetilde{Q}}\left[D^{o,\mathbb{F}}_{T} -D^{o,\mathbb{F}}_{s-}\big|{\cal G}_{s}\right]&=\Delta D^{o,\mathbb{F}}_{s}+ E\left[\int_{s\wedge\tau}^{T\wedge\tau} {1\over{{\cal E}_u(G^{-1}_{-}\is m)}} d D^{o,\mathbb{F}}_u \big|{\cal G}_{s}\right]{\cal E}_{s\wedge\tau}(G^{-1}\is m)\\
     &=E\left[\int_{s\wedge\tau}^{T\wedge\tau} {1\over{{\cal E}_u(G^{-1}_{-}\is m)}} d D^{o,\mathbb{F}}_u \big|{\cal F}_{s}\right]{{{\cal E}_{s\wedge\tau}(G^{-1}\is m)}\over{G_s}}+\Delta D^{o,\mathbb{F}}_{s}\\
     &=E\left[\int_{s}^{T} {\cal E}_{u-}(-{\widetilde G}^{-1}_{-}\is D^{o,\mathbb F}) d D^{o,\mathbb{F}}_u \big|{\cal F}_{s}\right]{1\over{{\cal E}_{s}(-{\widetilde G}^{-1}_{-}\is D^{o,\mathbb F})}}+\Delta D^{o,\mathbb{F}}_{s}\\
     &\leq 2\Delta D^{o,\mathbb{F}}_{s} I_{\{t<\tau\}}.
     \end{align*}
     This proves assertion (b). The remaining part of this proof addresses assertion (c). Remark that $1-(1-x)^a\leq\max(a,1) x$ for any $0\leq{x}\leq 1$. Thus, in virtue of (\ref{Vepsilon}), we get 
     \begin{eqnarray*}
    \Delta\widetilde{V}^{(a)}= 1-\left(1-{{\Delta D^{o,\mathbb F}}\over{\widetilde G}}\right)^a\leq \max(1,a){{\Delta D^{o,\mathbb F}}\over{\widetilde G}}.
     \end{eqnarray*}
     Hence, by putting      \begin{eqnarray*}
     W:= {{\max(1,a)}\over{\widetilde G}}\is  D^{o,\mathbb F}- \widetilde{V}^{(a)},\end{eqnarray*}
     we deduce that both 
     \begin{eqnarray*}
I_{\{\Delta D^{o,\mathbb F}\not=0\}} \is W=\sum\left\{ \max(1,a){{\Delta D^{o,\mathbb F}}\over{\widetilde G}}- 1+\left(1-{{\Delta D^{o,\mathbb F}}\over{\widetilde G}}\right)^a\right\}
 \end{eqnarray*}
 and $ I_{\{\Delta D^{o,\mathbb F}=0\}} \is W= {{(1-a)^+}\over{\widetilde G}}I_{\{\Delta D^{o,\mathbb F}=0\}}\is{D}^{o,\mathbb F}$ are nondecreasing processes. By combining this with
   \begin{eqnarray*}
   W= I_{\{\Delta D^{o,\mathbb F}=0\}} \is W+ I_{\{\Delta D^{o,\mathbb F}\not=0\}} \is W
        \end{eqnarray*}
     we deduce that assertion (c) holds. This ends the proof of the lemma.  \end{proof}
    
 \begin{proof}[Proof of Lemma \ref{ExpecationQtilde2P}] Remark that, for any process $H$, we have 
     $$H_{T\wedge\tau}=H_{\tau}I_{\{0<\tau\leq T\}}+H_T I_{\{\tau>T\}}+H_0I_{\{\tau=0\}}.$$
     Thus, by applying this to the process $X/{\cal E}(G_{-}^{-1}\is m)$, we derive
     \begin{eqnarray*}
     E^{\widetilde Q}[X_{T\wedge\tau}]&&=E\left[{{X_{T\wedge\tau}}\over{{\cal E}_{T\wedge\tau}(G_{-}^{-1}\is m)}}\right]=E\left[{{X_{\tau}}\over{{\cal E}_{\tau}(G_{-}^{-1}\is m)}}I_{\{0<\tau\leq T\}}+{{X_T}\over{{\cal E}_T(G_{-}^{-1}\is m)}}I_{\{\tau> T\}} +X_0I_{\{\tau=0\}}\right]\\
     &&=E\left[\int_0^T {{X_s}\over{{\cal E}_s(G_{-}^{-1}\is m)}}dD_s^{o,\mathbb F}+{{X_T}\over{{\cal E}_T(G_{-}^{-1}\is m)}}G_T+X_0(1-G_0)\right]\\
      &&= E\left[G_0^2\int_0^T X_sdV_s^{\mathbb F}+G_0X_T{\widetilde {\cal E}}_T+X_0(1-G_0)\right].
     \end{eqnarray*}
     Thus, due to $X_0=0$, (\ref{XunderQtilde}) follows immediately from the latter equality. To prove assertion (b), we take the limit on both sides of  (\ref{XunderQtilde})  and we use the fact that $G_{\infty-}=\lim_{t\longrightarrow+\infty}G_t=0$ $P$-a.s. and this ends the proof of the lemma. \end{proof}
      
     \begin{proof}[Proof of Lemma \ref{technicallemma1}]This proof has four parts where we prove the four assertions respectively. \\
    {\bf Part 1.} Let $a\in (0,+\infty)$ and $Y$ be a RCLL $\mathbb G$-semimartingale, and put $Y^*_t:=\sup_{0\leq s\leq t}\vert Y_s\vert$. Then, on the one hand,  we remark that 
\begin{eqnarray}\label{remark1}
\sup_{0\leq t\leq{T}\wedge\tau} {\widetilde{\cal E}}_t\vert Y_t\vert^a\leq \sup_{0\leq t\leq{T}\wedge\tau} {\widetilde{\cal E}}_t (Y_t^*)^a.
\end{eqnarray}
On the other hand, thanks to It\^o, we derive 
\begin{eqnarray}\label{remark2}
{\widetilde{\cal E}}(Y^*)^a=(Y_0^*)^a+{\widetilde{\cal E}}\is (Y^*)^a+(Y^*_{-})^a\is {\widetilde{\cal E}}\leq (Y_0^*)^a+{\widetilde{\cal E}}\is (Y^*)^a.
\end{eqnarray}
Thus, by combining (\ref{remark1}) and (\ref{remark2}) with ${\widetilde{\cal E}}=G/\left(G_0{\cal E}(G_{-}^{-1}\is m)\right)$, we get 
\begin{eqnarray*}
E\left[\sup_{0\leq t\leq{T}\wedge\tau} {\widetilde{\cal E}}_t\vert Y_t\vert^a\right]&&\leq{E}\left[(Y_0^*)^a+\int_0^{T\wedge\tau} {\widetilde{\cal E}_s} d(Y^*_s)^a\right]\nonumber \\
&&=E[(Y_0^*)^a]+{1\over{G_0}}E^{\widetilde{Q}}\left[\int_0^{T\wedge\tau} G_s{d}(Y^*_s)^a\right]\leq G_0^{-1}E^{\widetilde{Q}}\left[(Y^*_{T\wedge\tau})^a\right].
\end{eqnarray*}
This proves assertion (a). \\
{\bf Part 2.} Let $a\in (0,+\infty)$ and  $K$ be a RCLL nondecreasing and $\mathbb G$-optional process with $K_0=0$. Then, we remark that 
\begin{eqnarray}\label{equa300}
\widetilde{\cal E}_{-}^a \is K=K\widetilde{\cal E}^a-K\is \widetilde{\cal E}^a=K\widetilde{\cal E}^a+K{\widetilde{\cal E}_{-}^a}\is \widetilde{V}^{(a)}=K\widetilde{\cal E}^a+K_{-}{\widetilde{\cal E}_{-}^a}\is \widetilde{V}^{(a)}+\Delta{K}{\widetilde{\cal E}_{-}^a}\is \widetilde{V}^{(a)},\end{eqnarray}
where $\widetilde {V}^{(a)}$ is defined in (\ref{Vepsilon}). As a result, by combining the above equality,  the fact that $(\sum_{i=1}^n x_i)^{1/a}\leq n^{1/a}\sum _{i=1}^n x_i^{1/a}$ for any sequence of nonnegative numbers and Lemma \ref{Lemma4.11}, we derive  
\begin{eqnarray*}
&&E\left[(\widetilde{\cal E}_{-}^a \is K_{T\wedge\tau})^{1/a}\right]\nonumber\\
&&\leq 3^{1/a}E\left[(K_{T\wedge\tau})^{1/a}\widetilde{\cal E}_{T\wedge\tau}+(K_{-}{\widetilde{\cal E}_{-}^a}\is \widetilde{V}^{(a)}_{T\wedge\tau})^{1/a}+(\Delta{K}{\widetilde{\cal E}_{-}^a}\is \widetilde{V}^{(a)}_{T\wedge\tau})^{1/a}\right]\nonumber\\
&&\leq 3^{1/a}E^{\widetilde{Q}}\left[(K_{T\wedge\tau})^{1/a}{{G_{T\wedge\tau}}\over{G_0}}\right]+4\times3^{1/a}E\left[\sup_{0\leq{t}\leq{T\wedge\tau}}K_{t}^{1/a}{\widetilde{\cal E}_t}\right]+3^{1/a}E\left[(\Delta{K}{\widetilde{\cal E}_{-}^a}\is \widetilde{V}^{(a)}_{T\wedge\tau})^{1/a}\right].\end{eqnarray*}
Then, due to  $K^{1/a}\widetilde{\cal E}\leq \widetilde{\cal E}\is{K}^{1/a}$ and ${\widetilde{\cal E}}=G/\left(G_0{\cal E}(G_{-}^{-1}\is m)\right)$, the above inequality leads to 
\begin{eqnarray}\label{equa299}
&&E\left[(\widetilde{\cal E}_{-}^a \is K_{T\wedge\tau})^{1/a}\right]\nonumber\\
 &&\leq  {{3^{1/a}}\over{G_0}}E^{\widetilde{Q}}\left[(K_{T\wedge\tau})^{1/a}\right]+4\times3^{1/a}E\left[\int_0^{T\wedge\tau}{\widetilde{\cal E}_t} dK_{t}^{1/a}\right]+3^{1/a}E\left[(\Delta{K}{\widetilde{\cal E}_{-}^a}\is \widetilde{V}^{(a)}_{T\wedge\tau})^{1/a}\right]\nonumber\\
&&\leq  5{{3^{1/a}}\over{G_0}}E^{\widetilde{Q}}\left[(K_{T\wedge\tau})^{1/a}\right]+3^{1/a}E\left[(\Delta{K}{\widetilde{\cal E}_{-}^a}\is \widetilde{V}^{(a)}_{T\wedge\tau})^{1/a}\right].\end{eqnarray}
Thus, it remain to deal with the last term in the right-hand-side term of the above inequality. To this end, we distinguish the cases whether $a\geq 1$ or $a<1$. \\
The case when $a\geq 1$, or equivalently $1/a\leq 1$. Then we use the fact that $(\sum x_i)^{1/a}\leq \sum x_i^{1/a}$ for any sequence of nonnegative numbers, and get 
\begin{eqnarray*}
E\left[(\Delta{K}{\widetilde{\cal E}_{-}^a}\is \widetilde{V}^{(a)}_{T\wedge\tau})^{1/a}\right]&&= E\left[\left(\sum_{0\leq t\leq _{T\wedge\tau}} \Delta{K}_t\widetilde{\cal E}_{t-}^a\Delta\widetilde{V}^{(a)}_t\right)^{1/a}\right]\leq E\left[\sum_{0\leq t\leq _{T\wedge\tau}} (\Delta{K}_t)^{1/a}\widetilde{\cal E}_{t-}(\Delta\widetilde{V}^{(a)}_t)^{1/a}\right]\\
&&\leq a^{1/a} E\left[\sum_{0\leq t\leq _{T\wedge\tau}} (\Delta{K}_t)^{1/a}\widetilde{\cal E}_{t-}\right]= a^{1/a} E\left[\sum_{0\leq t\leq _{T\wedge\tau}} (\Delta{K}_t)^{1/a}{{\widetilde{G}_t}\over{G_t}}\widetilde{\cal E}_{t}\right] \\
&&={{a^{1/a} }\over{G_0}}{E}^{\widetilde{Q}}\left[\sum_{0\leq t\leq _{T\wedge\tau}}\widetilde{G}_t(\Delta{K}_t)^{1/a}\right]  .
\end{eqnarray*}
The last equality follows from $\widetilde{\cal E}/G=G_0^{-1}/{\cal E}(G_{-}^{-1}\is m)$. Thus, by combining this latter inequality with (\ref{equa299}), assertion (b) follows immediately for this case of $a\geq 1$. \\
For the case of $a\in (0,1)$, or equivalently $1/a>1$, we use Lemma \ref{Lemma4.11} and derive 

\begin{eqnarray*}
&&E\left[(\Delta{K}{\widetilde{\cal E}_{-}^a}\is \widetilde{V}^{(a)})_{T\wedge\tau}-(\Delta{K}{\widetilde{\cal E}_{-}^a}\is \widetilde{V}^{(a)})_{{t\wedge\tau}-}\bigg|\ \mathcal{G}_{t}\right]\\
&&=E\left[\int_{t\wedge\tau}^{T\wedge\tau}\Delta{K_s}{\widetilde{\cal E}_{s-}^a}d\widetilde{V}^{(a)}_s + (\Delta{K_{t\wedge\tau}}{\widetilde{\cal E}_{{t\wedge\tau}-}^a}\Delta \widetilde{V}^{(a)})_{t\wedge\tau}\bigg|\ \mathcal{G}_{t}\right]\\
&&\leq E\left[\int_{t\wedge\tau}^{T\wedge\tau}\sup_{0\leq u\leq s}\Delta{K_u}{\widetilde{\cal E}_{u-}^a}d\widetilde{V}^{(a)}_s + \sup_{0\leq u\leq{t\wedge\tau}} \Delta{K_u}{\widetilde{\cal E}_{u-}^a}\bigg|\ \mathcal{G}_{t}\right]\\
&&=E\left[\int_{t\wedge\tau}^{T\wedge\tau}E[\widetilde{V}^{(a)}_{T\wedge\tau}-\widetilde{V}^{(a)}_{s-}\big|{\cal G}_{s}]d\sup_{0\leq u\leq s}\Delta{K_u}{\widetilde{\cal E}_{u-}^a} + \sup_{0\leq u\leq{t\wedge\tau}} \Delta{K_u}{\widetilde{\cal E}_{u-}^a}\bigg|\ \mathcal{G}_{t}\right]\\
&&\leq E\left[ \sup_{0\leq u\leq{T\wedge\tau}} \Delta{K_u}{\widetilde{\cal E}_{u-}^a}\bigg|\ \mathcal{G}_{t}\right].\end{eqnarray*}
Therefore, a direct application of Theorem \ref{DellacherieAndMeyer}, we obtain
\begin{eqnarray*}
E\left[(\Delta{K}{\widetilde{\cal E}_{-}^a}\is \widetilde{V}^{(a)}_{T\wedge\tau})^{1/a}\right]&&\leq a^{-1/a} E\left[ \sup_{0\leq u\leq{T\wedge\tau}} \Delta{K_u}^{1/a}{\widetilde{\cal E}_{u-}}\right]\leq a^{-1/a} E\left[ \sum_{0\leq u\leq{T\wedge\tau}} \Delta{K_u}^{1/a}{\widetilde{\cal E}_{u-}}\right]\\
&&=  a^{-1/a}G_0^{-1} E^{\widetilde{Q}}\left[ \sum_{0\leq u\leq{T\wedge\tau}} \widetilde{G}_t\Delta{K_u}^{1/a}\right].
\end{eqnarray*}
Hence, by combining this inequality with (\ref{equa299}), assertion (b) follows immediately in this case of $a\in (0,1)$, and the proof of assertion (b) is complete.\\
{\bf Part 3.}  Here we prove assertion (c). To this end, we consider $p>1$,  a $\mathbb G$-optional process $H$, and we apply assertion (b) to the process $K=H\is [N^{\mathbb G}, N^{\mathbb G}]$ and $a=2/p$, and get 
\begin{eqnarray*}
    E\left[({\widetilde{\cal E}}_{-}^{2/p}H\is [N^{\mathbb G},N^{\mathbb G}])_{T\wedge\tau} ^{p/2}\right]\leq C(a)G_0^{-1}  E^{\widetilde{Q}}\left[(H\is [N^{\mathbb G},N^{\mathbb G}]_{T\wedge\tau})^{p/2}+ \sum_{0\leq t\leq {T\wedge\tau}}{\widetilde{G}_t}H^{p/2}_t\vert\Delta{N}^{\mathbb G}\vert^p\right].
     \end{eqnarray*}
     Therefore, assertion (c) follows from combining this inequality with $\sum_{0\leq{t}\leq\cdot}{\widetilde{G}_t}H^{p/2}_t\vert\Delta{N}^{\mathbb G}_t\vert= {\widetilde{G}}H^{p/2}\is \mbox{Var}(N^{\mathbb G})$ and $\vert\Delta{N}^{\mathbb G}\vert^{p-1}\leq 1$.\\
 {\bf Part 4.}  Consider  consider $p>1$ and a nonnegative and $\mathbb H$-optional process $H$. Thus, by applying assertion (c), we obtain the inequality (\ref{Equality4MG}). Hence, to get (\ref{Equality4MGOptionalF}), we remark that $\mbox{Var}(N^{\mathbb G})=(G/\widetilde{G})\is {D}+ {\widetilde{G}}^{-1}I_{\Rbrack0,\tau\Lbrack}\is D^{o,\mathbb F}$, and due to the $\mathbb F$-optinality of $H$ we have 
    \begin{eqnarray*} 
 E^{\widetilde{Q}}\left[{\widetilde{G}_t}H^{p/2}_t\mbox{Var}(N^{\mathbb G})_T\right]=2E\left[\int_0^T {{H^{p/2}_t}\over{{\cal E}_t(G_{-}^{-1}\is{m})}}I_{\Rbrack0,\tau\Lbrack}(t)d D^{o,\mathbb F}_t\right]=2E^{\widetilde{Q}}\left[(H^{p/2}I_{\Rbrack0,\tau\Lbrack}\is {D}^{o,\mathbb F})_T\right].
      \end{eqnarray*}
    Therefore, by combining this with (\ref{Equality4MG}), assertion (d) follows immediately. This ends the proof of the lemma.\end{proof} 
\end{appendices}

\bibliographystyle{amsplain}
\bibliography{DefaultBib}

\end{document}